\documentclass[12pt]{article}

\usepackage[margin=0.9in]{geometry} 
\usepackage{amsmath,amsthm,amssymb,mathrsfs,bbm,titling,mathtools,hyperref,nameref,framed,bm,enumitem,dsfont,esint}
\usepackage{color}
\usepackage[style=numeric,backend=bibtex]{biblatex}
 \let\oldref\ref
\renewcommand{\ref}[1]{(\oldref{#1})}

\newcommand{\N}{\mathbb{N}}
\newcommand{\Z}{\mathbb{Z}}
\newcommand{\C}{\mathbb{C}}
\newcommand{\R}{\mathbb{R}}
\newcommand{\Q}{\mathbb{Q}}

\newcommand{\F}{\mathbb{F}}
\newcommand{\K}{\mathbb{K}}
\newcommand{\eps}{\varepsilon}

\renewcommand{\bf}{\mathbf}
\renewcommand{\cal}{\mathcal}
\renewcommand{\frak}{\mathfrak}

\newcommand{\llparen}{(\!(}
\newcommand{\rrparen}{)\!)}

\theoremstyle{definition}

\newtheorem{theorem}{Theorem}[section]
\newtheorem{lemma}[theorem]{Lemma}
\newtheorem{corollary}[theorem]{Corollary}

\newtheorem{definition}[theorem]{Definition}
\newtheorem{example}[theorem]{Example}

\newtheorem{proposition}[theorem]{Proposition}

\theoremstyle{remark}
\newtheorem{remark}[theorem]{Remark}

\numberwithin{equation}{section}

\begin{document}
 
% --------------------------------------------------------------
%                         Start here
% --------------------------------------------------------------

\title{Discrete nonlinear H\"older--Brascamp--Lieb inequalities: a local approach}
\author{Ben Johnsrude}
 \date{}
\maketitle
\vspace{-0.9in}

\begin{abstract}
    We study nonlinear H\"older--Brascamp--Lieb inequalities over $\Z^n$. Our method embeds the integers into a suitable compactification, over which local H\"older--Brascamp--Lieb theory may be applied (e.g.\ heat flow). Our main result is a bound for multilinear H\"older--Brascamp--Lieb functionals over suitable sublattices in $\Z^n$.
\end{abstract}

\vspace{1em}

\section{Introduction}

For $1\leq j\leq m$, let $P_j:\Z^n\to\Z^{n_j}$ be a polynomial map, whose image we may assume to not be contained in any affine subspace. For which exponents $1\leq q_1,\ldots,q_m\leq +\infty$ is it the case that, for any finite subset $\Omega\subseteq\Z^n$, we have the bound
\begin{equation}\label{ineq:intro_setbd}
    \#\Omega\lesssim\prod_{j=1}^m(\#P_j(\Omega))^{\frac{1}{q_j}},
\end{equation}
with implicit constant uniform over choices of $\Omega$?

A natural ``zooming-in'' pseudo-symmetry is relevant. Given $N\in\N$ and $\Omega\subseteq\Z^n$, we write
\begin{equation*}
    \mathrm{Dil}_N(\Omega)=N\Omega+\{0,\ldots,N-1\}^n=\{Na+b:a\in\Omega, 0\leq b_j\leq N-1\}.
\end{equation*}
Then $\#(\mathrm{Dil}_N(\Omega))=N^n\#\Omega$, and (if the $P_j$ behaves enough like a projection) one roughly expects
\begin{equation*}
    \#P_j(\mathrm{Dil}_N(\Omega))\approx N^{n_j}\# P_j(\Omega),
\end{equation*}
for most choices of $\Omega$ and $P_j$; certainly, this occurs when each $P_j$ is an affine-linear full-rank map. In such cases,~\eqref{ineq:intro_setbd} fails unless $n\leq\sum_j\frac{n_j}{q_j}$. In the case that $q_j=1$ for all $j$,~\eqref{ineq:intro_setbd} trivially follows from a mild injectivity assumption on the map $(P_1,\ldots,P_m)$. We will restrict attention to the alternate extreme, where the powers of $N$ balance on both sides of~\eqref{ineq:intro_setbd}, i.e.\ along the \emph{scaling line}
\begin{equation}\label{eq:intro_scaling}
    n=\sum_{j=1}^m\frac{n_j}{q_j}.
\end{equation}

In the case of linear maps $P_j$, a full characterization of inequalities~\eqref{ineq:intro_setbd} was achieved in~\cite{christ2015discrete}. There,~\eqref{ineq:intro_setbd} is essentially the \emph{H\"older--Brascamp--Lieb} inequality; the real versions of such estimates were fully described in~\cite{bennett2008brascamp}, where it was demonstrated that boundedness was equivalent to the \emph{rank} and \emph{scaling conditions}. Thus, for surjective linear maps $L_j$ between real vector spaces, boundedness in the sense of~\eqref{ineq:intro_setbd} (where cardinality is replaced with Lebesgue measure) with $P_j=L_j$ amounts to the scaling condition~\eqref{eq:intro_scaling} and the rank conditions
\begin{equation}\label{ineq:intro_rank}
    \hspace{12em}\dim V\leq\sum_{j=1}^m\frac{\dim (L_jV)}{q_j},\quad\forall \bf 0\neq V\subsetneq\R^n\,\,\text{linear subspace}.
\end{equation}
Returning to linear maps over the integers,~\cite{christ2015discrete} showed that finiteness of~\eqref{ineq:intro_setbd} is equivalent to the rank conditions, in the sense that the ranks of subgroups respect the inequalities~\eqref{ineq:intro_rank}.

Utilizing the rank and scaling conditions, we are able to show the following.
\begin{theorem}[Main theorem]\label{thm:int_by_padic} Suppose $\bf P=(P_j)_j$ is as above, $1\leq q_1,\ldots,q_m\leq+\infty$ satisfy~\eqref{eq:intro_scaling}, and suppose that $a\in\Z^n$ is such that the $P_j$ satisfy submersion and rank conditions:
\begin{equation*}
    \mathrm{rank}(dP_j(a))=n_j,\quad\forall 1\leq j\leq m,
\end{equation*}
\begin{equation*}
    \mathrm{rank}(V)\leq\sum_{j=1}^m\frac{\mathrm{rank}\big(dP_j(a)(V)\big)}{q_j},\quad\forall V\leq\Z^n.
\end{equation*}
    Then there is a positive integer $\Delta_a=\Delta_a(d\bf P(a))\in\N$, determinable via an algorithm as a function of $d\bf P(a)$, such that the following holds. If $p\in\N$ is prime such that $p\nmid\Delta_a$, then for any $(f_j)_{1\leq j\leq m}, f_j:\Z^{n_j}\to\C$, we have
    \begin{equation}\label{ineq:mainthm_func}
        \left|\sum_{x\in a+p\Z^n}\prod_{j=1}^mf_j(P_j(x))\right|\leq\prod_{j=1}^m\|f_j\|_{\ell^{q_j}(\Z^{n_j})}.
    \end{equation}
\end{theorem}

\eqref{ineq:mainthm_func} asserts $\ell^{p_1}\times\cdots\times\ell^{p_m}$ boundedness of a suitable multilinear functional, whereas~\eqref{ineq:intro_setbd} is the corresponding restricted weak-type bound. A ``global'' bound would take the form
\begin{equation}\label{ineq:intro_funbd}
    \left|\sum_{x\in\Z^n}\prod_{j=1}^mf_j(P_j(x))\right|\lesssim\prod_{j=1}^m\|f_j\|_{\ell^{q_j}(\Z^{n_j})};
\end{equation}
one should compare to local/global functional bounds over $\R$ such as
\begin{equation*}
    \left|\int_{B(0,1)}\prod_{j=1}^mf_j(P_j(x))\right|\lesssim\prod_{j=1}^m\|f_j\|_{L^{q_j}(\R^{n_j})},\quad\quad\left|\int_{\R^n}\prod_{j=1}^mf_j(P_j(x))\right|\lesssim\prod_{j=1}^m\|f_j\|_{L^{q_j}(\R^{n_j})}.
\end{equation*}
These latter bounds, especially in the case that the $P_j$ are affine-linear, are the usual sense of H\"older--Brascamp--Lieb inequalities.

\subsection{Discussion of the approach}

The nonlinear H\"older--Brascamp--Lieb inequality over the integers, i.e.~\eqref{ineq:intro_funbd} in complete generality, should be expected to be much more difficult than the linear one. In the case of linear maps (either over the integers or over the reals), one critically relies on ``factorization'' methods. The idea is that, if at a particular exponent tuple $(q_1,\ldots,q_m)$, one of the rank inequalities~\eqref{ineq:intro_rank} at $V$ holds with an equality, then finiteness of~\eqref{ineq:intro_funbd} is equivalent to the simultaneous finiteness of the two H\"older--Brascamp--Lieb estimates with maps $\left.L_j\right|_V:V\to L_j(V)$ and $\left.L_j\right|_{\R^n/V}:\R^n/V\to \R^{n_j}/L_j(V)$ (and analogously over the integers). On the other hand, for ``simple'' data (i.e.\ when all rank conditions are strict inequalities), a heat-flow argument demonstrates that the near-extremizing $\Omega$ are forced to be highly regular (e.g.\ something like convex neighborhoods of $0$ with low eccentricity). An obvious problem, then, for non-linear H\"older--Brascamp--Lieb is that polynomial maps are not homomorphisms, and so do not ``split'' in a useful sense along subspaces.

In the nonlinear case, there have been developments regarding \emph{local} H\"older--Brascamp--Lieb inequalities with nonlinear maps over the reals; these are generally assumed to be sufficiently regular, e.g.\ $C^2$ or $C^{1+\alpha}$. The idea is that, near a given point $x$ in the domain, the nonlinear maps $B_j$ are approximately linear. If the H\"older--Brascamp--Lieb inequality with the linearized maps $dB_j(x)$ exhibit sufficiently regular extremizers, then by averaging against local extremizers one may force the nonlinear functional to be bounded by the linearized functional. Thus,~\cite{bennett2018nonlinear,bennett2020nonlinear} showed that, near a point $x$, the H\"older--Brascamp--Lieb inequality holds with constant essentially equal to the constant arising from the linearization $dB_j(x)$.

We would like to adopt the spirit of local analysis to control nonlinear H\"older--Brascamp--Lieb functionals over the integers. The major concern is the fact that the integers are not easily localizable: the behavior near a point is exactly trivial, and gives virtually no information about the global functional. One might hope to use ``Fourier duality,'' as described in~\cite{bennett2022fourier}. We briefly sketch a difficulty. The general idea is to use Plancherel to say
\begin{equation*}
    \begin{split}
        \sum_{x\in\Z^n}\prod_{j=1}^mf_j(P_j(x))&=\big\langle f_1\otimes\cdots\otimes f_m,\delta_{\bf P(\Z^n)}\big\rangle_{\Z^{n_1}\times\cdots\times\Z^{n_m}}\\
        &=\big\langle\hat f_1\otimes\cdots\otimes\hat f_m,\hat\delta_{\bf P(\Z^n)}\big\rangle_{(\R/\Z)^{n_1}\times\cdots\times(\R/\Z)^{n_m}};
    \end{split}
\end{equation*}
here, $\delta_{\bf P(\Z^n)}$ is the pushforward of the sum of $\delta$-masses along $\Z^n$ under the map $\bf P$. In the case that each $P_j$ is linear, then $\hat\delta_{\bf P(\Z^n)}=\mu_{\bf P(\Z^n)^\top}$ is just a matching surface measure of the ``annihilator'' (working over $\R$, this is the orthocomplement); thus, our analysis gets moved to the Pontryagin dual $\widehat\Z^n=(\R/\Z)^n$, which \emph{does} permit a useful ``local'' analysis. However, if the polynomials $P_j$ are nonlinear, then the Fourier transform $\hat\delta_{\bf P(\Z^n)}$ is an immensely delicate object, which is very different from a surface measure. For this reason, we avoid Fourier duality in this nonlinear case.

Our approach is quite different. We choose to embed the lattice within a suitable space on which the local analysis of~\cite{bennett2020nonlinear} is possible. Specifically, we embed the integers $\Z$ into the $p$-adic rings $\Z_p$. The latter are compact, and structurally rich enough to permit arguments such as ``near the point $x$, the polynomial $P_j$ is approximately linear....'' Neighborhoods of points occupy a positive proportion of the measure of the overall domain, so obtaining a good functional bound on such a set is nontrivial and interesting. Finally, the scaling relation~\eqref{eq:intro_scaling} guarantees that the bounds on the compact domains $\Z_p^n$ imply matching bounds on the discrete domains $\Z^n$.

Our main results are stated for polynomial functions over the integers, but critical intermediate results deal more generally with ``$C^1$ functions'' over $p$-adic and other non-Archimedean fields. We provide the theoretical minimum for $C^k$ functions over non-Archimedean fields in the Appendix~\ref{sec:appendix}. The text~\cite{schikhof1985} is the classical reference for calculus over non-Archimedean fields; our definitions agree with the latter's in the special case of one variable, but the theory of non-Archimedean calculus in several variables is underdeveloped there. Some useful references for the basics of non-Archimedean analysis are~\cite{cassels1986local,serre2013local}.

Our workhorse in the analysis of nonlinear Brascamp--Lieb functionals over non-Archimedean fields is a variant of Ball's inequality. The upshot of that analysis is that, along the scaling line $n=\sum_{j=1}^m\frac{n_j}{p_j}$, a nonlinear Brascamp--Lieb bound is essentially equivalent to a supremal \emph{linear} Brascamp--Lieb bound, ranging over the derivatives of the nonlinear maps.

As a consequence, we will need to engage in a non-Archimedean variant of differential geometry, albeit an elementary one; all maps are only $C^1$, so in particular no second-order information appears. Nevertheless, several interesting subtleties appear that require special treatment. Namely,
\begin{enumerate}
    \item A $C^k$ map $F:\Q_p^n\to\Q_p^{n'}$ with $dF\equiv 0$ need not be constant, even locally.
    \item A nonzero vector $v\in\Q_p^n$ may have $v\cdot v=0$, for certain primes $p$ and dimensions $n$; in particular, $\|v\|^2\neq |v\cdot v|$, in general. Similarly, $A\mapsto\mathrm{tr}(AA^\top)$ and $A\mapsto\mathrm{tr}(A^\top A)$ fail to define anything like a norm on matrices $A$. In contrast to the complex case, there is no ``conjugation trick'' to create a definite form; this is essentially the fact that non-Archimedean fields do not have totally real subfields.
\end{enumerate}
In contrast to~\cite{bennett2020nonlinear}, we avoid utilizing convex analysis. The basic problem in attempting to translate the latter work to the non-Archimedean setting is that the restriction of a convex bounded-below function $\Gamma:\R^d\to\R$ to the lattice $\Z^d$, may not witness any minimizers. This leads to issues relating to the ``discrete'' nature of non-Archimedean fields; viz., that they have only a discrete set of scales.

\subsection{Intermediate results}

In order to state the intermediate results used in the proof of Theorem~\ref{thm:int_by_padic}, we briefly develop the basic terminology.

Fix a non-discrete non-Archimedean local field $\K$. One may show that $\K$ is necessarily either a finite extension of $\Q_p$ for some prime $p$, or is of the form $\mathbb{F}_{p^n}\llparen t\rrparen$. In either case, $\K$ is a LCA group, so has a Haar measure $\mu=\mu_1$. Write $|\cdot|=|\cdot|_\K:\K\to[0,+\infty)$ for the modular function of $\mu_1$; that is to say, for a compact open subset $U\subseteq\K$,
\begin{equation*}
    |\lambda|=\frac{\mu_1(\lambda\cdot U)}{\mu_1(U)},\quad\text{where}\,\,\lambda\cdot U=\{\lambda x:x\in U\}.
\end{equation*}
Thus, $|1|=1$, and $|\lambda|=0$ if and only if $\lambda=0$, for any $\lambda\in\K$. Additionally, $|\lambda\eta|=|\lambda|\cdot|\eta|$ and $|\lambda+\eta|\leq\max(|\lambda|,|\eta|)$, for any $\lambda,\eta\in\K$. The latter is referred to as the \emph{ultrametric triangle inequality}, making $\K$ an \emph{ultrametric space}. We write $\cal O_1=\{x\in\K:|x|\leq 1\}$. $\mu_1$ may be rescaled so that $\mu_1(\cal O_1)=1$; note that this does not affect the definition of $|\cdot|$. This $\cal O_1$ is the \emph{ring of integers} in $\K$.

$\cal O_1$ is a local ring, so has a unique maximal ideal $\frak m=\{x\in\cal O:|x|<1\}$. Fix once and for all a \emph{uniformizer} $\varpi$; i.e.\ $\varpi\in\K$ is such that $\varpi\cal O_1=\frak m$. It follows that, if $\lambda\in\K^\times$, then there is some $e\in\Z$ such that $|\lambda|=|\varpi^e|$.

\begin{example}
    If $\K=\Q_p$, then $\cal O_1=\Z_p$ and $\varpi$ may be chosen to be $p$. If $\K=\mathbb{F}_{p^n}\llparen t\rrparen$, then $\cal O_1$ is composed of the power series in $t$ with coefficients in $\mathbb{F}_{p^n}$; i.e.\ with no negative powers on $t$. In the latter case, $\varpi$ may be chosen to be $t$.

    If $\K$ is a finite extension of $\Q_p$, say $[\K:\Q_p]=d$, then $|p|=p^{-d}$ and $p\cal O_1=\frak m^d$. Depending on $\K$, we may have $|\varpi|=p^{-\ell}$ for any fixed $\ell|d$. The integer $\ell$ is known as the \emph{residual degree}, and the quotient $d/\ell$ the \emph{ramification index}, of the extension $\K/\Q_p$. In the special case $\ell=1$, one may imagine $\varpi$ as a suitable ``$d$th root of $p$.''
\end{example}

We will be interested in functions on $\K^k$ for $k\in\N$. Each $\K^k$ is equipped with the norm $\|(x_i)_{1\leq i\leq n}\|=\max_{1\leq i\leq k}|x_i|$. We will write $\cal O_k=\cal O_1^k=\{x\in\K^k:\|x\|\leq 1\}$. $\K^k$ will be equipped with its own Haar measure $\mu_k$, normalized so that $\mu_k(\cal O_k)=1$.

\definition[Schwartz--Bruhat functions] For each $k\in\N$, we write $\operatorname{SB}(\K^k)$ for the collection of functions $f:\K^k\to\C$ for which there are $h,\gamma\in\K^\times$ such that $f$ is supported in $\gamma\cal O_k$ and $f$ is $h\cal O_k$-invariant: that is, if $\|x-y\|\leq|h|$, then $f(x)=f(y)$. This class is known as the collection of \emph{Schwartz--Bruhat functions}. Thus, Schwartz--Bruhat functions are those that are compactly supported and locally constant.

The Brascamp--Lieb functionals of interest are defined in terms of choices of (linear) maps and (recipocal) Lebesgue exponents. It is convenient to package that data as follows.

\definition[Brascamp--Lieb datum] A \emph{Brascamp--Lieb datum} is a pair $(\bf L,\bf c)$ of tuples $\bf L=(L_j)_{1\leq j\leq m}$ and $\bf c=(c_j)_{1\leq j\leq m}$, where each $L_j$ is a linear map $L_j:\K^n\to\K^{n_j}$ for some $n_j\in\N$, and each $c_j\in[0,1]$.

For a given Brascamp--Lieb datum, we may define the associated functional.

\definition[Brascamp--Lieb constant] Let $(\bf L,\bf c)$ be a Brascamp--Lieb datum. If $\bf f=(f_j)_{1\leq j\leq m}$ is a tuple of nonnegative nonzero Schwartz--Bruhat functions, with $f_j$ defined on $\K^{n_j}$, then we write
\begin{equation*}
    \operatorname{BL}(\bf L,\bf c;\bf f)=\frac{\int_{\K^n}\prod_{j=1}^mf_j^{c_j}\circ L_j}{\prod_{j=1}^m\big(\int_{\K^{n_j}}f_j\big)^{c_j}}.
\end{equation*}
Here we have abbreviated $\bf f=(f_j)_{1\leq j\leq m}$. The \emph{Brascamp--Lieb constant} is then the quantity
\begin{equation*}
    \operatorname{BL}(\bf L,\bf c)=\sup\big\{\operatorname{BL}(\bf L,\bf c;\bf f):\bf f=(f_j)_{1\leq j\leq m},\,\,f_j:\K^{n_j}\to\R_{\geq 0}\text{ Schwartz--Bruhat, nonzero}\big\}.
\end{equation*}

Our first result concerning Brascamp--Lieb inequalities over non-Archimedean fields $\K$ is the following, concerning the \emph{linear} case.

\begin{theorem}[Non-Archimedean linear Brascamp--Lieb inequalities admit extremizers]\label{thm:main_extr_functional}
    Let $(\bf L,\bf c)$ be any Brascamp--Lieb datum such that $\bigcap_{j=1}^m\ker L_j=\{0\}$. Then $\operatorname{BL}(\bf L,\bf c)<+\infty$ if and only if there is some tuple $\bf f=(f_j)_{1\leq j\leq m}$ of nonzero Schwartz--Bruhat functions such that $\operatorname{BL}(\bf L,\bf c)=\operatorname{BL}(\bf L,\bf c;\bf f)$.

    Moreover, in this case, there is an invertible  $n\times n$ matrix $A$ over $\K$ such that the extremizer $f_j$ may be taken to be $1_{(L_j\circ A)(\cal O_n)}$.
\end{theorem}

Our main goal in studying linear Brascamp--Lieb inequalities over $\K$ is in the application to nonlinear Brascamp--Lieb inequalities. Our nonlinear maps will be assumed to be $C^1$, in a technical sense specified in Definition~\ref{def:Ck}. Morally, a $C^1$ map is one that is well-approximated by an affine map near each point; for reasons peculiar to ultrametric geometry, it is necessary to somewhat strengthen the naive definition.

For a tuple $\bf B=(B_j)_{1\leq j\leq m}$ of $C^1$ maps, we write
\begin{equation*}
    \operatorname{BL}(\bf B,\bf c;\bf f)=\frac{\int\prod_{j=1}^mf_j^{c_j}\circ B_j}{\prod_{j=1}^m\big(\int f_j\big)^{c_j}},\quad\operatorname{BL}(\bf B,\bf c)=\sup_{\bf f}\operatorname{BL}(\bf B,\bf c;\bf f),
\end{equation*}
where the $\bf f$ ranges over tuples of nonnegative nonzero Schwartz--Bruhat functions, as in the linear case.

Our main theorem concerning these nonlinear Brascamp--Lieb functionals over local fields is the following.
\begin{theorem}\label{thm:main_K_NL}
    Suppose $(\bf B,\bf c)$ is a nonlinear Brascamp--Lieb tuple, with $\bf B$ defined on $U\subseteq\K^n$. Suppose $dB_j(x)$ is surjective whenever $c_j\neq 0$. Then there is a $\delta\in\K^\times$ depending on $(\bf B,\bf c)$ such that
    \begin{equation*}
        \operatorname{BL}(\left.\bf B\right|_{u+\delta\cal O_n},\bf c)\leq\sup_{x\in u+\delta\cal O_n}\operatorname{BL}(d\bf B(x),\bf c),\quad\forall u\in U_\delta.
    \end{equation*}
    Here $U_\delta=\{u\in U:u+\delta\cal O_n\subseteq U\}$.
\end{theorem}

Thus, at least locally, the study of nonlinear $C^1$ H\"older--Brascamp--Lieb functionals is partially reducible to linear Brascamp--Lieb functionals. A matching bound is available under an additional hypothesis:

\begin{theorem}\label{thm:main_L_by_NL}
    Suppose $(\bf B,\bf c)$ is a nonlinear Brascamp--Lieb tuple. Suppose we have the super-scaling inequality
    \begin{equation*}
        n\leq\sum_{j=1}^mc_jn_j.
    \end{equation*}
    Then $\operatorname{BL}(\bf B,\bf c)\geq\sup_{x\in U}\operatorname{BL}(d\bf B(x),\bf c)$.
\end{theorem}

We repeat the observation that the inequality $|\sum_x\prod_jf_j(P_j(x))|\lesssim\prod_{j=1}^m\|f_j\|_{\ell^1}$ is trivial, under the assumption that $\bf P$ has finite fibers. Thus, Theorem~\ref{thm:int_by_padic} supplies the other, nontrivial endpoint result in suitable local sense. We will see in Example~\ref{ex:local_to_global_enemy} that it is too much to hope for the local bounds in Theorem~\ref{thm:int_by_padic} to imply global bounds, even away from the endpoint.

\subsection{Overview of the remainder of the manuscript}
\begin{itemize}
    \item In Section~\ref{sec:NA_geom}, we will supply the basic theory of the geometry of non-Archimedean finite-dimensional vector spaces that will be needed for the Brascamp--Lieb analysis.
    \item In Section~\ref{sec:lin_BL}, we will prove various critical lemmas about the theory of linear Brascamp--Lieb over non-Archimedean fields.
    \item In Section~\ref{sec:lin_extr}, we will prove that linear Brascamp--Lieb inequalities over non-Archimedean fields always exhibit extremizers (i.e.\,Theorem~\ref{thm:main_extr_functional}), and we will provide regularity bounds for those extremizers.
    \item In Section~\ref{sec:nonlinear_constraints}, we will provide necessary conditions for nonlinear Brascamp--Lieb inequalities over non-Archimedean fields.
    \item In Section~\ref{sec:nonlin_to_lin}, we will prove Theorems~\ref{thm:main_K_NL} and~\ref{thm:main_L_by_NL}; that is, nonlinear Brascamp--Lieb inequalities over non-Archimedean fields reduce to linear Brascamp--Lieb inequalities.
    \item In Section~\ref{sec:nonlin_nec}, we will provide necessary conditions for polynomial Brascamp--Lieb inequalities over the integers.
    \item In Section~\ref{sec:rank_transfer}, we will prove a critical theorem showing that the rank conditions over $\Z$ locally imply rank conditions over $\F_p$.
    \item In Section~\ref{sec:nonlin_suff}, we will put all the previous theory together to prove Theorem~\ref{thm:int_by_padic}.
    \item In Section~\ref{sec:applications}, we indicate applications and special cases of Theorem~\ref{thm:int_by_padic}.
    \item In Section~\ref{sec:appendix}, the Appendix, we develop the basic theory of $C^k$ mappings that is needed in various sections throughout the manuscript.
\end{itemize}

\subsection{Acknowledgments}

I am indebted to Betsy Stovall for supplying the original idea of studying nonlinear Brascamp--Lieb over non-Archimedean local fields, and for providing helpful guidance in understanding the corresponding real theory. I would also like to acknowledge Jaume de Dios Pont for suggesting and discussing the idea of using $p$-adic analysis to attack related functionals over the integers.

\section{Geometry of non-Archimedean vector spaces}\label{sec:NA_geom}

In this section, we discuss some elementary features of the geometry of finite-dimensional vector spaces over a non-Archimedean local field $\K$. Every vector space in this manuscript is taken to be equipped with an ultrametric norm, satisfying the two compatibility laws:
\begin{equation}\tag{S}\label{eq:norm_scaling}
    \|\lambda v\|_V=|\lambda|_\K\cdot\|v\|_V,\quad\forall\lambda\in\K,\,\,v\in V,
\end{equation}
\begin{equation}\tag{V}\label{eq:norm_values}
    \{\|v\|_V:v\in V\}=\{0\}\cup\{|\varpi|^n:n\in\Z\}.
\end{equation}

\begin{definition}[Unit ball]
    Suppose $V$ is a $n$-dimensional ultrametric vector space over $\K$ satisfying~\eqref{eq:norm_scaling} and~\eqref{eq:norm_values}. We write $\cal O_V=\{v\in V:\|v\|_V\leq 1\}$. If $V=\K^n$, we will usually simply write $\cal O_{\K^n}=\cal O_n$; thus, $\cal O_1$ is the ring of integers of $\K$.
\end{definition}

We have some important families of matrices over $\K$.

\begin{definition}[General linear groups]
    For a $\K$-vector space $V$, we write $\mathrm{GL}(V)$ for the linear isomorphisms of $V$. If $V$ is ultrametric and satisfies~\eqref{eq:norm_scaling} and~\eqref{eq:norm_values}, we write $\operatorname{GL}(\cal O_V)$ for the linear isometries of $V$.

    If $V=\K^n$, we will frequently write $\mathrm{GL}(\K^n)=\mathrm{GL}_n(\K)$ and $\mathrm{GL}(\cal O_{\K^n})=\mathrm{GL}_n(\cal O_1)$.
\end{definition}
The $\mathrm{GL}_n(\cal O_1)$ notation is chosen for the following reason: a linear isometry of $\K^n$ is exactly an invertible $n\times n$ matrix over the ring $\cal O_1$.

Closely related is the following non-Archimedean variant of ``orthonormal.''
\begin{definition}[Isometric tuples] A tuple of vectors $(u_1,\ldots,u_k)$ in $V$ are said to be \emph{isometric} if the map
\begin{equation*}
    \K^k\ni(\alpha_1,\ldots,\alpha_k)\mapsto\alpha_1u_1+\ldots+\alpha_ku_k
\end{equation*}
defines an isometric embedding of $\K^k$ into $V$. If they further form a basis for $V$, we describe the tuple as an \emph{isometric basis}.
\end{definition}
\begin{lemma}[$\operatorname{GL}_n(\cal O_1)$ preserves isometric tuples]\label{lem:isom_permute}
    Suppose $(u_1,\ldots,u_k)$ form an isometric tuple in an ultrametric vector space $V$ over $\K$. Suppose as well that $A=[A_{ij}]\in\mathrm{GL}_k(\cal O_1)$. Then, defining
    \begin{equation*}
        v_j=\sum_{i=1}^kA_{ji}u_i,
    \end{equation*}
    we have that $(v_1,\ldots,v_k)$ is isometric, and $\mathrm{span}_\K(u_1,\ldots,u_k)=\mathrm{span}_\K(v_1,\ldots,v_k)$.
\end{lemma}
\begin{proof}
    The ``span'' part of the claim is trivial to see, so we just consider the ``isometric'' part. Write $T:\K^k\to V$ for the map $(\alpha_1,\ldots,\alpha_k)\mapsto\alpha_1u_1+\ldots+\alpha_ku_k$. Then, for any $\bar\beta=(\beta_1,\ldots,\beta_k)\in\K^k$,
    \begin{equation*}
        \beta_1v_1+\ldots+\beta_kv_k=\sum_{i,j=1}^k\beta_j A_{ji}u_i=T(\gamma_1,\ldots,\gamma_k),
    \end{equation*}
    where $\gamma_i=\sum_{j=1}^k\beta_j A_{ji}=(A^\top\bar\beta)_i$. Since $A\in\mathrm{GL}_k(\cal O_1)$, we have that $\|A^\top\bar\beta\|=\|\bar\beta\|$. Thus,
    \begin{equation*}
        \|\beta_1v_1+\ldots+\beta_kv_k\|_V=\|T(\gamma_1,\ldots,\gamma_k)\|_V=\|A^\top\bar\beta\|_{\K^k}=\max_i|\beta_i|,
    \end{equation*}
    and the result follows.
\end{proof}
\begin{lemma}[Non-Archimedean Gram--Schmidt]\label{lem:G-S}
    Suppose $V$ is an ultrametric vector space over $\K$ satisfying~\eqref{eq:norm_scaling} and~\eqref{eq:norm_values}. Suppose that $v_1,\ldots,v_k\in V$ are $\K$-linearly independent. Then there are $u_1,\ldots,u_k\in V$ satisfying the following.
    \begin{itemize}
        \item[(a)] $(u_1,\ldots,u_k)$ form an isometric tuple.
        \item[(b)] For each $1\leq j\leq k$, we have $\mathrm{span}_\K(v_1,\ldots,v_j)=\mathrm{span}_\K(u_1,\ldots,u_j)$.
    \end{itemize}
\end{lemma}
\begin{proof}
    By induction on $k$. The case $k=1$ is solved by finding $\gamma\in\K^\times$ with $|\gamma|=\|v_1\|$, and then choosing $u_1=\gamma^{-1}v_1$.

    Now, assume that $k\geq 2$ and the result is known for smaller $k$. Thus, we may find $u_1,\ldots,u_{k-1}$ forming an isometric tuple such that $\mathrm{span}_\K(v_1,\ldots,v_j)=\mathrm{span}_\K(u_1,\ldots,u_j)$ for each $1\leq j\leq k-1$. As before, find $\gamma\in\K^\times$ such that $|\gamma|=\|v_k\|$, and write $ v_k^1=\gamma^{-1}v_k$.

    Next, we construct $v_k^2$; subsequently, we will have a slightly general different construction for $v_k^i$, $3\leq i\leq k$. If $(u_1,v_k^1)$ form an isometric tuple, then we simply set $v_k^2=v_k^1$. Otherwise, choose an $\alpha\in\K$ which minimizes $\|\alpha u_1+v_k^1\|_V$. Since $(u_1,v_k^1)$ are not isometric but $\|u_1\|_V=1=\|v_k^1\|_V$, we must have $|\alpha|=1$ and $\|\alpha u_1+v_k^1\|_V<1$. Since $u_1,v_k^1$ are $\K$-linearly independent, there exists $\rho\in\K^\times$ with $|\rho|=\|\alpha u_1+v_k^1\|_V$. From this, we define
    \begin{equation*}
        v_k^2=\rho^{-1}(v_k^1+\alpha u_1).
    \end{equation*}
    Note that $\mathrm{span}_\K(v_1,v_k^2)=\mathrm{span}_\K(u_1,u_k)$, and $\|v_k^2\|_V=1$. We claim that $(u_1,v_k^2)$ are isometric. It suffices to consider expressions of the form $\beta u_1+v_k^1$ with $|\beta|=1$. If $\|\beta u_1+v_k^2\|_V<1$, then
    \begin{equation*}
        \|\beta u_1+v_k^2\|_V=\|(\beta+\rho^{-1}\alpha)u_1+\rho^{-1}v_k^1\|_V=|\rho|^{-1}\|(\rho\beta+\alpha)u_1+v_k^1\|_V
    \end{equation*}
    implies that $\|(\rho\beta+\alpha)u_1+v_k^1\|_V<|\rho|$, contradicting the minimality of $|\rho|$. Thus, $\|\beta u_1+v_k^2\|_V=1=\max(|\beta|,1)$; all other cases are either trivial or reducible to this one, using~\eqref{eq:norm_scaling}.

    Next, suppose $i\geq 3$ and we have already constructed $v_k^{i-1}$, satisfying the property that $\mathrm{span}_\K(u_1,\ldots,u_{i-2},v_k^{i-1})=\mathrm{span}_\K(u_1,\ldots,u_{i-2},u_k)$, and that $(u_1,\ldots,u_{i-2},v_k^{i-1})$ is isometric. If $(u_1,\ldots,u_{i-1},v_k^{i-1})$ is isometric, then we choose $v_k^i=v_k^{i-1}$. Otherwise, suppose that $\alpha_1,\ldots,\alpha_{i-1}\in\K$ are such that
    \begin{equation*}
        \|\alpha_1u_1+\ldots+\alpha_{i-1}u_{i-1}+v_k^{i-1}\|_V
    \end{equation*}
    is minimized. Certainly then $|\alpha_t|\leq 1$ for all $t$. Define
    \begin{equation*}
        \tilde v_k^{i-1}=\alpha_1u_1+\ldots+\alpha_{i-2}u_{i-2}+v_k^{i-1};
    \end{equation*}
    by Lemma~\ref{lem:isom_permute}, we still have $(u_1,\ldots,u_{i-2},\tilde v_k^{i-1})$ is isometric, and that $\mathrm{span}_\K(u_1,\ldots,u_{i-2},\tilde v_k^{i-1})=\mathrm{span}_\K(u_1,\ldots,u_{i-2},u_k)$. Thus, we minimize
    \begin{equation*}
        \|\beta_1u_1+\ldots+\beta_{i-2}u_{i-2}+\beta_{i-1}u_{i-1}+\tilde v_k^{i-1}\|_V
    \end{equation*}
    by choosing $\beta_t=0$ for $t\leq i-2$, and $\beta_{i-1}=\alpha_{i-1}$. Since $u_1,\ldots,u_{i-2},\tilde v_k^{i-1}$ span an $i-1$-dimensional space, and $u_{i-1}\not\in\mathrm{span}_\K(u_1,\ldots,u_{i-2},u_k)=\mathrm{span}_\K(u_1,\ldots,u_{i-2},\tilde v_k^{i-1})$, we have that $(u_1,\ldots,u_{i-1},\tilde v_k^{i-1})$ are $\K$-linearly independent, and hence the above minimum is nonzero. Let $\rho\in\K^\times$ be such that
    \begin{equation*}
        |\rho|=\|\alpha_{i-1}u_{i-1}+\tilde v_k^{i-1}\|_V.
    \end{equation*}
    Finally, define $v_k^i=\rho^{-1}(\alpha_{i-1}u_{i-1}+\tilde v_k^{i-1})$. Clearly, $\mathrm{span}_\K(u_1,\ldots,u_{i-1},v_k^{i})=\mathrm{span}_\K(u_1,\ldots,u_{i-1},u_k)$ and $\|v_k^i\|_V=1$. We check the isometry property. It suffices to consider $\beta_1,\ldots,\beta_{i-1}\in\K$ with $|\beta_t|\leq 1$ for all $t$, and to suppose that for some such choice we have
    \begin{equation*}
        \|\beta_1u_1+\ldots+\beta_{i-1}u_{i-1}+v_k^i\|_V<1.
    \end{equation*}
    Recalling the definition of $v_k^i$ in terms of $\tilde v_k^{i-1}$, we conclude that
    \begin{equation*}
        \|\rho\beta_1u_1+\ldots+\rho\beta_{i-2}u_{i-2}+(\rho\beta_{i-1}+\alpha_{i-1})u_{i-1}+\tilde v_k^{i-1}\|_V<|\rho|;
    \end{equation*}
    on the other hand, this contradicts the minimality assumption with regards to $\tilde v_k^{i-1}$. Thus, $(u_1,\ldots,u_{i-1},v_k^i)$ is isometric.

    Eventually, we arrive at an isometric tuple $(u_1,\ldots,u_{k-1},v_k^{k})$, which satisfies 
    \begin{equation*}
        \mathrm{span}_\K(u_1,\ldots,u_{k-1},v_k^k)=\mathrm{span}_\K(u_1,\ldots,u_{k-1},u_k).
    \end{equation*}
    Defining $u_k=v_k^{k}$, we conclude by induction on $k$.
\end{proof}
\begin{remark}\label{rmk:isom_tuple_facts}
    An isometric tuple is necessarily $\K$-linearly independent. Using Lemma~\ref{lem:G-S}, we see that for every finite-dimensional subspace $V\subseteq W$ of an ultrametric $\K$-vector space $W$ satisfying~\eqref{eq:norm_scaling} and~\eqref{eq:norm_values}, $V$ admits a basis which is an isometric tuple; we call these \emph{isometric bases}. It is easy to show, in fact, that a tuple $(u_1,\ldots,u_k)$ in $\K^n$ is isometric if and only if (a) each $u_i\in\cal O_n$, and (b) their reductions $(u_1+\varpi\cal O_n,\ldots,u_k+\varpi\cal O_n)$ are linearly independent in the vector space $\cal O_n/\varpi\cal O_n$ over the residue field $\cal O_1/\varpi\cal O_1$.

    By virtue of the latter, an $n$-tuple $(u_1,\ldots,u_n)$ in $\K^n$ is isometric precisely when the matrix $U=[\,u_1\,\cdots\, u_n\,]$ with the $u_j$ as columns, belongs to $\operatorname{GL}_n(\cal O_1)$.
\end{remark}

\begin{remark}
    Let $u_1,\ldots,u_n$ be an isometric basis of an ultrametric $\K$-vector space $V$ satisfying~\eqref{eq:norm_scaling} and~\eqref{eq:norm_values}. In particular, the module $\cal O_1\langle u_1,\ldots,u_n\rangle$ generated by $u_1,\ldots,u_n$, i.e. the set of $\cal O_1$-linear combinations of $u_1,\ldots,u_n$, satisfies $\cal O_1\langle u_1,\ldots,u_n\rangle=\cal O_V$.
\end{remark}

The first special feature of $\K$-vector spaces, as against real vector spaces, is that metric balls are also subgroup cosets.

\begin{definition}[Regular subgroups]
    For a finite-dimensional ultrametric $\K$-vector space $V$ satisfying~\ref{eq:norm_scaling} and~\ref{eq:norm_values}, we write $\cal N(V)$ for the family of sets of the form $A.(\cal O_V)$, with $A\in\mathrm{GL}(V)$. Each $C\in\cal N(V)$ is in particular a compact open subgroup of $V$. We refer to members of $\cal N(V)$ as \emph{regular subgroups}.
\end{definition}

There are two ways of specifying the class $\cal N(V)$.
\begin{lemma}
    Let $V$ be a finite-dimensional ultrametric vector space over $\K$ satisfying~\eqref{eq:norm_scaling} and~\eqref{eq:norm_values}. Suppose $C\subseteq V$. Then the following are equivalent:
    \begin{itemize}
        \item[(a)] $C\in\cal N(V)$.
        \item[(b)] $C$ is compact, open, nonempty, and is an $\cal O_1$-module; that is, $C$ is closed under linear combinations with coefficients taken from $\cal O_1$.
    \end{itemize}
\end{lemma}
\begin{proof}
    It is easy to see that (a) implies (b). Consider the reverse. Write $n=\dim_\K V$. By virtue of being compact, we in particular have that $C$ is a submodule of a free $\cal O_1$-module of the form $\cal O_1\langle v_1,\ldots,v_n\rangle$ for some $\K$-linearly independent vectors $v_1,\ldots,v_n$. By~\cite{rotman2010advanced}, Corollary B-4.115, we see that $C$ is a free $\cal O_1$-module (using the fact that $\cal O_1$ is a PID). Since $C$ is open, it must be rank $n$ as an $\cal O_1$-module. Thus, $C$ is of the form $C=\cal O_1\langle w_1,\ldots,w_n\rangle$ for some $\K$-linearly independent $w_1,\ldots,w_n\in V$.

    Fix now $(u_1,\ldots,u_n)$ an isometric basis of $V$, and let $T:V\to V$ be the linear map defined by $T(u_j)=w_j$. It follows that
    \begin{equation*}
        T(\cal O_V)=T(\cal O_1\langle u_1,\ldots,u_n\rangle)=\cal O_1\langle w_1,\ldots,w_m\rangle=C.
    \end{equation*}
    Since $w_1,\ldots,w_n$ are $\K$-linearly independent, $T$ is invertible. Thus, $C\in\cal N(V)$, as was to be shown.
\end{proof}

\begin{remark}
    If $\mathrm{char}(\K)\neq 2$, then we may add a third equivalent condition: $C$ is compact, open, and \emph{convex} in the sense that $x,y\in C$ and $|\lambda|\leq 1$ implies $\lambda x+(1-\lambda)y\in C$. The characteristic $2$ case fails by virtue of ``convex hyperbolae'' such as $\{(x,y)\in\K^2:|x|,|y|\leq 1,\,\,|xy|<1\}$.
\end{remark}

We may now point out that there is only one ultrametric $\K$-vector space satisfying~\eqref{eq:norm_scaling} and~\eqref{eq:norm_values} in a given dimension.

\begin{proposition}
    Let $V$ be an $n$-dimensional vector space over $\K$, equipped with an ultrametric norm $\|-\|_V$ satisfying~\eqref{eq:norm_scaling} and~\eqref{eq:norm_values}. Then there is a linear isometry $A:V\to\K^n$, with the latter equipped with the usual max-norm.
\end{proposition}
\begin{proof}
    Take first a basis $v_1,\ldots,v_n$ of $V$. Letting $T:V\to\K^n$ be the linear transformation sending $v_i$ to $\bf e_i$, and writing $\|x\|'=\|T^{-1}x\|_V$, we obtain an ultrametric norm $\|-\|'$ on $\K^n$. The unit ball $C$ of $\K^n$ under $\|-\|'$ is then a nonempty compact open $\cal O_1$-module, and thus $C=B(\cal O_n)$ for some $B\in\mathrm{GL}_n(\K)$. In particular, for $y\in\K^n$, then (writing $y=Bz$) we have (using the value-restriction and the compatibility law on $\|-\|_V$)
    \begin{equation*}
        \|y\|'=\inf\{|\lambda|:\lambda^{-1}y\in C\}=\inf\{|\lambda|:\lambda^{-1}Bz\in C\}=\inf\{|\lambda|:\lambda^{-1}z\in\cal O_n\}=\|z\|.
    \end{equation*}
    Write then
    \begin{equation*}
        w_i=T^{-1}(B(\bf e_i)).
    \end{equation*}
    Then, for any $\alpha_1,\ldots,\alpha_n\in\K$,
    \begin{equation*}
        \|\alpha_1w_1+\ldots+\alpha_nw_n\|_V=\|B(\alpha_1\bf e_1+\ldots+\alpha_n\bf e_n)\|'=\|\alpha_1\bf e_1+\ldots+\alpha_n\bf e_n\|=\max_i|\alpha_i|.
    \end{equation*}
    In particular, the linear map $A=B^{-1}T:V\to\K^n$, $A(w_i)=\bf e_i$, defines a suitable linear isometry.
\end{proof}

As a consequence, we will generally be free to just consider the spaces $\K^n$. 

We have said that a $\K$-vector space will always be equipped with a norm. We should clarify what this means for the specific examples of vector spaces, other than $\K^n$, that will turn up.

\begin{definition}[Normed subspace]
    If $V$ is a finite-dimensional normed $\K$-vector space, and $W\subseteq V$ is a vector subspace, then we understand $W$ to be equipped with the restriction of the norm to $W$.
\end{definition}

More interesting is the dual notion.

\begin{definition}[Quotient subspace]
    If $W\subseteq V$ is a subspace of a normed space, then the abstract $\K$-vector space $V/W$ is equipped with the norm
    \begin{equation*}
        \|v+W\|=\inf_{w\in W}\|v+w\|.
    \end{equation*}
\end{definition}

It is easy to check that this properly defines an ultrametric norm on $V/W$, satisfying the compatibility conditions~\eqref{eq:norm_scaling} and~\eqref{eq:norm_values}.

\begin{definition}[Direct sum]
    If $V,W$ are normed $\K$-vector spaces, then the $\K$-vector space $V\oplus W$, equipped with the norm $\|(x,y)\|=\max(\|x\|_V,\|y\|_W)$, is the \emph{direct sum} of $V$ and $W$.
\end{definition}

The relevant fact of the choices above is the following splitting relation.

\begin{lemma}[Splitting lemma]\label{lem:splitting}
    Let $V$ be a $\K$-vector space and $W$ a subspace. Then $V$ is $\K$-linearly isometric to $W\oplus (V/W)$.
\end{lemma}
\begin{proof}
    It is easy to see that, up to isometric isomorphism, we may assume that $V=\K^n$ and $W=\K^k\times\{0\}^{n-k}$. Then
    \begin{equation*}
        \|(x_1,\ldots,x_n)\|_{V/W}=\max_{k<i\leq n}|x_i|,
    \end{equation*}
    so that, if we write
    \begin{equation*}
        \Phi:W\oplus(V/W)\to V,\quad\Phi((x_1,\ldots,x_k,0,\ldots,0),(y_1,\ldots,y_n)+W)=(x_1,\ldots,x_k,y_{k+1},\ldots,y_n)
    \end{equation*}
    defines an isometric isomorphism.
    
\end{proof}

Lastly, we will implicitly use the following matrix factorization throughout the manuscript.
\begin{lemma}[Smith normal form]\label{lem:smith}
    Suppose $A$ is an $n\times n$ matrix over $\K$ of rank $r$. Then there exist $B,C\in\operatorname{GL}_n(\cal O_1)$ such that
    \begin{equation*}
        BAC=\mathrm{diag}(\lambda_1,\ldots,\lambda_r,0,\ldots,0),
    \end{equation*}
    where (say) $0<|\lambda_1|\leq\cdots\leq |\lambda_r|$.
\end{lemma}
\begin{proof}
    Row/column reduction over $\cal O_1$.
\end{proof}

\section{Linear Brascamp--Lieb theory over non-Archimedean fields}\label{sec:lin_BL}

In this section, we discuss the general theory of linear Brascamp--Lieb inequalities over non-Archimedean fields. In the real theory, it is productive to focus ones attention on the restriction of Brascamp--Lieb functionals to Gaussian inputs; there, the optimization problems are finite-dimensional, and it is known since~\cite{lieb1990gaussian,bennett2008brascamp} that Brascamp--Lieb functionals are \emph{saturated} by Gaussians, i.e.\ that the constants are computable using only Gaussians. Over non-Archimedean fields, this is essentially true as well; in place of Gaussians, the relevant `simple' functions are \emph{indicators of regular subgroups}.

\begin{definition}[Regular subgroup Brascamp--Lieb] Given a Brascamp--Lieb datum $(\bf L,\bf c)$ and $G\in\cal N(\K^n)$, where $L_j$ has domain $\K^n$ for all $j$, we write
\begin{equation*}
    \operatorname{BL}_{\mathrm{grp}}(\bf L,\bf c;G)=\frac{\mu_n(G)}{\prod_{j=1}^m\mu_{n_j}(L_j(G))^{c_j}}\in(0,+\infty],
\end{equation*}
and
\begin{equation*}
    \operatorname{BL}_{\mathrm{grp}}(\bf L,\bf c)=\sup_{G\in\cal N(\K^n)}\operatorname{BL}(\bf L,\bf c;G).
\end{equation*}
\end{definition}

\begin{remark}[Directly comparing subgroup/functional Brascamp--Lieb] If $G\in\cal  N(\K^n)$, then upon writing $f_j=1_{L_j(G)}$ it is easy to discover that
\begin{equation*}
    \operatorname{BL}(\bf L,\bf c;\bf f)=\operatorname{BL}_{\mathrm{grp}}\Big(\bf L,\bf c;\bigcap_{j=1}^mL_j^{-1}(L_j(G))\Big).
\end{equation*}
Note that $G\subseteq \bigcap_{j=1}^mL_j^{-1}(L_j(G))$ and that $L_j\Big(\bigcap_{j=1}^mL_j^{-1}(L_j(G))\Big)=L_j(G)$. Thus (at least in the finite case $\operatorname{BL}_{\mathrm{grp}}(\bf L,\bf c)<+\infty$) we may freely replace each $G$ with $\bigcap_{j=1}^mL_j^{-1}(L_j(G))$, and thereafter assume that $G=\bigcap_{j=1}^mL_j^{-1}(L_j(G))$.

In particular,
\begin{equation*}
    \operatorname{BL}_{\mathrm{grp}}(\bf L,\bf c)\leq\operatorname{BL}(\bf L,\bf c).
\end{equation*}
It will transpire that the two are, in fact, always equal.
    
\end{remark}

We note that $\operatorname{BL}_{\mathrm{grp}}(\bf L,\bf c)<+\infty$ can only happen when the ``scaling condition'' and ``rank conditions'' hold, as in the real case. From now on, it will be helpful to assume that each $L_j$ is surjective:
\begin{definition}[Full Brascamp--Lieb datum]
    A Brascamp--Lieb datum $(\bf L,\bf c)$ is said to be \emph{full} if each $L_j$ is surjective onto its codomain.
\end{definition}

\begin{proposition}[Scaling and rank conditions]
    Suppose $(\bf L,\bf c)$ is a full linear Brascamp--Lieb datum such that $\operatorname{BL}_{\mathrm{grp}}(\bf L,\bf c)<+\infty$.

    Then, $\ker\bf L=\bigcap_{j=1}^m\ker L_j=\{0\}$. Also, $(\bf L,\bf c)$ satisfies the scaling condition
    \begin{equation}\label{eq:scaling}
        n=\sum_{j=1}^mc_jn_j,
    \end{equation}
    and the rank condition: for each $\bf 0\neq V\subsetneq\K^n$ proper subspace,
    \begin{equation}\label{ineq:rank}
        \dim V\leq\sum_{j=1}^m c_j\dim (L_jV)
    \end{equation}
\end{proposition}
\begin{proof}
    The condition $\ker\bf L=\{0\}$ is an immediate consequence of~\eqref{ineq:rank}. So, we handle the latter first.

    Let $\bf 0\neq V\subsetneq\K^n$ be a proper subspace. Let $h\in\K^\times$ have $|h|\leq 1$. Then
    \begin{equation*}
        \mu_n(\cal O_n\cap(V+h\cal O_n))=|h|^{n-\dim V},
    \end{equation*}
    and, when $h$ is sufficiently small,
    \begin{equation*}
        \mu_{n_j}(L_j(\cal O_n\cap(V+h\cal O_n)))=\eps|h|^{n_j-\dim(L_jV)}.
    \end{equation*}
    Thus,
    \begin{equation*}
        \operatorname{BL}_{\mathrm{grp}}(\bf L,\bf c;\cal O_n\cap(V+h\cal O_n))=|h|^{-\dim V+\sum_{j=1}^mc_j\dim(L_jV)}\eps^{-\sum_{j=1}^mc_j}.
    \end{equation*}
    Since $\sup_{0<|h|\leq 1}\operatorname{BL}_{\mathrm{grp}}(\bf L,\bf c;\cal O_n\cap(V+h\cal O_n))<+\infty$, it follows that $-\dim V+\sum_{j=1}^m c_j\dim(L_jV)\geq 0$. Thus,~\eqref{ineq:rank} holds.

    Now, we consider the scaling condition~\eqref{eq:scaling}. For $e\in\Z$, we have
    \begin{equation*}
        \operatorname{BL}_{\mathrm{grp}}(\bf L,\bf c;\varpi^e\cal O_n)=\frac{\mu_n(\varpi^e\cal O_n)}{\prod_{j=1}^m\mu_{n_j}(L_j(\varpi^e\cal O_n))^{c_j}}.
    \end{equation*}
    Clearly $\mu_n(\varpi^e\cal O_n)=|\varpi|^{en}\mu_n(\cal O_n)$. Since $L_j$ is surjective, it is in particular open, so $L_j(\varpi^e\cal O_n)=\varpi^eL_j(\cal O_n)$ is a compact open subset of $\K^{n_j}$; hence, $\mu_{n_j}( L_j(\varpi^e\cal O_n))^{c_j}=|\varpi|^{e c_jn_j}\mu_{n_j}(L_j(\cal O_n))$. Thus,
    \begin{equation*}
        \operatorname{BL}_{\mathrm{grp}}(\bf L,\bf c;\varpi^e\cal O_n)=|\varpi|^{e\big(n-\sum_{j=1}^mc_jn_j\big)}\operatorname{BL}_{\mathrm{grp}}(\bf L,\bf c;\cal O_n).
    \end{equation*}
    If $|\varpi|^{n-\sum_{j=1}^mc_jn_j}\neq 1$, then by either sending $e\to+\infty$ or $e\to-\infty$ we get $\operatorname{BL}_{\mathrm{grp}}(\bf L,\bf c)=+\infty$. Thus, we are forced to have $|\varpi|^{n-\sum_{j=1}^mc_jn_j}=1$, i.e.~\eqref{eq:scaling}.
\end{proof}

In the theory of Brascamp--Lieb, it is natural to rank regular subgroups by a measure of their regularity, chosen here to be their \emph{eccentricity}.
\begin{definition}[Eccentricity of a regular subgroup] Given a regular subgroup $G\in\cal N(\K^n)$, we write:
\begin{equation*}
    \mathrm{radius}(G)=\sup\{|\varpi^e|:e\in\Z,\quad\varpi^e\cal O_n\subseteq G\};
\end{equation*}
\begin{equation*}
    \mathrm{diam}(G)=\inf\{|\varpi^e|:e\in\Z,\quad G\subseteq\varpi^e\cal O_n\};
\end{equation*}
\begin{equation*}
    \mathrm{ecc}(G)=\frac{\mathrm{diam}(G)}{\mathrm{rad}(G)}.
\end{equation*}
In particular, $\mathrm{ecc}(G)$ is invariant under isotropic rescalings of $G$.
\end{definition}

It will be important to refer to a few elementary constructions in the category of Brascamp--Lieb data $(\bf L,\bf c)$.
\begin{definition}[Equivalence of data]
    Let $(\bf L^1,\bf c)$ and $(\bf L^2,\bf c)$ be Brascamp--Lieb data with a common weight vector $\bf c$, multilinearity $m$, and dimensions $n,\{n_j\}_{j=1}^m$. An \emph{equivalence $\bf R:(\bf L^1,\bf c)\overset{\simeq}{\longrightarrow}(\bf L^2,\bf c)$} is a collection of isomorphisms (``intertwining maps'')
    \begin{equation*}
        R:\K^n\to\K^n,\quad R_j:\K^{n_j}\to\K^{n_j},
    \end{equation*}
    such that
    \begin{equation*}
        R_j^{-1}\circ L_j^2\circ R=L_j^1,\quad 1\leq j\leq m.
    \end{equation*}
    An equivalence will be said to be \emph{rigid} if $R$ and each $R_j$ is an isometry.
\end{definition}
\begin{remark}
    Let $(\bf L^1,\bf c)$ and $(\bf L^2,\bf c)$ be data with $\bf R:(\bf L^1,\bf c)\to (\bf L^2,\bf c)$ an equivalence. Suppose $\bf f=(f_j)_{1\leq j\leq m}$ are nonzero nonnegative Schwartz--Bruhat functions on the $\K^{n_j}$. Write $\bf g=(g_j)_{1\leq j\leq m}$ for their pullbacks under the $R_j$: $g_j=f_j\circ R_j$. Then, since
    \begin{equation*}
        f_j\circ L_j^2=g_j\circ L_j^1\circ R^{-1},
    \end{equation*}
    we obtain
    \begin{equation*}
        \int_{\K^n}\prod_{j=1}^m f_j^{c_j}\circ L_j^2=|\det R|\int_{\K^n}\prod_{j=1}^m g_j^{c_j}\circ L_j^1,\quad\int_{\K^{n_j}}f_j=|\det R_j|\int_{\K^{n_j}}g_j.
    \end{equation*}
    Thus,
    \begin{equation*}
        \operatorname{BL}(\bf L^1,\bf c;\bf g)=\frac{\prod_{j=1}^m|\det R_j|^{c_j}}{|\det R|}\operatorname{BL}(\bf L^2,\bf c;\bf f),
    \end{equation*}
    and so certainly
    \begin{equation*}
        \operatorname{BL}(\bf L^1,\bf c)=\frac{\prod_{j=1}^m|\det R_j|^{c_j}}{|\det R|}\operatorname{BL}(\bf L^2,\bf c)
    \end{equation*}
    Similarly, given $G\in\cal N(\K^n)$, we may write $G'=R(G)$, so that
    \begin{equation*}
        \begin{split}
        \operatorname{BL}_{\mathrm{grp}}(\bf L^1,\bf c;G)=\frac{\mu_n(G)}{\prod_{j=1}^m\mu_{n_j}(L_j^1(G))^{c_j}}&=\frac{\mu_n(R^{-1}(G'))}{\prod_{j=1}^m\mu_{n_j}(R_j^{-1}(L_j^2(G')))^{c_j}}\\
        &=\frac{\prod_{j=1}^m|\det R_j|^{c_j}}{|\det R|}\operatorname{BL}_{\mathrm{grp}}(\bf L^2,\bf c;G'),
        \end{split}
    \end{equation*}
    and so certainly
    \begin{equation*}
        \operatorname{BL}_{\mathrm{grp}}(\bf L^1,\bf c)=\frac{\prod_{j=1}^m|\det R_j|^{c_j}}{|\det R|}\operatorname{BL}_{\mathrm{grp}}(\bf L^2,\bf c)
    \end{equation*}
    Lastly, note that if $(\bf L^1,\bf c)$ and $(\bf L^2,\bf c)$ are \emph{rigidly equivalent}, then we simply have $\operatorname{BL}(\bf L^1,\bf c)=\operatorname{BL}(\bf L^2,\bf c)$ and $\operatorname{BL}_{\mathrm{grp}}(\bf L^1,\bf c)=\operatorname{BL}_{\mathrm{grp}}(\bf L^2,\bf c)$.
\end{remark}

\begin{definition}[Direct sum]
    Suppose $(\bf L^1,\bf c)$ and $(\bf L^2,\bf c)$ are two Brascamp--Lieb data with common weights $\bf c$, and such that each $\bf L_i$ consists of the same number $m$ of linear maps. Then the \emph{direct sum datum} $(\bf L^1\oplus\bf L^2,\bf c)$ is the datum with weights $\bf c$ and whose linear maps are the direct sum of the linear maps from $\bf L^1,\bf L^2$. Thus, given
    \begin{equation*}
        L_j^1:\K^n\to\K^{n_j},\,\,L_j^2:\K^{n'}\to\K^{n_j'},
    \end{equation*}
    we write
    \begin{equation*}
        L_j^1\oplus L_j^2:\K^{n+n'}\to\mathrm{im}(L_j^1\oplus L_j^2)\subseteq\K^{n_j+n_j'},\quad (L_j^1\oplus L_j^2)(x,y)=(L_j^1x,L_j^2y),
    \end{equation*}
    and $\bf L^1\oplus\bf L^2=(L_j^1\oplus L_j^2)_{1\leq j\leq m}$.
\end{definition}

\begin{remark}[Necessary conditions sum]\label{rmk:dir_sum}
    By the definition of the $L_j^1\oplus L_j^2$, we see that $\bf L^1,\bf L^2$ full implies $\bf L^1\oplus\bf L^2$ full. If $(\bf L^1,\bf c)$ and $(\bf L^2,\bf c)$ satisfy the rank and scaling conditions, then so does $(\bf L^1\oplus \bf L^2,\bf c)$. These are checked in the real case in~\cite{bennett2008brascamp}; the $\K$ case is no different.

    The Brascamp--Lieb constants themselves multiply under this construction, as well. If $\bf f=(f_{j})_{1\leq j\leq m}$ and $\bf g=(g_j)_{1\leq j\leq m}$ are nonzero nonnegative functions, then upon writing
    \begin{equation*}
        \bf f\otimes\bf g=(f_j\otimes g_j)_{1\leq j\leq m},\quad (f_j\otimes g_j)(x_1,\ldots,x_{n_j},y_1,\ldots,y_{n_j'})=f_j(x_1,\ldots,x_{n_j})g_j(y_1,\ldots,y_{n_j'}),
    \end{equation*}
    we have
    \begin{equation*}
        \operatorname{BL}(\bf L^1\oplus\bf L^2,\bf c;\bf f\otimes\bf g)=\operatorname{BL}(\bf L^1,\bf c;\bf f)\operatorname{BL}(\bf L^2,\bf c;\bf g).
    \end{equation*}
    In particular, $\operatorname{BL}(\bf L^1,\bf c)\operatorname{BL}(\bf L^2,\bf c)\leq\operatorname{BL}(\bf L^1\oplus\bf L^2,\bf c)$. The other inequality follows by a Fubini--Tonelli argument: for a given $\bf f=(f_j)_j$ tuple of nonnegative nonzero Schwartz--Bruhat functions on the $\K^{n_j+n_j'}$, then
    \begin{equation*}
        \begin{split}
        \int_{\K^n}\int_{\K^{n'}}\prod_{j=1}^mf_j^{c_j}(L_j^1x,L_j^2y)d\mu_{n'}(y)d\mu_n(x)&\leq\operatorname{BL}(\bf L^2,\bf c)\int_{\K^n}\prod_{j=1}^m\left(\int_{\K^{n_j'}}f_j(L_j^1x,z')d\mu_{n_j'}(z')\right)^{c_j}d\mu_n(x)\\
        &\leq\operatorname{BL}(\bf L^1,\bf c)\operatorname{BL}(\bf L^2,\bf c)\prod_{j=1}^m\left(\int_{\K^{n_j+n_j'}}f_j(z,z')d\mu_{n_j+n_j'}(z,z')\right)^{c_j}.
        \end{split}
    \end{equation*}
    So, $\operatorname{BL}(\bf L^1\oplus\bf L^2,\bf c)=\operatorname{BL}(\bf L^1,\bf c)\operatorname{BL}(\bf L^2,\bf c)$.

    Similarly, if $G\in\cal N(\K^n)$ and $H\in\cal N(\K^{n'})$, then
    \begin{equation*}
        \operatorname{BL}_{\mathrm{grp}}(\bf L^1\oplus\bf L^2,\bf c;G\times H)=\operatorname{BL}_{\mathrm{grp}}(\bf L^1,\bf c;G)\operatorname{BL}_{\mathrm{grp}}(\bf L^2,\bf c;H),
    \end{equation*}
    and in particular $\operatorname{BL}_{\mathrm{grp}}(\bf L^1,\bf c)\operatorname{BL}(\bf L^2,\bf c)\leq\operatorname{BL}_{\mathrm{grp}}(\bf L^1\oplus\bf L^2,\bf c)$. On the other hand, if $G\in\cal N(\K^{n+n'})$, then writing $G'=G\cap(\K^n\times\{0\}^{n'})$ and $G''=G/(\{0\}^n\times\K^{n'})$ we have
    \begin{equation*}
        \mu_{n+n'}(G)=\mu_n(G')\mu_{n'}(G''),
    \end{equation*}
    and it follows that
    \begin{equation*}
        \operatorname{BL}_{\mathrm{grp}}(\bf L^1\oplus\bf L^2,\bf c;G)=\operatorname{BL}_{\mathrm{grp}}(\bf L^1,\bf c;G')\operatorname{BL}_{\mathrm{grp}}(\bf L^2,\bf c;G'')
    \end{equation*}
    and so $\operatorname{BL}_{\mathrm{grp}}(\bf L^1\oplus\bf L^2,\bf c)=\operatorname{BL}_{\mathrm{grp}}(\bf L^1,\bf c)\operatorname{BL}_{\mathrm{grp}}(\bf L^2,\bf c)$.
\end{remark}

Our last operation on Brascamp--Lieb data is that of removing zero exponents.

\begin{definition}[Contraction of trivial exponents]
    Let $(\bf L,\bf c)$ be a Brascamp--Lieb datum. Then the \emph{zero-contraction of $(\bf L,\bf c)$} is another $(\bf L',\bf c')$ where we have omitted those $L_j,c_j$ for which $c_j=0$, and preserved the rest.
\end{definition}
\begin{remark}
    Suppose $(\bf L,\bf c)$ is a Brascamp--Lieb datum with zero-contraction $(\bf L',\bf c')$. For any tuple $\bf f=(f_j)_{1\leq j\leq m}$ of nonzero nonnegative functions $f_j$, we have
    \begin{equation*}
        \operatorname{BL}(\bf L,\bf c;\bf f)=\operatorname{BL}(\bf L',\bf c';\bf f'),
    \end{equation*}
    where $\bf f'$ is obtained by also removing those $f_j$ corresponding to $c_j=0$. Similarly, if $G\in\cal N(\K^n)$, then
    \begin{equation*}
        \operatorname{BL}_{\mathrm{grp}}(\bf L,\bf c; G)=\operatorname{BL}_{\mathrm{grp}}(\bf L',\bf c';G).
    \end{equation*}
    In particular, $\operatorname{BL}(\bf L,\bf c)=\operatorname{BL}(\bf L',\bf c')$ and $\operatorname{BL}_{\mathrm{grp}}(\bf L,\bf c)=\operatorname{BL}_{\mathrm{grp}}(\bf L',\bf c')$.
\end{remark}

\begin{remark}
    Suppose $(\bf L,\bf c)$ is a Brascamp--Lieb datum which is not necessarily full. Suppose that $L_j$ is not surjective. If $c_j>0$, then by choosing $f_j=1_{\cal O_{n_j}}$ (resp. $G=\cal O_n)$ we see that $\operatorname{BL}(\bf L,\bf c)=+\infty$ (resp. $\operatorname{BL}_{\mathrm{grp}}(\bf L,\bf c)=+\infty$). If $c_j=0$, then by zero-contraction we may remove $L_j$. Thus, in the case that $\operatorname{BL}(\bf L,\bf c)<+\infty$, we will always be able to freely assume that $\bf L$ is full.
\end{remark}

\section{Existence of extremizers: the linear case}\label{sec:lin_extr}

In this section we demonstrate the following equivalence
\begin{equation*}
    \begin{split}
        \operatorname{BL}(\bf L,\bf c)=\operatorname{BL}(\bf L,\bf c;\bf f)\quad&\text{for some tuple $\bf f=(f_j)_j$, with each $f_j\in\operatorname{SB}(\K^{n_j})$ nonzero nonnegative}\\
        &\quad\iff\quad\operatorname{BL}(\bf L,\bf c)<+\infty.
    \end{split}
\end{equation*}
That is, \emph{all} finite linear Brascamp--Lieb inequalities are extremized by some functions $\bf f$. Moreover, the functions may be taken to be the indicators of regular subgroups. In the case of ``simple'' Brascamp--Lieb data $(\bf L,\bf c)$, one can show rather more: the regular subgroup may be taken to be the unit ball $\cal O_n$; under additional assumptions, depending on the ``compression factors'' $\Phi_\alpha$, one may further show that $\cal O_n$ and its multiples are the \emph{only} extremizing regular subgroups.

The existence of extremizers will follow by a non-Archimedean variant of the ``heat flow'' method, which was used in~\cite{bennett2008brascamp} to establish the existence of extremizers for \emph{real} Brascamp--Lieb inequalities, in the case of \emph{semisimple} data. We first indicate where the argument of~\cite{bennett2008brascamp} fails to apply to our problem. The latter shows that an extremizer may be found by identifying a suitable tuple $\bf A=(A_j)_j$ of matrices satisfying the relation
\begin{equation*}
    A_j^{-1}=L_j\Big(\sum_{j=1}^mc_jL_j^\top A_jL_j\Big);
\end{equation*}
the extremizer is then identified as a tuple of Gaussians with covariance matrices $A_j$ as an extremizer. We quickly see an issue for us: the sum
\begin{equation*}
    \sum_{j=1}^mc_jL_j^\top A_jL_j
\end{equation*}
involves a product of \emph{real} parameters $c_j$ with matrices $L_j$ whose entries lie in $\K$!

We may now state the main result of this section.
\begin{theorem}[Regular subgroup Brascamp--Lieb inequalities have extremizers]\label{thm:lin_sub_extr}
    Suppose $(\bf L,\bf c)$ is a linear Brascamp--Lieb datum such that $\operatorname{BL}_{\mathrm{grp}}(\bf L,\bf c)<+\infty$. Then there is a $G\in\cal N(\K^n)$ such that $\operatorname{BL}_{\mathrm{grp}}(\bf L,\bf c)=\operatorname{BL}_{\mathrm{grp}}(\bf L,\bf c;G)$.

    Moreover, the eccentricity of $G$ may be bounded as follows. Let $0=V_1\subsetneq\cdots\subsetneq V_r=\K^n$ be a maximal sequence of critical subspaces. Let $I=\big\{1\leq i\leq r-1:\dim(V_{i+1})>\dim(V_i)+1\big\}$. If $I=\emptyset$, then we may take $G=\cal O_n$ and hence $\mathrm{ecc}(G)=1$. Otherwise,
    \begin{equation*}
        \mathrm{ecc}(G)\leq\max_{i\in I}\min_{\substack{\alpha> 0\\
        (\bf L|_{W_{i+1}/W_i,}\bf c)\text{ is $\alpha$-simple}}}\max_{W_i\subsetneq V\subsetneq W_{i+1}}\Phi_\alpha(V)^{-\alpha^{-1}m(\dim W_{i+1}-\dim W_i)^2}.
    \end{equation*}
    Here, we say that $(\bf L,\bf c)$ is $\alpha$-simple if
    \begin{equation*}
        \dim V\leq-\alpha+\sum_{j=1}^mc_j\dim L_jV,\quad\text{for all subspaces $0\subsetneq V\subsetneq\K^n$}
    \end{equation*}
    and $\Phi_\alpha(V)$ is defined in Definition~\ref{def:compression} below.
\end{theorem}

The quantities $\Phi_\alpha(V)$ used in the above theorem record a suitable ``compression factor,'' concerning how $k$-dimensional volumes in $V$ are compressed to lower-dimensional volumes under $L_j$. The utility of these factors appears in Lemma~\ref{lem:lb_plate} below; they are defined in Definition~\ref{def:compression}.

In subsection~\ref{subsec:compression}, we define the compression factors and demonstrate their main application. In subsection~\ref{subsec:linear_heat}, we prove a ``heat flow'' bound on subgroups, which reduces the subgroup Brascamp--Lieb functional to the case of plates. In subsection~\ref{subsec:bl_module}, we prove that in ``$R$-module Brascamp--Lieb'' theory, functional Brascamp--Lieb is controlled by submodule Brascamp--Lieb. Finally, in subsection~\ref{subsec:lin_extremizers_exist}, we put the previous theory together to prove Theorems~\ref{thm:lin_sub_extr} and~\ref{thm:main_extr_functional}.

\subsection{Compression factors}\label{subsec:compression}

In this section, we define compression factors and prove basic lemmas related to them.

\begin{definition}[$\alpha$-compression factor]\label{def:compression}
    Suppose $V$ is a vector subspace of $\K^n$ of dimension $1\leq k\leq n-1$, and $\alpha\geq 0$. We build up the following definition of $\Phi_\alpha(V)$.

    Given tuples $\bf c\in[0,1]^m$ and $\bf n=(n_j)_j$ with $1\leq n_j\leq n$,  write
    \begin{equation*}
        \cal C_{k,\alpha}(\bf c,\bf n)=\Big\{\bf k\in\N_{\geq 0}^m: k_j\leq\min(k,n_j),\quad k\leq-\alpha+\sum_{j=1}^mc_jk_j\Big\}.
    \end{equation*}
    When working with a Brascamp--Lieb datum $(\bf L,\bf c)$, we will understand the two $\bf c$ to agree, and we will take $\bf n$ for the dimensions of the codomains of the $L_j$. For $\ell\geq k$, write also
    \begin{equation*}
        \mathbb P_{k,\ell}=\Big\{\text{surjective rank-$k$ coordinate projections } \K^\ell\to\K^k\Big\};
    \end{equation*}
    thus, an element of $\mathbb{P}_{k,\ell}$ picks out a particular $k$-tuple from an $\ell$-tuple of real numbers. We understand these elements to be $k\times\ell$ matrices.

    Given an isometric tuple $u_1,\ldots,u_k$ in $\K^n$, arrange them into a $n\times k$ matrix $U=[u_1\cdots u_k]$. We'll write $\cal U_k$ for the collection of all such matrices. Finally, we define the function
    \begin{equation*}
        \Phi_{k,\alpha}(U)=\max_{\bf k\in\cal C_{k,\alpha}(\bf c,\bf n)}\min_{1\leq j\leq m}\max_{P_j\in \mathbb P_{k_j,n_j}}\max_{Q_j\in\mathbb P_{k_j,k}}|\det(P_jL_jUQ_j^\top)|^{1/n_j}.
    \end{equation*}
    Thus, $\Phi_{k,\alpha}$ is a function with the property: for $\eps>0$,
    \begin{center}
        $\Phi_{k,\alpha}(U)\geq\eps$ if and only if there is $\bf k\in\cal C_{k,\alpha}(\bf c,\bf n)$ such that, for each $1\leq j\leq m$, there is a $k_j\times k_j$ minor of $L_jU$ of magnitude $\geq\eps^{n_j}$.
    \end{center}
    The exponent $n_j$ on the $\eps$ will be useful in cleaning up our main result. Since an implicit $(\bf L,\bf c)$ is present, we suppress both in the notation for $\Phi_{k,\alpha}$. Note that $\Phi_{k,\alpha}$ is a max/min combination of a finite number of polynomial functions in the entries of $U$ and $L_j$; in particular, it is continuous in the latter.

    Finally, recalling that $V$ is a $k$-dimension subspace of $\K^n$, we write $\Phi_\alpha(V)=\Phi_{k,\alpha}(U)$, where $U$ is an $n\times k$ matrix whose columns form an isometric basis of $V$. By~\ref{rmk:isom_tuple_facts}, such a $U$ always exists. It is reasonable to wonder if this definition depends on $U$; in Lemma~\ref{lem:comp_factor_founded} below, we will show that $U$ may be freely chosen. For the sake of exposition, we postpone the proof at this time.
\end{definition}

    We offer the following interpretation of the $\Phi_{k,\alpha}$. If the datum $(\bf L,\bf c)$ satisfies the rank condition, then for each $U\in\cal U_k$ whose columns span a subspace $V$, the inequality
    \begin{equation*}
        \dim V\leq\sum_jc_j\dim L_jV
    \end{equation*}
    implies that, for each $1\leq j\leq m$, writing $k_j=\dim L_jV$, the matrix $L_jU$ has rank $k_j$. Thus, $L_jU$ has a nonzero $k_j\times k_j$ minor. In particular, $\Phi_{k,0}(U)>0$. Since $\Phi_{k,0}$ is continuous with compact domain, and is pointwise positive, it follows that there is an $\eps>0$ such that $\Phi_{k,0}(U)\geq\eps$ for all $U\in\cal U_k$.

    Similarly, if $(\bf L,\bf c)$ is \emph{simple}, then there is an $\alpha>0$ and $\eps>0$ such that $\Phi_{k,\alpha}(U)\geq\eps$ for all $1\leq k\leq n-1$ and $U\in\cal U_k$.

    Having established the above, we now state the main lemma of this subsection.

    \begin{lemma}[Projections of of plates are large]\label{lem:lb_plate}
    Suppose $\eps>0,\alpha\geq 0$ are such that $\Phi_{k,\alpha}(U)\geq\eps$ for all $U\in\cal U_k$. Let $V$ be a $k$-dimensional subspace with isometric basis $u_1,\ldots,u_k$. Then, since $\Phi_{k,\alpha}([u_1\cdots u_k])\geq\eps$, for each $j$ there is a $k_j$ such that $L_jU$ has a $k_j$-minor of magnitude $\geq \eps^{n_j}$, and such that $(k_j)_j\in\cal C_{k,\alpha}(\bf c,\bf n)$. Let $e\in\Z_{\geq 0}$ be arbitrary. Then
    \begin{equation*}
        \mu_{n_j}\Big(L_j\big[\cal O_n\cap(V+\varpi^e\cal O_n)\big]\Big)\geq\eps^{n_j}|\varpi|^{e(n_j-k_j)}.
    \end{equation*}
        
    \end{lemma}
    \begin{proof}
        We may freely assume that $V=\K^k\times\{0\}^{n-k}$, that $u_\iota=\bf e_\iota$ are the standard basis vectors, and that the $k_j$-minor of interest is the top-left minor in $L_j.[u_1\cdots u_k]=L_j.\begin{bmatrix}I_k\\0_{(n-k)\times k}\end{bmatrix}$ (which we may note is just the top-left minor of $L_j$ itself). Thus, writing $A$ for the top-left $k_j\times k_j$ submatrix of $L_j$, the image $A(\cal O_{k_j})$ is a volume $\geq\eps^{n_j}$ regular subgroup of $\K^{k_j}$, which (by normal form of $\bf L$) is contained in $\cal O_{k_j}$.

        Now,
        \begin{equation*}
            \begin{split}
            L_j[\cal O_n\cap(V+\varpi^e\cal O_n)]=L_j[\cal O_k\times\{0\}^{n-k}+\varpi^e\cal O_n]&=L_j[\cal O_k\times\{0\}^{n-k}]+\varpi^e\cal O_{n_j}\\
            &\supseteq L_j[\cal O_{k_j}\times\{0\}^{n-k_j}]+\varpi^e\cal O_{n_j}.
            \end{split}
        \end{equation*}
        It will suffice to lower-bound the measure of this last set. The lower bound is obtained by (a) noting that if a pair $(\bf x,\bf  y)\in\K^{k_j}\times\K^{n_j-k_j}$ belongs to this set, then $(\bf x,\bf z)$ belongs to the same set for any $\bf z\in\bf y+\cal O_{n_j-k_j}$, and (b) the set of $\bf x$ for which such a $\bf y$ exists amounts to a projection of $L_j[\cal O_{k_j}\times\{0\}^{n-k_j}]$, which is large by our assumption on the minor of $L_j$. In detail, we proceed using Fubini:
        \begin{equation*}
            \begin{split}
            \mu_{n_j}&\big(L_j[\cal O_{k_j}\times\{0\}^{n-k_j}]+\varpi^e\cal O_{n_j}\big)\\
            &=\int_{\K^{k_j}\times\{0\}^{n_j-k_j}}\int_{\{0\}^{k_j}\times\K^{n_j-k_j}}1_{L_j[\cal O_{k_j}\times\{0\}^{n_j-k_j}]+\varpi^e\cal O_{n_j}}(\bf x+\bf y)d\mu_{n_j-k_j}(\bf y)d\mu_{k_j}(\bf x)\\
            &\geq|\varpi|^{e(n_j-k_j)}\int_{\K^{k_j}\times\{0\}^{n_j-k_j}}1_{[I_{k_j}\,\,0_{k_j\times(n_j-k_j)}]\big(L_j[\cal O_{k_j}\times\{0\}^{n_j-k_j}]+\varpi^e\cal O_{n_j}\big)}(\bf x)d\mu_{k_j}(\bf x)\\
            &=|\varpi|^{e(n_j-k_j)}\mu_{k_j}\big(A(\cal O_{k_j})+\varpi^e\cal O_{k_j}\big).
            \end{split}
        \end{equation*}
        Finally, the trivial inclusion $A(\cal O_{k_j})\subseteq A(\cal O_{k_j})+\varpi^e\cal O_{k_j}$ and our assumed lower bound implies
        \begin{equation*}
            \mu_{n_j}\Big(L_j\big[\cal O_n\cap(V+\varpi^e\cal O_n)\big]\Big)\geq|\varpi|^{e(n_j-k_j)}\mu_{k_j}(A(\cal O_{k_j}))\geq\eps^{n_j}|\varpi|^{e(n_j-k_j)},
        \end{equation*}
        as was to be justified.
        \end{proof}

        We end this subsection by indicating why the compression factors $\Phi_\alpha(V)$ are well-defined.
        \begin{lemma}[Compression factors are well-defined]\label{lem:comp_factor_founded}
            Suppose $U,U'$ are two $n\times k$ matrices over $\K$ whose columns form isometric tuples, each of which is a basis for a common vector subspace $V$. Then $\Phi_{k,\alpha}(U)=\Phi_{k,\alpha}(U')$.
        \end{lemma}
        \begin{proof}
            Fix some $\bf k\in\cal C_{k,\alpha}$ and each $1\leq j\leq m$, and abbreviate $L_jU=\bar U,L_jU'=\bar U'$. Writing
            \begin{equation*}
                \Gamma(\Xi)=\max_{P_j\in\mathbb{P}_{k_j,n_j}}\max_{Q_j\in\mathbb{P}_{k_j,k}}|\det(P_j\Xi Q_j^\top)|^{1/n_j},         
            \end{equation*}
            for arbitrary $n_j\times k$ matrices $\Xi$ of rank $\geq k_j$, our goal will be to show that
            \begin{equation*}
                \Gamma(\bar U)=\Gamma(\bar U').
            \end{equation*}
            We note in passing that $\Gamma$ is equivalently the functional defined by taking $|\det(-)|^{1/n_j}$ over all submatrices, and choosing the largest such value. To establish our equality, note first that $\bar U'=\bar UR$, for some isometry $R\in\operatorname{GL}(\cal O_k)$. Each such $R$ is a finite product of elementary matrices, representing the column operations of (a) swapping two columns, or (b) replacing column $i$ by column $i$ plus $\beta$ times column $j$, where $|\beta|\leq 1$, or (c) multiplying column $i$ by a scalar $\beta$ with $|\beta|=1$.

            We claim that $\Xi\mapsto\Gamma(\Xi)$ is invariant under such moves. Indeed, swapping two columns of $\Xi$ simply changes the choice of the submatrix (more precisely, we change the choice of $Q_j$). Applying operation (c) similarly does not alter the absolute determinant. If $R$ has the effect (b), and column $i$ was among the columns defining submatrix $\Xi$, and column $j$ was not, then (writing $\Xi'$ for the same-location submatrix, after the column operation),
            \begin{equation*}
                |\det (\Xi')|=|\det (\Xi)+\beta\det (\Xi'')|;
            \end{equation*}
            here, $\Xi''$ is obtained by swapping columns $i$ and $j$. If $|\beta|<1$ and $\Xi$ maximizes $\Gamma$, then (since some $k_j\times k_j$ submatrix of $\Xi$ nonsingular) $|\det (\Xi)|>|\beta\det (\Xi'')|$, and so $|\det(\Xi')|=|\det (\Xi)|$. If instead $|\beta|=1$, then either $|\det \Xi'|=|\det \Xi|$ and we are done, or else
            \begin{equation*}
                |\det (\Xi)+\beta\det (\Xi'')|<|\det (\Xi)|,
            \end{equation*}
            and in particular $|\det (\Xi)|=|\det (\Xi'')|$. But then $|\det (\Xi'')|>|\det (\Xi')|$, and $\Xi''$ appears as a submatrix of $L_jUR$ (select column $j$ instead of column $i$!), so the maximum over submatrices is unaffected after swapping the choice of submatrix. The last case is that columns $i$ and $j$ are both among the columns of the chosen submatrix; but then $|\det (\Xi)|=|\det (\Xi)||\det R|=|\det (\Xi R)|=|\det (\Xi')|$, because applying the column operation to $\Xi$ and then passing to the submatrix $P$ has the same effect as just applying the column operation to $P$!

            Thus, $\Gamma$ is invariant under $\Xi\mapsto\Xi R$, for any $R\in\mathrm{GL}(\cal O_k)$. Since $\bar U$ and $\bar U'$ differ by such a move, we conclude that $\Gamma(\bar U)=\Gamma(\bar U')$. Since we have established this for every choice of $\bf k$ and $1\leq j\leq m$, we certainly have that $\Phi_{k,\alpha}(U)=\Phi_{k,\alpha}(U')$.
        \end{proof}

\subsection{Linear heat flow}\label{subsec:linear_heat}

We start to build towards the proof of Theorem~\ref{thm:lin_sub_extr}. As described above, the method will be a version of the heat flow argument. Lemma~\ref{lem:lb_plate} will be critically used to show monotonicity under the flow.

It will be convenient to assume that, by using symmetries of the Brascamp--Lieb problem (i.e.\ performing linear changes of variables in $\K^n$ and each $\K^{n_j}$), the maps $\bf L$ are taken to satsify a suitable ``normal form'' hypothesis. This guarantees that $\bf L$ is an isometric embedding into $\prod_{j=1}^m\K^{n_j}$.

\definition[Normal form] A tuple of maps $\bf L=(L_j)_j$ is said to be in \emph{normal form} if $L_j(\cal O_n)=\cal O_{n_j}$ for all $j$, and $\bigcap_{j=1}^mL_j^{-1}(\cal O_{n_j})=\cal O_n$. A Brascamp--Lieb datum $(\bf L,\bf c)$ is in normal form just when $\bf L$ is.

\begin{remark}
    Any $(\bf L,\bf c)$ is equivalent to some $(\bf L',\bf c)$ in normal form.
\end{remark}

A technical upshot of the normal form hypothesis is that ``factor-wise thickening'' is equivalent to ``ordinary thickening,'' in the following technical sense, for any regular subgroup that is a candidate for being an extremizer of $\operatorname{BL}_{\mathrm{grp}}(\bf L,\bf c)$.

\begin{lemma}[Thickening lemma]\label{lem:thick}
    Suppose $\bf L=(L_j)_{1\leq j\leq m}$ is injective (i.e.\ $\bigcap_j\ker L_j=0$) and in normal form, and that $U\in\cal N(\K^n)$ satisfies
    \begin{equation*}
        U=\bigcap_{j=1}^mL_j^{-1}\big(L_j(U)\big).
    \end{equation*}
    Then, for any $h\in\K^\times$,
    \begin{equation*}
        U+\varpi^e\cal O_n=\bigcap_{j=1}^mL_j^{-1}\big(L_j(U+h\cal O_n)\big).
    \end{equation*}
\end{lemma}
\begin{proof}
    The containment $\subseteq$ is trivial. Consider $\bf L$ as a map $\K^n\to\K^{n_1}\times\cdots\times\K^{n_m}$. Since $\bf L$ is injective, it suffices to show that
    \begin{equation*}
        \bf L(U+h\cal O_n)=\bigcap_{j=1}^m\bf L\Big(L_j^{-1}\big(L_j(U+h\cal O_n)\big)\Big).
    \end{equation*}
    Write $H=\bf L(\K^n)$, and
    \begin{equation*}
        V=L_1(U)\times\cdots\times L_m(U).
    \end{equation*}
    Note that, if $K\in\cal N(\K^n)$ satisfies
    \begin{equation}\label{eq:k_reg}
        K=\bigcap_{j=1}^mL_j^{-1}(L_jK),
    \end{equation}
    then
    \begin{equation}\label{eq:set_lem}
        \bf L(K)=\big(L_1(K)\times\cdots\times L_m(K)\big)\cap H.
    \end{equation}
    The containment $\subseteq$ is trivial; for the other, if $\bf y$ belongs to the right-hand side, then (by virtue of $\bf y\in H$) we may find $\bf x\in\K^n$ with $\bf L(\bf x)=\bf y$. If we write
    \begin{equation*}
        \pi_j:\K^{n_1}\times\cdots\times\K^{n_m}\to\K^{n_j}
    \end{equation*}
    for the projection map, then $L_j\bf x=\pi_j\bf y\in L_j(K)$, so
    \begin{equation*}
        \bf x\in\bigcap_{j=1}^m(\pi_j\circ\bf L)^{-1}(L_j(K))=\bigcap_{j=1}^mL_j^{-1}(L_j(K)),
    \end{equation*}
    so~\eqref{eq:set_lem} holds upon applying~\eqref{eq:k_reg} and $\bf L$. Consequently,
    \begin{equation*}
        \begin{split}
        \bf L(U+h\cal O_n)&=\bf L(U)+h\bf L(\cal O_n)\\
        &=H\cap\Big(L_1(U)\times\cdots\times L_m(U)\Big)+H\cap h\cal O_{n_1+\ldots+n_m}\\
        &=H\cap\Big(L_1(U)\times\cdots\times L_m(U)+h\cal O_{n_1+\ldots+n_m}\Big).
        \end{split}
    \end{equation*}
    We use this to examine the right-hand side of the desired set identity. Note that
    \begin{equation*}
        \pi_j\big(\bf L\big (L_j^{-1}(L_j(U+h\cal O_n)\big)\big)=L_j(U)+h\cal O_{n_j},
    \end{equation*}
    so that
    \begin{equation*}
        \begin{split}
        \bigcap_{j=1}^m\bf L\Big(L_j^{-1}\big(L_j(U+h\cal O_n)\big)\Big)&\subseteq\bigcap_{j=1}^m\pi_j^{-1}\big(L_jU+h\cal O_{n_j}\big)\\
        &=L_1(U)\times\cdots\times L_m(U)+h\cal O_{n_1+\ldots+n_m},
        \end{split}
    \end{equation*}
    and of course
    \begin{equation*}
        \bigcap_{j=1}^m\bf L\Big(L_j^{-1}\big(L_j(U+h\cal O_n)\big)\Big)\subseteq H.
    \end{equation*}
    Thus
    \begin{equation*}
        \bigcap_{j=1}^m\bf L\Big(L_j^{-1}\big(L_j(U+h\cal O_n)\big)\Big)\subseteq H\cap \big(L_1(U)\times\cdots\times L_m(U)+h\cal O_{n_1+\ldots+n_m}\big)=\bf L\big(U+h\cal O_n\big).
    \end{equation*}
    Since $\bf L$ is injective,
    \begin{equation*}
        \bigcap_{j=1}^mL_j^{-1}\big(L_j(U+h\cal O_n)\big)\subseteq U+h\cal O_n.
    \end{equation*}
    From the trivial reverse inequality, we are done.
\end{proof}

\begin{proposition}[Normal form Ball's inequality, regular subgroup case]\label{prop:ball_subg}
    Suppose $\bf L$ is full, injective, and in normal form. Let $U\in\cal U(\K^n)$ be an arbitrary regular subgroup, which we assume to satisfy
    \begin{equation*}
        U=\bigcap_{j=1}^mL_j^{-1}(L_j(U)),
    \end{equation*}
    \begin{equation*}
        U\subseteq\cal O_n.
    \end{equation*}
    Then, for any $e\in\Z_{\geq 0}$,
    \begin{equation*}
        \operatorname{BL}_{\mathrm{grp}}(\bf L,\bf c;U)\leq\operatorname{BL}_{\mathrm{grp}}(\bf L,\bf c;U\cap\varpi^e\cal O_n)\operatorname{BL}_{\mathrm{grp}}(\bf L,\bf c;U+\varpi^e\cal O_n).
    \end{equation*}
\end{proposition}
\begin{proof}
    By replacing each $L_j$ with $R_j\circ L_j$ for a suitable isometry $R_j$, we may assume that each $L_jU$ takes the form
    \begin{equation*}
        L_jU=\varpi^{f_{1j}}\cal O_1\times\cdots\times\varpi^{f_{n_jj}}\cal O_1,\quad (0\leq f_{1j}\leq\cdots\leq f_{n_jj}).
    \end{equation*}
    Note that, under the normal form assumption, $\operatorname{BL}_{\mathrm{grp}}(\bf L,\bf c;\varpi^e\cal O_n)=1$. Thus, we have the calculation
    \begin{equation*}
        \begin{split}
            \mu_n(U)&=\int_{\K^n}\Big[\prod_{j=1}^m1_{L_jU}(L_jy)\Big]\operatorname{BL}_{\mathrm{grp}}(\bf L,\bf c;\varpi^e\cal O_n)dy\\
            &=\int_{\K^n}\Big[\prod_{j=1}^m 1_{L_jU}(L_jy)\Big]\int_{\K^n}\Big[\prod_{j=1}^m|\varpi|^{-en_jc_j}1_{\varpi^e\cal O_{n_j}}(L_j(x-y))\Big]dxdy\\
            &=\int_{\K^n}\int_{\K^n}\Big[\prod_{j=1}^m|\varpi|^{-en_jc_j}1_{\varpi^e\cal O_{n_j}}(L_j(x-y))1_{L_jU}(L_jy)\Big]dydx\\
            &=\int_{\K^n}\int_{\K^n}\prod_{j=1}^m h_{j,x}^{c_j}(L_jy)dydx,
        \end{split}
    \end{equation*}
    where
    \begin{equation*}
        h_{j,x}(z)=|\varpi|^{-en_j}1_{\varpi^e\cal O_{n_j}}(L_jx-z)1_{L_jU}(z).
    \end{equation*}
    Thus,
    \begin{equation*}
        \begin{split}
            \mu(U)&\leq\sup_{x\in\K^n}\operatorname{BL}(\bf L,\bf c;\bf h_x)\int_{\K^n}\prod_{j=1}^m(1_{L_jU}*|\varpi|^{-en_j}1_{\varpi^e\cal O_{n_j}})^{c_j}(L_jx)dx\\
            &\leq\sup_{x\in\K^n}\operatorname{BL}(\bf L,\bf c;\bf h_x)\operatorname{BL}(\bf L,\bf c;(1_{L_jU}*|\varpi|^{-en_j}1_{\varpi^e\cal O_{n_j}})_{1\leq j\leq m})\prod_{j=1}^m\mu_{n_j}(L_jU)^{c_j}.
        \end{split}
    \end{equation*}
    We analyze the two Brascamp--Lieb quantities on the right-hand side. Consider first
    \begin{equation*}
        \int_{\K^{n_j}} h_{j,x}(z)dz=|\varpi|^{-en_j}\mu_{n_j}(L_jU\cap(\varpi^e\cal O_{n_j}+L_jx)).
    \end{equation*}
    If $L_jU\cap(\varpi^e\cal O_{n_j}+L_jx)\neq\emptyset$, then any $a\in L_jU\cap(\varpi^e\cal O_{n_j}+L_jx)$ defines a translation $z\mapsto z+a$ carrying $L_jU\cap \varpi^e\cal O_{n_j}$ to $L_jU\cap(\varpi^e\cal O_{n_j}+L_jx)$; it follows that 
    \begin{equation*}
        \mu_{n_j}(L_jU\cap(\varpi^e\cal O_{n_j}+L_jx))=\begin{cases}
            \mu_{n_j}(L_jU\cap \varpi^e\cal O_{n_j}) & L_jx\in L_jU+\varpi^e\cal O_{n_j}\\
            0 & L_jx\not\in L_j(U)+\varpi^e\cal O_{n_j}.
        \end{cases}
    \end{equation*}
    On the other hand, if $x\in\bigcap_{j=1}^mL_j^{-1}(L_j(U)+\varpi^e\cal O_{n_j})$,
    \begin{equation*}
        \int_{\K^n}\prod_{j=1}^mh_{j,x}^{c_j}(L_jy)dy=|\varpi|^{-en}\int_{\K^n}\prod_{j=1}^m1_{\varpi^e\cal O_{n_j}}(L_j(x-y))1_{L_jU}(L_jy)dy=|\varpi|^{-en}\mu_n\Big(U\cap\varpi^e\cal O_n\Big).
    \end{equation*}
    Thus,
    \begin{equation*}
        \sup_x\operatorname{BL}(\bf L,\bf c;\bf h_x)\leq\operatorname{BL}_{\mathrm{grp}}(\bf L,\bf c;U\cap\varpi^e\cal O_n).
    \end{equation*}
    
    We now study the next Brascamp--Lieb quantity. For each $j$, if $1\leq\iota_j\leq n_j$ is such that
    \begin{equation*}
        f_{\iota_jj}\leq e<f_{(\iota_j+1)j},
    \end{equation*}
    then
    \begin{equation*}
        L_j(U+\varpi^e\cal O_n)=L_jU+\varpi^e\cal O_{n_j}=\varpi^{f_{1j}}\cal O_1\times\cdots\times\varpi^{f_{\iota_jj}}\cal O_1\times\varpi^e\cal O_{n_j-\iota_j},
    \end{equation*}
    so that
    \begin{equation*}
        \mu_{n_j}(L_jU)=|\varpi|^{\sum_{\iota=\iota_j+1}^{n_j}(f_{\iota j}-e)}\mu_{n_j}\big(L_j(U+\varpi^e\cal O_n)\big).
    \end{equation*}

    Considering the second factor,
    \begin{equation*}
        \int_{\K^{n_j}}1_{L_jU}*|\varpi|^{-en_j}1_{\varpi^e\cal O_{n_j}}(z)dz=\mu_{n_j}(L_jU)=|\varpi|^{\sum_{\iota=\iota_j+1}^{n_j}(f_{\iota j}-e)}\mu_{n_j}\big(L_j(U+\varpi^e\cal O_n)\big),
    \end{equation*}
    and, using Lemma~\ref{lem:thick},
    \begin{equation*}
        \begin{split}
        \int_{\K^n}\prod_{j=1}^m(1_{L_jU}*|\varpi|^{-en_j}1_{\varpi^e\cal O_{n_j}})^{c_j}(L_jy)dy&=\int_{\K^n}\prod_{j=1}^m|\varpi|^{c_j\sum_{\iota=\iota_j+1}^{n_j}(f_{\iota j}-e)}1_{L_jU+\varpi^e\cal O_{n_j}}(L_jy)dy\\
        &=|\varpi|^{\sum_{j=1}^mc_j\sum_{\iota=\iota_j+1}^{n_j}(f_{\iota j}-e)}\mu_n(U+\varpi^e\cal O_n).
        \end{split}
    \end{equation*}
    
    Thus,
    \begin{equation*}
        \operatorname{BL}(\bf L,\bf c;(1_{L_jU}*|\varpi|^{-en_j}1_{\varpi^e\cal O_{n_j}})_{1\leq j\leq m})=\frac{\mu_n(U+\varpi^e\cal O_n)}{\prod_{j=1}^m\mu_{n_j}\big(L_j(U+\varpi^e\cal O_n)\big)}=\operatorname{BL}_{\mathrm{grp}}(\bf L,\bf c;U+\varpi^e\cal O_n),
    \end{equation*}
    and we conclude that
    \begin{equation*}
        \operatorname{BL}_{\mathrm{grp}}(\bf L,\bf c;U)\leq\operatorname{BL}_{\mathrm{grp}}(\bf L,\bf c;U\cap\varpi^e\cal O_n)\operatorname{BL}_{\mathrm{grp}}\big(\bf L,\bf c;U+\varpi^e\cal O_n\big).
    \end{equation*}
\end{proof}

\begin{corollary}[Reduction to plates]\label{cor:subg_bd_by_subsp}
    Suppose $\bf L$ is full, injective, and in normal form. Suppose $\bf c\in[0,1]^m$ satisfies the scaling law $n=\sum_jc_jn_j$. Let $U\in\cal N(\K^n)$ have radius $|\varpi|^{f+e}$ and diameter $|\varpi|^{f}$, where $e,f\in\Z$ and $e\geq 0$. Then there is some $1\leq\kappa\leq n-1$, integers $e_1,\ldots, e_\kappa\in\Z_{\geq 0}$, and subspaces $V_1,\ldots,V_\kappa$ such that
    \begin{equation*}
        e_1+\ldots+e_\kappa=e,
    \end{equation*}
    and
    \begin{equation}\label{ineq:subg_bd_by_subsp}
        \operatorname{BL}_{\mathrm{grp}}(\bf L,\bf c; U)\leq\prod_{\iota=1}^\kappa\operatorname{BL}_{\mathrm{grp}}\big(\bf L,\bf c;\cal O_n\cap(V_\iota+\varpi^{e_\iota}\cal O_n)\big).
    \end{equation}
    
\end{corollary}

\begin{proof}
    By rescaling, we may take $f=0$, so that in particular $U\subseteq\cal O_n$; we may clearly assume $U\neq\cal O_n$, so that $e>0$.
    We may also freely assume that
    \begin{equation*}
        U=\bigcap_{j=1}^mL_j^{-1}(L_j(U)).
    \end{equation*}
    Replacing each $L_j$ with $L_j\circ R$ for a suitable isometry $R$, we may assume that $U$ takes the form
    \begin{equation*}
        U=\varpi^{f_1}\cal O_1\times\cdots\times\varpi^{f_n}\cal O_1,\quad (0\leq f_1\leq\cdots\leq f_n).
    \end{equation*}
    The diameter assumption implies that $f_1=0$. Since $U\neq\cal O_n$, we have $f_n=e>0$. Taking $e_1=\max(f_\iota:f_\iota<f_n)$ and applying Proposition~\ref{prop:ball_subg}, we get
    \begin{equation*}
        \operatorname{BL}_{\mathrm{grp}}(\bf L,\bf c; U)\leq\operatorname{BL}_{\mathrm{grp}}(\bf L,\bf c;U\cap\varpi^{e_1}\cal O_n)\operatorname{BL}_{\mathrm{grp}}(\bf L,\bf c;U+\varpi^{e_1}\cal O_n).
    \end{equation*}
    Note that
    \begin{equation*}
        U\cap\varpi^{e_1}\cal O_n=\varpi^{e_1}\cal O_\iota\times\varpi^{f_n}\cal O_{n-\iota},
    \end{equation*}
    where $1\leq\iota<n$ is such that $f_\iota=e_1$, $f_{\iota+1}=f_n$. Thus, using the scaling law,
    \begin{equation*}
        \operatorname{BL}_{\mathrm{grp}}(\bf L,\bf c;U\cap\varpi^{e_1}\cal O_n)=\operatorname{BL}_{\mathrm{grp}}(\bf L,\bf c;\cal O_\iota\times\varpi^{f_n-e_1}\cal O_{n-\iota}),
    \end{equation*}
    which is a factor as desired in our claimed result. On the other hand, note that $U+\varpi^e\cal O_n\in\cal N(\K^n)$ and has diameter $1=|\varpi|^0$, and radius $|\varpi|^{e_1}>|\varpi|^{f_n}$. Thus, applying the same inequality with $U$ replaced by $U+\varpi^{e_1}$, and continuing as long as $U\neq\cal O_n$, we eventually obtain
    \begin{equation*}
        \operatorname{BL}_{\mathrm{grp}}(\bf L,\bf c;U)\leq\operatorname{BL}_{\mathrm{grp}}(\bf L,\bf c;U\cap\varpi^{e_1}\cal O_n)\times\cdots\times\operatorname{BL}_{\mathrm{grp}}(\bf L,\bf c;U\cap\varpi^{e_\kappa}\cal O_n);
    \end{equation*}
    here $1\leq\kappa\leq n-1$ and $e_1+\ldots+e_\kappa$ is the a sum of differences of the form $f_\iota-f_j$, $\iota>j$, telescoping to $f_n-f_1=e$. Choosing subspaces $V_\iota$ to be subspaces of the form $\K^{t_\iota}\times\{0\}^{n-t_\iota}$ for suitable $t_\iota$, we obtain the desired result.

\end{proof}

\subsection{\texorpdfstring{H\"older}{Holder}--Brascamp--Lieb inequalities over modules}\label{subsec:bl_module}

In this subsection, we compare set and functional H\"older--Brascamp--Lieb inequalities over $R$-modules, for a finite ring $R$. Our analysis is intended to mimic both the Abelian group case (with homomorphisms) and vector space case (with linear maps). This is needed to show that, not only does linear Brascamp--Lieb over $\K$ reduces to the case of indicators of compact open subgroups, but further of \emph{regular} subgroups (that is to say, compact open $\cal O_1$-modules). The latter is needed in order to apply Lemma~\ref{lem:smith}, so that a typical candidate $\prod_j f_j\circ L_j$ takes the form $1_U$ with $U=\varpi^{f_1}\cal O_1\times\cdots\times\varpi^{f_n}\cal O_1$, up to symmetries. Our main result is Theorem~\ref{thm:module_reduction}.

Let $R$ be a commutative ring with identity, and $M,M_1,\ldots,M_m$ be $R$-modules. Suppose $\phi_j:M\to M_j$ are $R$-module homomorphisms, and $\bf c\in[0,1]^m$. We define the following functionals.
\begin{definition}[$R$-module Brascamp--Lieb]\label{def:rmod_BL}
    If $\bf f=(f_j)_{1\leq j\leq m}$, $f_j:M_j\to\R_{\geq 0}$ are finitely-supported and nonzero, then we define
    \begin{equation*}
        \operatorname{BL}^{R\text{-mod}}(\phi,\bf c;\bf f)=\frac{\sum_{x\in M}f_j(\phi_j(x))}{\prod_{j=1}^m(\sum_{a\in M_j}|f_j(a)|^{1/c_j})^{c_j}}.
    \end{equation*}
    We also write
    \begin{equation*}
        \operatorname{BL}^{R\text{-mod}}(\phi,\bf c)=\sup_{\bf f}\operatorname{BL}^{R\text{-mod}}(\phi,\bf c;\bf f)\in[0,+\infty].
    \end{equation*}
    If $\Omega\subseteq M$ is a nonzero submodule (which we write as $\Omega\leq M$), then we write
    \begin{equation*}
        \operatorname{BL}_{\mathrm{sub}}^{R\text{-mod}}(\phi,\bf c;\Omega)=\frac{\#\Omega}{\prod_{j=1}^m(\#\phi_j(\Omega))^{c_j}}\in[0,+\infty].
    \end{equation*}
    We write
    \begin{equation*}
        \operatorname{BL}_{\mathrm{sub}}^{R\text{-mod}}(\phi,\bf c)=\sup_{\Omega\leq M}\operatorname{BL}_{\mathrm{sub}}^{R\text{-mod}}(\phi,\bf c;\Omega)
    \end{equation*}
\end{definition}
If $\Omega\leq M$ is a submodule, then so is $\tilde\Omega=\bigcap_{j=1}^mL_j^{-1}(L_j\Omega)$, and $\operatorname{BL}_{\mathrm{sub}}^{R\text{-mod}}(\phi,\bf c;\Omega)\leq\operatorname{BL}_{\mathrm{sub}}^{R\text{-mod}}(\phi,\bf c;\tilde\Omega)$. Thus, we may always take $\Omega=\tilde\Omega$. Similarly, we have that $\operatorname{BL}^{R\text{-mod}}(\phi,\bf c)\leq \operatorname{BL}_{\mathrm{sub}}^{R\text{-mod}}(\phi,\bf c)$. Our aim will be to show the reverse inequality.

The proof of $\operatorname{BL}^{R\text{-mod}}(\phi,\bf c)\leq \operatorname{BL}_{\mathrm{sub}}^{R\text{-mod}}(\phi,\bf c)$ will be nearly identical to the corresponding result over discrete Abelian groups in~\cite{christ2013optimal}. Indeed, most of the results in the latter may be used off the shelf: an $R$-module is in particular an Abelian group, and an $R$-module homomorphism is in particular a group homomorphism. If $\Omega\leq M$ is a submodule and $f:M\to N$ is an $R$-module homomorphism, then (regarding $\Omega$ as a subgroup and $f$ as a group homomorphism) the induced group homomorphism $\bar f:M/\Omega\to N/f(\Omega)$ is in fact an $R$-modular homomorphism. Thus, essentially all of the factorization theory developed in~\cite{christ2013optimal} applies.

Rather than carefully stepping through all the theory, we constrain ourselves to considering the critical lemmas that require meaningful alteration. \emph{In each of the following lemmas, we preserve the following working assumption:} $R$ is a commutative unital ring, $M$ and $M_j$, $1\leq j\leq m$, are finite $R$-modules, and $\phi_j:M\to M_j$ are $R$-module homomorphisms.
\begin{lemma}[Critical submodules exist; c.f.~\cite{christ2013optimal}, Lemma 5.3]\label{lem:mod_crit_case}
    For $\bf A\in[1,+\infty)$, write $\cal P_\bf A=\{\bf c\in[0,1]^m:\operatorname{BL}^{R\text{-mod}}_{\mathrm{sub}}(\phi,\bf c)\leq\bf A\}$. Suppose $\bf c=(c_j)_{1\leq j\leq m}$ is an extreme point of $\cal P_\bf A$,  and that $\operatorname{BL}^{R\text{-mod}}_{\mathrm{sub}}(\phi,\bf c;\Omega)\leq\operatorname{BL}^{R\text{-mod}}_{\mathrm{sub}}(\phi,\bf c;M)$ for all $\Omega\leq M$.

    Then at least one of the following holds.
    \begin{itemize}
        \item[(a)] $c_j\in\{0,1\}$ for all but at most one $j$.
        \item[(b)] For some proper submodule $\bf 0<\Omega<M$, we have $\operatorname{BL}^{R\text{-mod}}_{\mathrm{sub}}(\phi,\bf c;\Omega)=\bf A$.
    \end{itemize}
\end{lemma}
\begin{proof}
    Suppose first that $\operatorname{BL}^{R\text{-mod}}_{\mathrm{sub}}(\phi,\bf c;M)<\bf A$. If $0<c_j<1$, then by considering the strict inequalities
    \begin{equation*}
        \log(\#\Omega)-\sum_{j'=1}^mc_{j'}\log(\#\phi_j(\Omega))\leq\log(\# M)-\sum_{j'=1}^mc_{j'}\log(\#\phi_j(M))<\log \bf A,
    \end{equation*}
    we see that $c_j$ may be freely varied along a small (two-sided) interval without contradicting the inequalities defining $\cal P_\bf A$; thus, since $\bf c$ is an extreme point, we conclude that $\bf c\in\{0,1\}^m$ if $\operatorname{BL}^{R\text{-mod}}_{\mathrm{sub}}(\phi,\bf c;M)<\bf A$.

    So, now suppose that $\operatorname{BL}^{R\text{-mod}}_{\mathrm{sub}}(\phi,\bf c;M)=\bf A$. Suppose that $\operatorname{BL}^{R\text{-mod}}_{\mathrm{sub}}(\phi,\bf c;\Omega)<\bf A$ for all proper $\bf 0<\Omega<M$. Suppose also that $1\leq i\neq j\leq m$ are such that $c_i,c_j\in(0,1)$. Suppose initially that $\#\phi_i(M)>1$, and set $\lambda=\frac{\log(\#\phi_j(M))}{\log(\#\phi_i(M))}$. For $t\in\R$, we write
    \begin{equation*}
        c_i^t=c_i+\lambda t,\quad c_j^t=c_j-t,\quad c_{j'}^t=c_{j'}\quad\forall j'\neq i,j,
    \end{equation*}
    and $\bf c^t=(c_{j'}^t)_{1\leq j'\leq m}$. Then we have
    \begin{equation*}
        \begin{split}
        \log\operatorname{BL}^{R\text{-mod}}_{\mathrm{sub}}(\phi,\bf c^t;M)&=\log(\#M)-\sum_{j'=1}^mc_{j'}^t\log(\phi_{j'}(\#M))\\
        &=\log(\#M)-\sum_{j'=1}^mc_{j'}\log(\phi_{j'}(\#M))=\log\operatorname{BL}^{R\text{-mod}}_{\mathrm{sub}}(\phi,\bf c;M),
        \end{split}
    \end{equation*}
    so $\operatorname{BL}^{R\text{-mod}}_{\mathrm{sub}}(\phi,\bf c^t;M)=\bf A$, for all $t\in\R$; on the other hand, since $\operatorname{BL}^{R\text{-mod}}_{\mathrm{sub}}(\phi,\bf c;\Omega)<\bf A$ for the finitely-many $0<\Omega<M$, we will also have $\operatorname{BL}^{R\text{-mod}}_{\mathrm{sub}}(\phi,\bf c^t;\Omega)<\bf A$ for $t$ sufficiently close to $0$. Lastly, $\operatorname{BL}^{R\text{-mod}}_{\mathrm{sub}}(\phi,\bf c^t;\bf 0)=1\leq\bf A$ irrespective of $t$. It follows that $\operatorname{BL}^{R\text{-mod}}_{\mathrm{sub}}(\phi,\bf c^t)\leq\bf A$ for all small enough $t$, which is to say that $\bf c^t\in\cal P_\bf A$ for all small $t$, which contradicts the fact that $\bf c$ is an extreme point of $\cal P_\bf A$.

    The last case is that $\phi_i(M)=1$; if $\phi_j(M)>1$, we may simply reverse the roles of $i,j$. So, we need to consider the case that $\#\phi_i(M)=1=\#\phi_j(M)$. But then the same holds for $\phi_i(\Omega),\phi_j(\Omega)$ for each $0\leq\Omega\leq M$, and hence $\operatorname{BL}^{R\text{-mod}}_{\mathrm{sub}}(\phi,\bf c;M)$ is independent of $c_i,c_j$; the same argument then contradicts that $\bf c$ is an extreme point of $\cal P_\bf A$.
\end{proof}

\begin{lemma}[Equality in an easy case]\label{lem:mod_easy}
    If $1\leq i\leq m$ is such that $c_j=0$ for all $j\neq i$, then
    \begin{equation*}
        \operatorname{BL}^{R\text{-mod}}_{\mathrm{sub}}(\phi,\bf c)=\operatorname{BL}^{R\text{-mod}}(\phi,\bf c)=(\#M)^{1-c_i}(\#\ker\phi_i)^{c_i}
    \end{equation*}
\end{lemma}
\begin{proof}
    Identical to~\cite{christ2013optimal}, Lemma 2.4.
\end{proof}

\begin{lemma}[Factorization]\label{lem:mod_factor}
    Suppose $\Omega\leq M$ is a submodule. Then, for $\left.\phi_j\right|_\Omega:\Omega\to\phi_j(\Omega)$ and $\left.\phi_j\right|_{M/\Omega}:M/\Omega\to M_j/\phi_j(\Omega)$ the induced $R$-module maps, we have
    \begin{equation*}
        \operatorname{BL}^{R\text{-mod}}(\phi,\bf c)\leq\operatorname{BL}^{R\text{-mod}}(\left.\phi\right|_\Omega,\bf c)\times\operatorname{BL}^{R\text{-mod}}(\left.\phi\right|_{M/\Omega},\bf c).
    \end{equation*}
    If $\operatorname{BL}_{\mathrm{sub}}^{R\text{-mod}}(\phi,\bf c;\Omega)=\operatorname{BL}_{\mathrm{sub}}^{R\text{-mod}}(\phi,\bf c)$, then we additionally have
    \begin{equation*}
        \operatorname{BL}_{\mathrm{sub}}^{R\text{-mod}}(\left.\phi\right|_\Omega,\bf c)\times\operatorname{BL}_{\mathrm{sub}}^{R\text{-mod}}(\left.\phi\right|_{M/\Omega},\bf   c)\leq\operatorname{BL}_{\mathrm{sub}}^{R\text{-mod}}(\phi,\bf c).
    \end{equation*}
\end{lemma}
\begin{proof}
    For the first, just apply~\cite{christ2013optimal}, Lemma 3.1. The second is a very easy translation of the proof of~\cite{christ2013optimal}, Lemma 3.2.
\end{proof}

\begin{lemma}[Equality in an interesting case]\label{lem:mod_corner}
    If $\bf c\in[0,1]^m$ is such that $c_j\in\{0,1\}$ except for at most one $j$, we have
    \begin{equation*}
        \operatorname{BL}_{\mathrm{sub}}^{R\text{-mod}}(\phi,\bf c)=\operatorname{BL}^{R\text{-mod}}(\phi,\bf c).
    \end{equation*}
\end{lemma}
\begin{proof}
    An easy translation of~\cite{christ2013optimal}, Lemma 4.1, incorporating Lemmas~\ref{lem:mod_easy},~\ref{lem:mod_factor} when appropriate.
\end{proof}

\begin{lemma}\label{lem:mod_ind_crit}
    Suppose $n\in\N$ is such that $\operatorname{BL}_{\mathrm{sub}}^{R\text{-mod}}(\tilde\phi,\bf c)=\operatorname{BL}^{R\text{-mod}}(\tilde\phi,\bf c)$, whenever $\tilde\phi$ is defined over an $R$-module $\widetilde M$ with $\#\widetilde{M}\leq n$. Suppose $\# M\leq n+1$, and $\bf 0<\Omega<M$ is a submodule for which one of the following holds:
    \begin{itemize}
        \item[(a)]  $\operatorname{BL}_{\mathrm{sub}}^{R\text{-mod}}(\left.\phi\right|_\Omega,\bf c;\Omega)=\operatorname{BL}_{\mathrm{sub}}^{R\text{-mod}}(\left.\phi\right|_\Omega,\bf c)$,
        \item[(b)] $\operatorname{BL}_{\mathrm{sub}}^{R\text{-mod}}(\phi,\bf c;\Omega)\geq 1$.
    \end{itemize}
    Then $\operatorname{BL}_{\mathrm{sub}}^{R\text{-mod}}(\phi,\bf c)=\operatorname{BL}^{R\text{-mod}}(\phi,\bf c)$.
\end{lemma}
\begin{proof}
    We take (a) first. By the inductive hypothesis, we have 
    \begin{equation*}
        \operatorname{BL}_{\mathrm{sub}}^{R\text{-mod}}(\left.\phi\right|_\Omega,\bf c)=\operatorname{BL}^{R\text{-mod}}(\left.\phi\right|_\Omega,\bf c)\quad\text{and}\quad\operatorname{BL}_{\mathrm{sub}}^{R\text{-mod}}(\left.\phi\right|_{M/\Omega},\bf c)=\operatorname{BL}^{R\text{-mod}}(\left.\phi\right|_{M/\Omega},\bf c).
    \end{equation*}
    By Lemma~\ref{lem:mod_factor},
    \begin{equation*}
        \begin{split}
        \operatorname{BL}^{R\text{-mod}}(\phi,\bf c)&\leq\operatorname{BL}^{R\text{-mod}}(\left.\phi\right|_\Omega,\bf c)\cdot\operatorname{BL}^{R\text{-mod}}(\left.\phi\right|_{M/\Omega},\bf c)\\
        &=\operatorname{BL}_{\mathrm{sub}}^{R\text{-mod}}(\left.\phi\right|_\Omega,\bf c)\cdot\operatorname{BL}_{\mathrm{sub}}^{R\text{-mod}}(\left.\phi\right|_{M/\Omega},\bf c)\leq\operatorname{BL}_{\mathrm{sub}}^{R\text{-mod}}(\phi,\bf c).
        \end{split}
    \end{equation*}
    The reverse inequality is trivial.

    Next, if (b) holds, it trivially follows that there is some $\bf 0<\Omega'<M$ for which (a) holds. Thus, the result follows once again.
\end{proof}

\begin{theorem}[Reduction to submodules]\label{thm:module_reduction}
    Suppose $R$ is a finite commutative unital ring, and $M$, $M_j$ ($1\leq j\leq m$) are finite $R$-modules; suppose also that $\phi_j:M\to M_j$ are $R$-module homomorphisms. Then, for any $\bf c\in[0,1]^m$, we have
    \begin{equation*}
        \operatorname{BL}^{R\text{-mod}}(\phi,\bf c)=\operatorname{BL}_{\mathrm{sub}}^{R\text{-mod}}(\phi,\bf c).
    \end{equation*}
\end{theorem}
\begin{proof}
    By induction on $\#M=n$. The result is trivial for $n=1$. Suppose we are done for modules of cardinality at most $n$, and consider $\#M=n+1$. Suppose $\operatorname{BL}_{\mathrm{sub}}^{R\text{-mod}}(\phi,\bf c;\Omega)<1$ for all $\bf 0<\Omega\leq M$. Since $\operatorname{BL}_{\mathrm{sub}}^{R\text{-mod}}(\phi,(0)_{1\leq j\leq m};\Omega)=1$ for all $\Omega$, there is a minimal $\theta\in(0,1)$ such that $\operatorname{BL}_{\mathrm{sub}}^{R\text{-mod}}(\phi,\theta\cdot\bf c)=1$ and, for some $\bf 0<\Omega\leq M$, we have $\operatorname{BL}_{\mathrm{sub}}^{R\text{-mod}}(\phi,\theta\cdot\bf c;\Omega)=1$. Since $\operatorname{BL}^{R\text{-mod}}(\phi,\bf c)\leq \operatorname{BL}^{R\text{-mod}}(\phi,\theta\cdot\bf c)$, it suffices to establish $\operatorname{BL}^{R\text{-mod}}(\phi,\theta\cdot\bf c)\leq\operatorname{BL}_{\mathrm{sub}}^{R\text{-mod}}(\phi,\theta\cdot\bf c)$; thus, we have reduced to the alternate case that $\operatorname{BL}_{\mathrm{sub}}^{R\text{-mod}}(\phi,\bf c)=\operatorname{BL}_{\mathrm{sub}}^{R\text{-mod}}(\phi,\bf c;\Omega)$ for some $\bf 0<\Omega\leq M$.

    In particular, we have $\operatorname{BL}_{\mathrm{sub}}^{R\text{-mod}}(\phi,\bf c;\Omega)\geq 1$. If $\Omega<M$, then by Lemma~\ref{lem:mod_ind_crit}, we are done. So we assume that $\Omega=M$. Write $\bf A=\operatorname{BL}_{\mathrm{sub}}^{R\text{-mod}}(\phi,\bf c)$. If $\bf c=(c_j)_{1\leq j\leq m}$ is an extreme point of $\cal P_\bf A$, then Lemma~\ref{lem:mod_crit_case} implies that either $c_j\in\{0,1\}$ for all but at most one $j$, or there is a submodule $\bf 0<\Omega'<M$ such that $\operatorname{BL}_{\mathrm{sub}}^{R\text{-mod}}(\phi,\bf c;\Omega')=\bf A=\operatorname{BL}_{\mathrm{sub}}^{R\text{-mod}}(\phi,\bf c;M)$. In the former case, we are done via Lemma~\ref{lem:mod_corner}. In the latter case, we are done via Lemma~\ref{lem:mod_ind_crit}. Thus, the result is true if $\bf c$ happens to be an extreme point of $\cal P_\bf A$.

    Finally, we have established $\operatorname{BL}_{\mathrm{sub}}^{R\text{-mod}}(\phi,\bf c^*)=\operatorname{BL}^{R\text{-mod}}(\phi,\bf c^*)$ whenever $\bf c^*$ is an extreme point of $\cal P_\bf A$. By interpolation, we have $\operatorname{BL}^{R\text{-mod}}(\phi,\bf c)\leq\bf A=\operatorname{BL}_{\mathrm{sub}}^{R\text{-mod}}(\phi,\bf c)$; by the reverse inequality, we are done.

    Thus, by induction on $n$, we obtain the desired equality.
\end{proof}

\subsection{Existence of extremizers, linear case}\label{subsec:lin_extremizers_exist}

    We are now in a position to prove Theorem~\ref{thm:lin_sub_extr}.

    \begin{proof}[Proof of Theorem~\ref{thm:lin_sub_extr}]

        We first handle the case that $(\bf L,\bf c)$ is simple, i.e.\ that $W_2=\K^n$. If $n=1$, then every regular subgroup evaluates to the same quantity under $\operatorname{BL}_{\mathrm{grp}}(\bf L,\bf c;-)$; so, we may assume that $n\geq 2$. Then the quantity
        \begin{equation*}
            \alpha'=\inf_{0\subsetneq V\subsetneq\K^n}\Big(-\dim V+\sum_{j=1}^mc_j\dim L_jV\Big)
        \end{equation*}
        is positive; thus $(\bf L,\bf c)$ is $\alpha$-simple, for each $0<\alpha\leq\alpha'$. So, we take arbitrary such $\alpha$.

        Now, let $U\in\cal N(\K^n)$ have diameter $1$ and radius $|\varpi|^e$. By Corollary~\ref{cor:subg_bd_by_subsp}, there are $1\leq\kappa\leq n-1$, subspaces $V_1,\ldots,V_\kappa$, and nonnegative integers $e_1,\ldots,e_\kappa$ such that $e_1+\ldots+e_\kappa=e$, such that
        \begin{equation*}
            \operatorname{BL}_{\mathrm{grp}}(\bf L,\bf c; U)\leq\prod_{\iota=1}^\kappa\operatorname{BL}_{\mathrm{grp}}\big(\bf L,\bf c;\cal O_n\cap(V_\iota+\varpi^{e_\iota}\cal O_n)\big).
        \end{equation*}
        Since $(\bf L,\bf c)$ is $\alpha$-simple, for each $\iota$ there is $\bf k_\iota=(k_{j\iota})_j$ such that
        \begin{equation*}
            \dim L_jV_\iota\geq k_{j\iota}\quad\text{for all $j$,\,\,\, and}\,\,\,\dim V_\iota\leq-\alpha+\sum_{j=1}^mc_jk_{j\iota},
        \end{equation*}
        and such that, whenever $U_\iota$ is an $(n\times \dim V_\iota)$ matrix with columns an isometric basis for $V_\iota$, it happens that $L_jU_\iota$ has a $k_j\times k_j$-minor of absolute determinant at least $\eps^{n_j}$, where $\eps>0$ is such that $\Phi_{\alpha}(V)\geq\eps$ for all $\bf 0\neq V\subsetneq\K^n$.
        Thus, by Lemma~\ref{lem:lb_plate},
        \begin{equation*}
            \operatorname{BL}_{\mathrm{grp}}\big(\bf L,\bf c;\cal O_n\cap (V_\iota+\varpi^{e_\iota}\cal O_n)\big)\leq|\varpi|^{e_\iota\big[n-\dim V-\sum_jc_j(n_j-k_{j\iota})\big]}\eps^{-n}\leq|\varpi|^{e_\iota\alpha}\eps^{-n}.
        \end{equation*}
        Multiplying these bounds together, we conclude that
        \begin{equation}\label{ineq:BL_bdd_simple}
            \operatorname{BL}_{\mathrm{grp}}(\bf L,\bf c; U)\leq|\varpi|^{e\alpha}\eps^{-n^2}.
        \end{equation}
        Thus, if $|\varpi|^{e}\leq\eps^{\alpha^{-1}n^2}$, we have
        \begin{equation*}
            \operatorname{BL}_{\mathrm{grp}}(\bf L,\bf c;U)\leq 1=\operatorname{BL}_{\mathrm{grp}}(\bf L,\bf c;\cal O_n).
        \end{equation*}
        That is,
        \begin{equation*}
            \operatorname{BL}_{\mathrm{grp}}(\bf L,\bf c)=\sup_{U\in\cal N(\K^n)}\operatorname{BL}_{\mathrm{grp}}(\bf L,\bf c;U)=\sup_{\substack{U\in\cal N(\K^n)\\\mathrm{ecc}(U)\leq\eps^{-\alpha^{-1}n^2}}}\operatorname{BL}_{\mathrm{grp}}(\bf L,\bf c;U).
        \end{equation*}
        Since there are only finitely many $U\in\cal N(\K^n)$ of eccentricity at most $C$, we conclude that there is an extremizer with the given bound.

        Now, we handle the general case. By the Lemma~\ref{lem:splitting},
        \begin{equation*}
            (\bf L,\bf c)=\bigoplus_{\iota=1}^{r-1}\big(\left.\bf L\right|_{W_{\iota+1}/W_\iota},\bf c\big);
        \end{equation*}
        here, we have written
        \begin{equation*}
            \left.L_j\right|_{W_{\iota+1}/W_\iota}:W_{\iota+1}/W_\iota\to L(W_{\iota+1})/L(W_\iota),\quad\left.\bf L\right|_{W_{\iota+1}/W_\iota}=\big(\left.L_j\right|_{W_{\iota+1}/W_\iota}\big)_{1\leq j\leq m}.
        \end{equation*}
        We have seen previously that each $\big(\left.\bf L\right|_{W_{\iota+1}/W_\iota},\bf c\big)$ satisfies the rank and scaling conditions. For $1\leq\iota\leq r-1$, we may find a regular subgroup extremizer $U_\iota\in\cal N(W_{\iota+1}/W_\iota)$ of diameter $1$ and eccentricity
        \begin{equation*}
            \mathrm{ecc}(U_\iota)\leq\max_{W_\iota\subsetneq V\subsetneq W_{\iota+1}}\Phi_\alpha(V)^{-\alpha^{-1}(\dim W_{\iota+1}-\dim W_\iota)^2},
        \end{equation*}
        for each $\alpha$ such that $\big(\left.\bf L\right|_{W_{\iota+1}/W_\iota},\bf c\big)$ is $\alpha$-simple. By Remark~\ref{rmk:dir_sum}, there is a $U\in\cal N(\K^n)$ which factors onto each $U_\iota$, with eccentricity bound
        \begin{equation*}
            \mathrm{ecc}(U)\leq\max_{\iota\in I}\min_{\substack{\alpha>0\\\big(\left.\bf L\right|_{W_{\iota+1}/W_\iota},\bf c\big)\text{ is $\alpha$-simple}}}\max_{W_\iota\subsetneq V\subsetneq W_{\iota+1}}\Phi_\alpha(V)^{-\alpha^{-1}(\dim W_{\iota+1}-\dim W_\iota)^2},
        \end{equation*}
        so that
        \begin{equation*}
            \operatorname{BL}_{\mathrm{grp}}(\bf L,\bf c;U)=\prod_{\iota=1}^{r-1}\operatorname{BL}_{\mathrm{grp}}\big(\left.\bf L\right|_{W_{\iota+1}/W_\iota},\bf c;U_\iota\big).
        \end{equation*}
        Since $\operatorname{BL}_{\mathrm{grp}}(\bf L,\bf c)=\prod_{\iota=1}^{r-1}\operatorname{BL}_{\mathrm{grp}}\big(\left.\bf L\right|_{W_{\iota+1}/W_\iota},\bf c\big)$, and each of the $U_\iota$ is an extremizer, it follows that $U$ is an extremizer:
        \begin{equation*}
            \operatorname{BL}_{\mathrm{grp}}(\bf L,\bf c;U)=\operatorname{BL}_{\mathrm{grp}}(\bf L,\bf c).
        \end{equation*}

    \end{proof}

    From this, we may show that there is a similar extremizer result in the case of functional Brascamp--Lieb; this is the content of Theorem~\ref{thm:main_extr_functional}. We require our earlier ``module Brascamp--Lieb'' result, Theorem~\ref{thm:module_reduction}. We restate the result to be shown for the convenience of the reader.

    \begin{corollary}
        Let $(\bf L,\bf c)$ be any Brascamp--Lieb datum with $\bigcap_{j=1}^m\ker L_j=\{0\}$. Then $\operatorname{BL}(\bf L,\bf c)<+\infty$ if and only if there is some tuple $\bf f=(f_j)_{1\leq j\leq m}$ of nonzero Schwartz--Bruhat functions such that $\operatorname{BL}(\bf L,\bf c)=\operatorname{BL}(\bf L,\bf c;\bf f)$.

        Moreover, in this case, there is a regular subgroup $U\in\cal N(\K^n)$ such that the extremizer $f_j$ may be taken to be $1_{L_j(G)}$. The eccentricity of $U$ may be bounded as in Theorem~\ref{thm:lin_sub_extr}.
    \end{corollary}

    \begin{proof}
        The ``only if'' part is trivial, so we focus on the ``if'' and assume $\operatorname{BL}(\bf L,\bf c)<+\infty$.
    
        For any regular subgroup $U\subseteq\K^n$, if $U=\bigcap_{j=1}^mL_j^{-1}(L_j(U))$,
        \begin{equation*}
            \operatorname{BL}_{\mathrm{grp}}(\bf L,\bf c;U)=\operatorname{BL}\big(\bf L,\bf c;(1_{L_j(U)})_j\big)\leq\operatorname{BL}(\bf L,\bf c).
        \end{equation*}
        Thus, $\operatorname{BL}_{\mathrm{grp}}(\bf L,\bf c)\leq\operatorname{BL}(\bf L,\bf c)$. By Theorem~\ref{thm:lin_sub_extr}, we may thus find an extremizer: there is a $U\in\cal N(\K^n)$ such that $\operatorname{BL}_{\mathrm{grp}}(\bf L,\bf c)=\operatorname{BL}_{\mathrm{grp}}(\bf L,\bf c;U)$. We claim that $(1_{L_j(U)})_j$ is an extremizer to $\operatorname{BL}(\bf L,\bf c)$.
    
        Let $\bf g=(g_j)_j$ be a tuple of nonnegative nonzero Schwartz--Bruhat functions. Let $\gamma,h\in\K^\times$ be such that each $g_j$ is supported in $\gamma\cal O_{n_j}$, and is constant on cosets of $h\cal O_{n_j}$. Write $K=\bigcap_{j=1}^mL_j^{-1}(h\cal O_{n_j})$ and $U=\bigcap_{j=1}^mL_j^{-1}(\gamma\cal O_{n_j})$. Since $\operatorname{BL}_{\mathrm{grp}}(\bf L,\bf c)$ is finite, $K$ is compact and open; further, $K$ is an $\cal O_1$-module. Then the $L_j$ descend to $\phi_j$,
        \begin{equation*}
            \phi_j:U/K=M\to M_j=\gamma\cal O_{n_j}/h\cal O_{n_j},
        \end{equation*}
        and the $g_j$ descend to $f_j:M_j\to\R_{\geq 0}$ nonzero. Then
        \begin{equation*}
            \operatorname{BL}(\bf L,\bf c;\bf g)=\frac{\mu_n(K)\sum_{a\in M}f_j^{c_j}(\phi_j(a))}{\prod_{j=1}^m\big(|h|^{n_j}\sum_{b\in M_j}f_j(b)\big)^{c_j}}=\mu_n(K)\cdot|h|^{-\sum_jc_jn_j}\operatorname{BL}^{\cal O_1\text{-mod}}(\phi,\bf c;\bf f).
        \end{equation*}
        the latter is the modular Brascamp--Lieb functional, in the sense of Definition~\ref{def:rmod_BL}. By Theorem~\ref{thm:module_reduction},
        \begin{equation*}
            \operatorname{BL}(\phi,\bf c;\bf f)\leq\frac{\#\Omega}{\prod_j(\#\phi_j(\Omega))^{c_j}}
        \end{equation*}
        where $\Omega$ is a suitable $\cal O_1$-submodule of $M$. On the other hand,
        \begin{equation*}
            \frac{\#\Omega}{\prod_j(\#\phi_j(\Omega))^{c_j}}=\frac{|h|^{\sum_jc_jn_j}}{\mu_n(K)}\frac{\int\prod_{j=1}^m1_{\phi_j(\Omega)+h\cal O_{n_j}}^{c_j}\circ L_j}{\prod_{j=1}^m\big(\int 1_{\phi_j(\Omega)+h\cal O_{j}}\big)^{c_j}}.
        \end{equation*}
        In particular, $\operatorname{BL}(\bf L,\bf c;\bf g)\leq\operatorname{BL}(\bf L,\bf c;(1_{\phi_j(\Omega)+h\cal O_{n_j}})_j)=\operatorname{BL}_{\mathrm{grp}}(\bf L,\bf c;\Omega+K)$.

        Thus, $\operatorname{BL}(\bf L,\bf c;\bf g)\leq\operatorname{BL}_{\mathrm{grp}}(\bf L,\bf c)=\operatorname{BL}_{\mathrm{grp}}(\bf L,\bf c;U)$, where $U$ is the extremizer from earlier. In particular,
        \begin{equation*}
            \operatorname{BL}(\bf L,\bf c;\bf g)\leq \operatorname{BL}_{\mathrm{grp}}(\bf L,\bf c;U)=\operatorname{BL}\big(\bf L,\bf c;(1_{L_j(U)})_j\big).
        \end{equation*}
        That is, $(1_{L_j(U)})_j$ dominates any other Schwartz--Bruhat tuple under the Brascamp--Lieb functional.
    \end{proof}

    \begin{remark}
        In particular, $\operatorname{BL}(\bf L,\bf c)$ is extremized by a ``Guassian;'' i.e.\ each function is an indicator of a regular subgroup.
    \end{remark} 

    We conclude this subsection with a small ``rearrangement'' lemma. Unfortunately, this is not powerful enough to imply a non-Archimedean variant of the functional rearrangement inequality from~\cite{brascamp1974general}; the critical issue is that $\K^n$ does not possess ``convex'' sets of arbitrary measure.

    \begin{lemma}[Rearrangement inequality]\label{lem:rearrange}
        Suppose $(\bf L,\bf c)$ is a Brascamp--Lieb datum. Let $U\in\cal N(\K^n)$ and $x_j\in\K^{n_j}$, for each $j$. Then
        \begin{equation*}
            \operatorname{BL}(\bf L,\bf c;\{1_{L_jU+x_j}\}_{1\leq j\leq m})\leq\operatorname{BL}(\bf L,\bf c; \{1_{L_jU}\}_{1\leq j\leq m}).
        \end{equation*}
    \end{lemma}
    \begin{proof}
        We may freely take $U=\bigcap_{j=1}^mL_j^{-1}(L_j(U))$. Let $V=\{y\in\K^n:L_j(y)\in L_jU+x_j,\,\,\forall j\}$. Then, for $y,z\in V$, we have $L_j(y-z)\in L_jU-L_jU=L_jU$. Thus, $V-V\subseteq U$, so $\mu_n(V)\leq\mu_n(V-V)\leq\mu_n(U)$. Thus,
        \begin{equation*}
            \begin{split}
            \operatorname{BL}(\bf L,\bf c;\{1_{L_jU+x_j}\}_{1\leq j\leq m})&=\frac{\mu_n(V)}{\prod_{j=1}^m\mu_{n_j}(L_jU+x_j)^{c_j}}\\
            &\leq\frac{\mu_n(U)}{\prod_{j=1}^m\mu_{n_j}(L_jU)^{c_j}}=\operatorname{BL}(\bf L,\bf c;\{1_{L_jU}\}_{1\leq j\leq m}).
            \end{split}
        \end{equation*}
    \end{proof}

\section{Nonlinear Brascamp--Lieb: general finiteness constrains}\label{sec:nonlinear_constraints}

In this section, we prove general necessary constraints on maps $\bf B$ to obtain finite Brascamp--Lieb constants. Throughout this section, we take $\bf B=(B_j)_{1\leq j\leq m}$ be a $C^1$ tuple defined on a nonempty open set $U\subseteq\K^n$,  such that $B_j$ takes values in $\K^{n_j}$.

We will show that $\bf B$ must be an immersion, and each $B_j$ must be a submersion: that is, $d\bf B(x)$ must have small kernel and $d B_j(x)$ must be nearly surjective (or $c_j=0$), for generic $x$.

\begin{proposition}[Submersion condition]
    Let $1\leq j\leq m$ be arbitrary with $c_j>0$. Suppose that $B_j$ is a $C^k$ map, $k\geq 1$. Suppose that $r\leq n_j-\frac{n-n_j}{k-1}$, and that $x\in U$ is such that, for some $\delta\in\K^\times$, we have $\mathrm{rank}(dB_j(y))\leq r$ for all $y\in x+\delta\cal O_n$. Then $\operatorname{BL}(\bf B,\bf c)=+\infty$.
\end{proposition}

\begin{proof}
    By Theorem~\ref{thm:sard} below, we have that $B_j(x+\delta\cal O_n)$ is a nullset in $\K^{n_j}$. Since $\operatorname{BL}(\bf B,\bf c)$ is decreasing under restrictions, we may freely replace $\bf B$ with $\left.\bf B\right|_{x+\delta\cal O_n}$. We may also, for simplicity, assume that each $B_j$ is $1$-Lipschitz on $x+\delta\cal O_n$; affine symmetries certainly accomplish this without affecting the conclusion.

    Now, for $j'\neq j$, write simply $f_{j'}=1_{B_{j'}(x)+\cal O_{n_j}}$. For $j$, let $\eps\in\delta\cal O_1$, and choose $X_\eps$ to be an open neighborhood of $B_j(x+\delta\cal O_n)$ with $\mu_{n_j}(X_\eps)\leq|\eps|$. Then, writing $f_{j}=1_{X_\eps}$, we obtain
    \begin{equation*}
        \int\prod_{j'=1}^mf_j(B_j(y))=|\delta|^n,
    \end{equation*}
    whereas
    \begin{equation*}
        \|f_j\|_{L^{1/c_j}\big(\K^{n_j}\big)}=\mu_{n_j}(X_\eps)^{c_j},\quad\|f_{j'}\|_{L^{1/c_{j'}}\big(\K^{n_{j'}}\big)}=1,
    \end{equation*}
    so that
    \begin{equation*}
        \operatorname{BL}\big(\bf B,\bf c\big)\geq|\delta|^n|\eps|^{-c_j n_j}.
    \end{equation*}
    Since $c_j>0$ and $\eps$ is only constrained by $|\eps|\leq|\delta|$, we may send $\eps\to 0$ and conclude that $\operatorname{BL}(\bf B,\bf c)=+\infty$, as desired.
\end{proof}

\begin{remark}
    In particular, if $c_j>0$ and $B_j$ is $C^{n-n_j+1}$, then $\operatorname{BL}(\bf B,\bf c)<+\infty$ forces $dB_j$ to be generically full rank.
\end{remark}

\begin{proposition}[Immersion condition]\label{prop:ker_cond}
    Suppose that there is some $x\in U$ such that
    \begin{equation*}
        n-\sum_{j=1}^m c_jn_j<\dim\ker d\bf B(x).
    \end{equation*}
    Then $\operatorname{BL}(\bf B,\bf c)=+\infty$.
\end{proposition}

\begin{proof}
    Abbreviate $V=\ker d\bf B(x)$; thus, for each $1\leq j\leq k$ and $z\in V$, we have $dB_j(x).z=0$. It will suffice to assume that $U$ is a metric ball; i.e.\ $U=x_0+\gamma\cal O_n$. Recall that $\bf B$ is $C^1$, i.e.\ that the functions
    \begin{equation*}
        \Psi_{ik}(y,z,\xi)=\frac{\bf B(\xi+y\bf e_i)_k-\bf B(\xi+z\bf e_i)_k}{y-z},\quad \xi\in U,\,\,y,z\in\gamma\cal O_1
    \end{equation*}
    extends to be continuous on $U\times \gamma\cal O_1\times\gamma\cal O_1$; in particular, we may find some $C>0$ such that
    \begin{equation*}
        |\Psi_{ik}(y,z,\xi)|\leq C,\quad\forall i,k,\quad\forall \xi\in U,\,\,y,z\in\gamma\cal O_1.
    \end{equation*}
    It follows immediately that $\bf B$ satisfies the Lipschitz bound
    \begin{equation*}
        \|\bf B(y)-\bf B(z)\|\leq C\|y-z\|,\quad\forall y,z\in U.
    \end{equation*}
    
    Next, given $\kappa\in\K^\times$ small, there exists $\eta\in\K^\times$ with $|\eta|\leq 1$ such that
    \begin{equation*}
        z\in\eta\cal O_n\implies B_j(x+z)-B_j(x)-dB_j(x).z\in\kappa\eta\cal O_{n_j}.
    \end{equation*}
    For $h\in\K^\times$ small (always $|h|\leq 1$), write
    \begin{equation*}
        f_j(y)=1_{h\cal O_{n_j}}(y-B_j(x)).
    \end{equation*}
    We choose for now $h=\kappa\eta$, where $\kappa$ is small and $\eta$ is obtained as above. From the Lipschitz bound on $\bf B$, if $\rho\in\K^\times$ is such that $|\rho|=\min\big(|\gamma|,|h|C^{-1}\big)$, it follows that
    \begin{equation*}
        \forall y\in U,\,\,z\in\rho\cal O_n:\quad\|\bf B(y+z)-\bf B(y)\|\leq|h|.
    \end{equation*}
    We will assume that $\kappa$ is so small that $|\rho|=|h|C^{-1}$.
    
    Thus, for $z\in \eta\cal O_n\cap V$ and $z'\in\rho\cal O_n$,
    \begin{equation*}
        B_j(x+z+z')\in B_j(x+z)+h\cal O_{n_j}
    \end{equation*}
    and
    \begin{equation*}
        B_j(x+z)\in B_j(x)+dB_j(x).z+\kappa\eta\cal O_{n_j}=B_j(x)+\kappa\eta\cal O_{n_j.}
    \end{equation*}
    In particular,
    \begin{equation*}
        \|\bf B(x+z+z')-\bf B(x)\|\leq |h|,
    \end{equation*}
    for each $z\in \eta\cal O_n\cap V$ and $z'\in\rho\cal O_n$. Thus,
    \begin{equation*}
        \prod_{j=1}^m(f_j^{c_j}\circ B_j)(x+z+z')\geq 1_{\eta\cal O_n\cap V}(z)1_{\rho\cal O_n}(z').
    \end{equation*}
    In particular, if $W\subseteq\K^n$ is complementary to $V$,
    \begin{equation*}
        \int_{V}\int_W\prod_{j=1}^m(f_j^{c_j}\circ B_j)(x+z+z')dz'dz\geq|\eta|^{\dim V}|\rho|^{n-\dim V},
    \end{equation*}
    and
    \begin{equation*}
        \prod_{j=1}^m\left(\int f_j\right)^{c_j}=|h|^{\sum_{j=1}^mn_jc_j}.
    \end{equation*}
    Since $h=\kappa\eta$, and $|\rho|\leq C^{-1}|h|$
    \begin{equation*}
        \operatorname{BL}(\bf B,\bf c;\bf f)\gtrsim|h|^{n-\dim V-\sum_jn_jc_j}|\eta|^{\dim V}\geq|\eta|^{n-\sum_jn_jc_j}|\kappa|^{n-\dim V-\sum_jn_jc_j}.
    \end{equation*}
    Since $n<\dim V+\sum_jc_jn_j$, and since $\kappa$ was permitted to be arbitrarily small, we obtain
    \begin{equation*}
        \operatorname{BL}(\bf B,\bf c)=+\infty=\operatorname{BL}(d\bf B(x),\bf c),
    \end{equation*}
    as claimed.
\end{proof}

\begin{remark}
    Thus, if $n\leq\sum_{j=1}^mc_jn_j$, it follows that $\operatorname{BL}(\bf B,\bf c)<+\infty$ only if $d\bf B(x)$ is injective, for every $x$. In particular, if $\bf c$ lies on the scaling line $n=\sum_jc_jn_j$, the same requirement holds. More generally, if $\operatorname{BL}(\bf B,\bf c)<+\infty$, it must hold that $\dim\ker d\bf B(x)\leq n-\sum_jc_jn_j$, for every $x\in U$.
\end{remark}

\section{Nonlinear Brascamp--Lieb: reduction to linear}\label{sec:nonlin_to_lin}

In this section, we will prove upper- and lower-bounds relating nonlinear Brascamp--Lieb inequalities to their worst linearizations, at least locally. The result will be proofs of Theorems~\ref{thm:main_K_NL} and~\ref{thm:main_L_by_NL}.

One of the useful features of this analysis is that we implicitly obtain a quantitative bound on near-extremizers. We introduce the following quantity to track the distance of a given tuple of test functions $\bf f$ from an extremizer.
\begin{definition}\label{def:inv_thry_Z}
    Suppose that $\operatorname{BL}(\bf L,\bf c)<+\infty$. Let $M\in\cal N(\K^n)$ be an extremizer to $\operatorname{BL}_{\mathrm{grp}}(\bf L,\bf c)$. For $\lambda\in\varpi^\Z$ and a tuple of nonnegative Schwartz--Bruhat functions $\bf f$, we write
    \begin{equation*}
        \cal Z(\bf f;\lambda)=\cal Z_{\bf L}(\bf f;\lambda)=\frac{\operatorname{BL}(\bf L,\bf c;\big(f_j*1_{\lambda (L_jM)})_{1\leq j\leq m}\big)}{\operatorname{BL}(\bf L,\bf c)}.
    \end{equation*}
\end{definition}

We restate Theorem~\ref{thm:main_K_NL} for convenience.

\begin{theorem}[Linearity is locally the worst case]\label{thm:nonlin_local_bdd}
    Let $U\subseteq\K^n$ be a nonempty compact open set, and $\bf B=(B_j)_{1\leq j\leq m}$ be a tuple of $C^1$ maps $B_j:U\to\K^{n_j}$, such that $\|\bf B\|_{C^1(U)}<+\infty$, and $\bf c=(c_j)_{1\leq j\leq m}\in[0,1]^m$. Suppose that each $B_j$ is a submersion on $U$, and that
    \begin{equation*}
        \sup_{u\in U}\operatorname{BL}(d\bf B(u),\bf c)<+\infty.
    \end{equation*}
    Then there is a constant $\delta\in\K^\times$, depending on $\bf B$ and $\bf c$, for which the following holds: for each $\epsilon\in\delta\cal O_1$,
    \begin{equation*}
        \operatorname{BL}(\left.\bf B\right|_{u+\epsilon\cal O_n},\bf c)\leq\sup_{x\in u+\eps\cal O_n}\operatorname{BL}(d\bf B(x),\bf c),\quad\text{for all $u\in U$ with $u+\epsilon\cal O_n\subseteq U$}.
    \end{equation*}
    Moreover, $\delta$ may be chosen as follows. There is a regular subgroup $M_u\in\cal N(\K^n)$ extremizing $\operatorname{BL}_{\mathrm{grp}}(d\bf B(u),\bf c)$, say with minimal eccentricity $\mathrm{ecc}(M_u)=E_u$. Let $E<+\infty$ be a uniform upper bound for all the $E_u$. Let also
    \begin{equation*}
        T=\max\big(1,\sup\big\{\|dB_j(u)\|:1\leq j\leq m,\,\, u\in U\big\}\big),
    \end{equation*}
    \begin{equation*}
        r=\min_{1\leq j\leq m}\inf_{u\in U}\mathrm{rad}\big(dB_j(u).(\cal O_n)\big),
    \end{equation*}
    \begin{equation*}
        c=\inf_{\|\bf v\|=1}\inf_{u\in U}\|d\bf B(u).\bf v\|.
    \end{equation*}
    Then, $\delta$ is chosen to arrange that
    \begin{equation*}
        x,y\in U\cap(u+\delta\cal O_n)\implies \|\bf B(x)-\bf B(u)-d\bf B(y).(x-y)\|\leq|\varpi|E^{-1}T^{-1}rc\|x-y\|.
    \end{equation*}
\end{theorem}

In subsection~\ref{subsec:ball_ineq}, we prove a nonlinear Ball's inequality over $\K$. In subsection~\ref{subsec:nonlin_to_lin}, we use Ball's inequality to prove Theorem~\ref{thm:nonlin_local_bdd}. In subsection~\ref{subsec:converse}, we prove Theorem~\ref{thm:main_L_by_NL}.

\subsection{Ball's inequality}\label{subsec:ball_ineq}

The core tool for controlling nonlinear Brascamp--Lieb functionals is the following. The technique originated in~\cite{ball2006volumes}, and was used critically in~\cite{bennett2020nonlinear}.

\begin{lemma}[Nonlinear Ball's inequality]\label{lem:nonlin_ball_ineq}
    Let $U\subseteq\K^n$ be a nonempty compact open set, and $\bf B=(B_j)_{1\leq j\leq m}$ be a tuple of $C^1$ submersions $B_j:U\to\K^{n_j}$, such that $\|\bf B\|_{C^1(U)}<+\infty$, and $\bf c=(c_j)_{1\leq j\leq m}\in[0,1]^m$ satisfy~\eqref{eq:scaling}. Fix $u\in U$, and suppose that $\gamma,\tau\in\K^\times$ are such that each $dB_j(u)$ has maximum entry at most $|\tau|$, and has an invertible $n_j\times n_j$ submatrix $A$ such that $\|A^{-1}\|\leq|\gamma|^{-1}$.  Suppose that $M\in\cal N(\K^n)$ is an extremizer of $\operatorname{BL}_{\mathrm{grp}}(d\bf B(u),\bf c)$ with $\mathrm{rad}(M)=|\rho|\leq 1=\mathrm{diam}(M)$. Write $M_j=dB_j(u).(M)$ and $\bf M=(M_j)_{1\leq j\leq m}$.

    Suppose $\delta\in\K^\times$ is such that $u+\delta\cal O_n\subseteq U$ and
    \begin{equation*}
        x\in u+\delta\cal O_n\quad\implies\quad\|B_j(x)-B_j(u)-dB_j(u).(x-u)\|\leq|\varpi\rho\gamma|\|x-u\|.
    \end{equation*}
    Then, for any $\bf f=(f_j)_{1\leq j\leq m}$ tuple of nonnegative nonzero Schwartz--Bruhat functions on $\K^{n_j}$, we have
    \begin{equation*}
        \operatorname{BL}\big(\left.\bf B\right|_{u+\delta\cal O_n},\bf c;\bf f\big)\leq\cal Z_{d\bf B(u)}(\tau_{\bf B(u)}\bf f;\varpi\delta)\sup_{x\in u+\delta\cal O_n}\operatorname{BL}\big(\left.\bf B\right|_{x+\varpi\delta\cal O_n},\bf c;\bf f1_{\bf B(u)+d\bf B(u).(x-u)+\varpi\delta\bf M}\big),
    \end{equation*}
    where $\tau_af(x)=f(x+a)$ is translation by $a$.
\end{lemma}

\begin{proof}
    Note that $\mathrm{rad}(M_j)\geq|\gamma\rho|$ and $\mathrm{diam}(M_j)\leq|\tau|$. Given $\bf f=(f_1,\ldots,f_m)$ Schwartz--Bruhat nonnegative nonzero,
    \begin{equation*}
        \begin{split}
            \int_{u+\delta\cal O_n}\prod_{j=1}^m(f_j^{c_j}\circ B_j)(x)dx&=\frac{1}{\operatorname{BL}(d\bf B(u),\bf c)}\int_{u+\delta\cal O_n}\prod_{j=1}^m(f_j^{c_j}\circ B_j)(x)\\
            &\hspace{3em}\times\int_{x+\varpi\delta\cal O_n}\prod_{j=1}^m\mu_{n_j}(\varpi\delta M_j)^{-c_j}(1_{\varpi\delta M_j}\circ dB_j(u))(y-x)dy\,dx\\
            &\hspace{-3em}=\frac{1}{\operatorname{BL}(d\bf B(u),\bf c)}\int_{u+\delta\cal O_n}\int_{x+\varpi\delta\cal O_n}\\
            &\prod_{j=1}^m\mu_{n_j}(\varpi\delta M_j)^{-c_j}(f_j^{c_j}\circ B_j)(x)(1_{\varpi\delta M_j}\circ dB_j(u))(y-x)dy\,dx.
        \end{split}
    \end{equation*}
    Writing
    \begin{equation*}
        dB_j(u).(y-x)=B_j(u)+dB_j(u).(y-u)-B_j(x)+\Big(B_j(x)-B_j(u)-dB_j(u).(x-u)\Big),
    \end{equation*}
    and noticing that the bracketed expression belongs to $\varpi\rho\gamma\delta\cal O_{n_j}$, together with the fact that each $1_{\varpi\delta M_j}$ is $\varpi\rho\gamma\delta\cal O_{n_j}$-invariant, we conclude that
    \begin{equation*}
        (1_{\varpi\delta M_j}\circ dB_j(u))(y-x)=1_{\varpi\delta M_j}\big(B_j(u)+dB_j(u).(y-u)-B_j(x)\big).
    \end{equation*}
    Define then
    \begin{equation*}
        h_{y,j}(z)=\mu_{n_j}(\varpi\delta M_j)^{-1}f_j(z)1_{\varpi\delta M_j}(B_j(u)+dB_j(u).(y-u)-z),
    \end{equation*}
    and we obtain
    \begin{equation*}
        \int_{u+\delta\cal O_n}\prod_{j=1}^m(f_j^{c_j}\circ B_j)(x)dx=\frac{1}{\operatorname{BL}(d\bf B(u),\bf c)}\int_{u+\delta\cal O_n}\int_{x+\varpi\delta\cal O_n}\prod_{j=1}^m (h_{y,j}^{c_j}\circ B_j)(x) dy\,dx.
    \end{equation*}
    Thus,
    \begin{equation*}
        \begin{split}
        \int_{u+\delta\cal O_n}\prod_{j=1}^m(f_j^{c_j}\circ B_j)(x)dx&\leq\frac{1}{\operatorname{BL}(d\bf B(u),\bf c)}\int_{u+\delta\cal O_n}\operatorname{BL}(\left.\bf B\right|_{y+\varpi\delta\cal O_n},\bf c;\bf h_{y})\\
        &\hspace{4em}\times\prod_{j=1}^m\big[\mu_{n_j}(\varpi\delta M_j)^{-1}f_j*1_{\varpi\delta M_j}\big]^{c_j}\big( B_j(u)+dB_j(u).(y-u)\big)dy\\
        \leq\frac{1}{\operatorname{BL}(d\bf B(u),\bf c)}&\sup_{y\in u+\delta\cal O_n}\operatorname{BL}(\left.\bf B\right|_{y+\varpi\delta\cal O_n},\bf c;\bf h_{y})\operatorname{BL}(d\bf B(u),\bf c;\tau_{\bf B(u)}[\bf f*1_{\varpi\delta\bf M}])\prod_{j=1}^m\left[\int f_j\right]^{c_j}.
        \end{split}
    \end{equation*}
    Thus, we have obtained
    \begin{equation*}
    \begin{split}
        \operatorname{BL}(\left.\bf B\right|_{u+\delta\cal O_n},\bf c;\bf f)&\leq\frac{\operatorname{BL}(d\bf B(u),\bf c;(\tau_{\bf B(u)}\bf f)*1_{\varpi\delta\bf M})}{\operatorname{BL}(d\bf B(u),\bf c)}\\&\quad\quad\times\sup_{x\in u+\delta\cal O_n}\operatorname{BL}\big(\left.\bf B\right|_{x+\varpi\delta\cal O_n},\bf c;\bf f1_{\bf B(u)+d\bf B(u).(x-u)+\varpi\delta\bf M}\big),
    \end{split}
    \end{equation*}
    and so we conclude the desired
    \begin{equation*}
        \operatorname{BL}(\left.\bf B\right|_{u+\delta\cal O_n},\bf c;\bf f)\leq\cal Z_{d\bf B(u)}(\tau_{\bf B(u)}\bf f;\varpi\delta)\sup_{x\in u+\delta\cal O_n}\operatorname{BL}\big(\left.\bf B\right|_{x+\varpi\delta\cal O_n},\bf c;\bf f1_{\bf B(u)+d\bf B(u).(x-u)+\varpi\delta\bf M}\big).
    \end{equation*}
\end{proof}

\subsection{Reducing nonlinear to linear}\label{subsec:nonlin_to_lin}

Repeatedly applying Ball's inequality, we obtain the following estimate. One should compare to heat-flow methods over the reals.

\begin{corollary}[Quantitative nonlinear-to-linear reduction]\label{cor:nonlin_quant_ineq}
    Let $U\subseteq\K^n$ be a nonempty compact open set, $\bf B=(B_j)_{1\leq j\leq m}$ be a tuple of $C^1$ submersions $B_j:U\to\K^{n_j}$, such that $\|\bf B\|_{C^1(U)}<+\infty$, and $\bf c=(c_j)_{1\leq j\leq m}\in[0,1]^m$ satisfy~\eqref{eq:scaling}. Fix $u\in U$. Suppose that $M\in\cal N(\K^n)$ is an extremizer of $\operatorname{BL}_{\mathrm{grp}}(d\bf B(u),\bf c)$ with $\mathrm{rad}(M)=|\rho|\leq 1=\mathrm{diam}(M)$; write as before $M_j=dB_j(u).(M)$ and $\bf M=(M_j)_{1\leq j\leq m}$. Suppose also that $\gamma,\tau\in\K^\times$ are such that each $dB_j(u)$ has maximum entry at most $|\tau|$, and has an invertible $n_j\times n_j$ submatrix $A$ such that $\|A^{-1}\|\leq|\gamma|^{-1}$. We assume as well that $|\tau|\geq 1$ and $|\gamma|\leq 1$. 
    
    Then there is a $\delta=\delta(\bf B,\bf c)\in\cal O_1\setminus\{0\}$ such that the following holds. Let $\bf f=(f_j)_{1\leq j\leq m}$ be a tuple of nonnegative nonzero Schwartz--Bruhat functions on $\K^{n_j}$. Suppose that each $f_j$ is constant on cosets of $\varpi^{K+1}\delta M_j$, where $K\in\N$. Then, for each $\eps\in\delta\cal O_1$ such that $u+\eps\cal O_n\subseteq U$, we have
    \begin{equation*}
        \begin{split}
        \operatorname{BL}\big(\left.\bf B\right|_{u+\eps\cal O_n},\bf c;\bf f\big)&\leq\cal Z_{d\bf B(u)}(\tau_{\bf B(u)}\bf f;\varpi\eps)\\
        &\times\sup_{\substack{x_1,\ldots,x_K\in u+\eps\cal O_n\\x_{\iota+1}\in x_{\iota}+\varpi^\iota\eps\cal O_n}}\prod_{\iota=1}^K\cal Z_{d\bf B(u)}\big((\tau_{\bf B(u)}\bf f)1_{d\bf B(u).(x_\iota-u)+\varpi^\iota\eps\bf M};\varpi^{\iota+1}\eps\big)\\
        &\times\sup_{x\in u+\eps\cal O_n}\operatorname{BL}\big(d\bf B(x),\bf c\big).
        \end{split}
    \end{equation*}
    
\end{corollary}

\begin{proof}
    We may freely assume that $\sup_{x\in U}\operatorname{BL}(d\bf B(x),\bf c)<+\infty$. In particular, we may assume that each $d\bf B(x)$ has trivial kernel. By compactness, there is an $h\in\cal O_1\setminus\{0\}$ such that
    \begin{equation*}
        \inf_{\substack{x\in U\\\|\bf v\|=1}}\|d\bf B(x).\bf v\|\geq|h|.
    \end{equation*}
    Take first $\delta'$ as in Lemma~\ref{lem:nonlin_ball_ineq}. Let $\kappa\in\cal O_1\setminus\{0\}$ so that,
    \begin{equation*}
        x,y\in u+\delta'\cal O_n\,\,\,\text{and}\,\,\,x-y\in\kappa\delta'\cal O_n\,\,\,\implies\,\,\, \|\bf B(y)-\bf B(x)-d\bf B(x).(y-x)\|\leq|\varpi \rho\gamma h\tau^{-1}|\|x-y\|.
    \end{equation*}
    In particular, if $y\neq x$, we have
    \begin{equation*}
        \|\bf B(y)-\bf B(x)-d\bf B(x).(y-x)\|<\|d\bf B(x).(y-x)\|;
    \end{equation*}
    thus,
    \begin{equation*}
        \|\bf B(y)-\bf B(x)\|=\|d\bf B(x).(y-x)\|.
    \end{equation*}
    Choose $\delta=\kappa\delta'$. The upshot is that we may freely assume that, over $u+\delta\cal O_n$, we have
    \begin{equation*}
        |h|\|x-y\|\leq\|\bf B(y)-\bf B(x)\|\leq\|\bf B\|_{C^1}\|x-y\|.
    \end{equation*}
    
    We proceed to proving the inequality. For simplicity, we assume that $\eps=\delta$. Iterate Lemma~\ref{lem:nonlin_ball_ineq}. After $K+1$ applications, we are left with
    \begin{equation*}
        \begin{split}
        \operatorname{BL}\big(\left.\bf B\right|_{u+\delta\cal O_n},\bf c;\bf f\big)&\leq\cal Z_{d\bf B(u)}(\tau_{\bf B(u)}\bf f;\varpi\delta)\\
        &\times\sup_{\substack{x_1,\ldots,x_K\in u+\delta\cal O_n\\x_{\iota+1}\in x_{\iota}+\varpi^\iota\delta\cal O_n}}\prod_{\iota=1}^K\cal Z_{d\bf B(u)}\big((\tau_{\bf B(u)}\bf f)1_{d\bf B(u).(x_\iota-u)+\varpi^\iota\delta\bf M},\varpi^{\iota+1}\delta\big)\\
        &\times\operatorname{BL}\big(\left.\bf B\right|_{u+\delta\cal O_n},\bf c;1_{\bf B(u)+d\bf B(u).(x_K-u)+\varpi^{K+1}\delta\bf M}\big).
        \end{split}
    \end{equation*}
    Any further iterations of Lemma~\ref{lem:nonlin_ball_ineq} will have $\cal Z=1$, so we halt at this point. If $x,y\in u+\delta\cal O_n$ are such that
    \begin{equation*}
        \bf B(x),\bf B(y)\in\bf B(u)+d\bf B(u).(x_K-u)+\varpi^{K+1}\delta\bf M,
    \end{equation*}
    then in particular
    \begin{equation*}
        \|\bf B(x)-\bf B(y)\|\leq|\tau\varpi^{K+1}\delta|,
    \end{equation*}
    and hence
    \begin{equation*}
        \|x-y\|\leq|h^{-1}\tau\varpi^{K+1}\delta|.
    \end{equation*}
    Thus,
    \begin{equation*}
        \|\bf B(y)-\bf B(x)-d\bf B(x).(y-x)\|\leq|\varpi^{K+2}\gamma\rho\delta|.
    \end{equation*}
    Since $1_{B_j(u)+dB_j(u).(x_K-u)+\varpi^{K+1}\delta M_j}$ is invariant under translations by elements of $\varpi^{K+1}\gamma\rho\delta\cal O_{n_j}$, we conclude that
    \begin{equation*}
        \begin{split}
            &1_{B_j(u)+dB_j(u).(x_K-u)+\varpi^{K+1}\delta M_j}(B_j(y))\\
            &=1_{B_j(u)+dB_j(u).(x_K-u)+\varpi^{K+1}\delta M_j}(B_j(x)+dB_j(x).(y-x)).
        \end{split}
    \end{equation*}
    Regarding $x$ as fixed, we obtain
    \begin{equation*}
        \begin{split}
        \int_{u+\delta\cal O_n}&\prod_{j=1}^m
        1_{B_j(u)+dB_j(u).(x_K-u)+\varpi^{K+1}\delta M_j}(B_j(y))dy\\
        &=\int_{u+\delta\cal O_n}\prod_{j=1}^m
        1_{B_j(u)-B_j(x)+dB_j(u).(x_K-u)+dB_j(x).x+\varpi^{K+1}\delta M_j}(dB_j(x).y)dy\\
        &=\operatorname{BL}\big(d\bf B(x),\bf c;1_{\bf B(u)-\bf B(x)+d\bf B(u).(x_k-u)+d\bf B(x).x+\varpi^{K+1}\delta\bf M}\big)\prod_{j=1}^m\mu_{n_j}(\varpi^{K+1}\delta M_j)^{c_j}.
        \end{split}
    \end{equation*}
    Dividing the left-hand side by the right-most factor, we conclude.
\end{proof}

By using the trivial inequality $\cal Z\leq 1$, we may arrive at Theorem~\ref{thm:nonlin_local_bdd}.

\begin{proof}[Proof of Theorem~\ref{thm:nonlin_local_bdd}]
    Let $\delta$ be as in Corollary~\ref{cor:nonlin_quant_ineq}. It suffices to test Schwartz--Bruhat functions. Such functions are locally constant at scale $\varpi^{K+1}\delta$ for sufficiently large $K$. Combine Corollary~\ref{cor:nonlin_quant_ineq} with the trivial inequality $\cal Z\leq 1$.
\end{proof}

\subsection{A converse estimate}\label{subsec:converse}

Above the scaling line, we show that we don't lose much by replacing $\bf B$ by its worst linearization $d\bf B(u)$. The following is a restatement of Theorem~\ref{thm:main_L_by_NL}.

\begin{proposition}[Matching lower bound above the scaling line]
    Let $U\subseteq\K^n$ be open. Suppose that $\bf B=(B_j)_{j=1}^m:U\to\K^{n_1}\times\cdots\times\K^{n_m}$ is a $C^1$ map. Suppose that $\bf c=(c_j)_{j=1}^m$ satisfies the super-scaling condition
    \begin{equation*}
        n\leq\sum_{j=1}^mc_jn_j.
    \end{equation*}
    Then we have
    \begin{equation*}
        \operatorname{BL}(\bf B,\bf c)\geq\sup_{x\in U}\operatorname{BL}(d\bf B(x),\bf c).
    \end{equation*}
\end{proposition}

\begin{proof}
    Fix $x\in U$; our goal will be to show that $\operatorname{BL}(d\bf B(x),\bf c)\leq\operatorname{BL}(\bf B,\bf c)$. Since $n\leq\sum_jc_jn_j$, by Proposition~\ref{prop:ker_cond} we may assume that $d\bf B(x)$ is injective. Write
    \begin{equation*}
        c=\min_{\|\bf v\|=1}\|d\bf B(x).\bf v\|>0.
    \end{equation*}
    Let $\eta\in\K^\times$ be such that $x+\eta\cal O_n\subseteq U$ and, if $y,z\in x+\eta\cal O_n$, then we have
    \begin{equation*}
        \|B_j(y)-B_j(x)-dB_j(x).(y-x)\|\leq|\varpi|c\|y-x\|,\quad\forall j.
    \end{equation*}
    It follows that
    \begin{equation*}
        \|\bf B(y)-\bf B(z)\|=\|d\bf B(x).(y-z)\|,
    \end{equation*}
    for all such $y,z$. Next, note that we have the restriction-monotonicity
    \begin{equation*}
        \operatorname{BL}\big(\left.\bf B\right|_{x+\eta\cal O_n},\bf c\big)\leq\operatorname{BL}(\bf B,\bf c).
    \end{equation*}
    From now on, we'll take $U=x+\eta\cal O_n$.
    
    Let $\bf f=(f_j)_j$ be a tuple of nonnegative Schwartz--Bruhat functions, such that $f_j$ is constant on cosets of $\kappa\cal O_{n_j}$, and supported in $\gamma\cal O_{n_j}$, for each $j$. For $h\in\K^\times$, write
    \begin{equation*}
        \bf g_h=(g_{h,j})_j,\quad g_{h,j}(y)=|h|^{-n_j}f_j(h^{-1}(y-B_j(x)).
    \end{equation*}
    Note that $\int g_{h,j}=\int f_j$ for each $h\in\K^\times$. Note also that the function
    \begin{equation*}
        \prod_{j=1}^mg_{h,j}^{c_j}\circ B_j
    \end{equation*}
    is supported on the set
    \begin{equation*}
        \bigcap_{j=1}^mB_j^{-1}(B_j(x)+h\gamma\cal O_{n_j}).
    \end{equation*}
    Since $\|\bf B(y)-\bf B(x)\|=\|d\bf B(x).(y-x)\|$ for all $y\in U$, we see that
    \begin{equation*}
        \begin{split}
        \|B_j(y)-B_j(x)\|\leq|h\gamma|\,\,\forall j&\quad\iff\quad\|\bf B(y)-\bf B(x)\|\leq|h\gamma|\\
        &\quad\iff\quad\|d\bf B(x).(y-x)\|\leq|h\gamma|.
        \end{split}
    \end{equation*}
    Thus, $\prod_{j=1}^mg_{h,j}^{c_j}\circ B_j$ is supported in $T=x+h\gamma\cdot d\bf B(x)^{-1}(\cal O_{n_1+\ldots+n_m})$. Write
    \begin{equation*}
        D=\sup\,\big\{\|\xi\|:\xi\in d\bf B(x)^{-1}(\cal O_{n_1+\ldots+n_m})\big\}.
    \end{equation*}
    
    We compute:
    \begin{equation*}
        \begin{split}
        g_{h,j}(B_j(z))&=|h|^{-n_j}f_j(h^{-1}(B_j(z)-B_j(x))\\
        &=|h|^{-n_j}f_j(h^{-1}(dB_j(x).(z-x))+h^{-1}(B_j(z)-B_j(x)-dB_j(x).(z-x)))
        \end{split}
    \end{equation*}
    Since $f_j$ is $\kappa\cal O_{n_j}$-invariant, we will obtain the identity
    \begin{equation*}
        g_{h,j}(B_j(z))=|h|^{-n_j}f_j(h^{-1}(d B_j(x).(z-x))),
    \end{equation*}
    as long as
    \begin{equation*}
        B_j(z)-B_j(x)-dB_j(x).(z-x)\in h\kappa\cal O_{n_j}.
    \end{equation*}
    If $h$ is sufficiently small, then for any $z\in T$ we have 
    \begin{equation*}
        \|B_j(z)-B_j(x)-dB_j(x).(z-x)\|\leq|h\gamma|D\times\frac{|\kappa|}{|\gamma|D}.
    \end{equation*}
    Thus, for such $h$,
    \begin{equation*}
        \begin{split}
            \int_U\prod_{j=1}^m(g_{h,j}^{c_j}\circ B_j)(z)dz&=\int_{T}\prod_{j=1}^m(g_{h,j}^{c_j}\circ B_j)(z)dz\\
            &=\int\prod_{j=1}^m|h|^{-n_jc_j}f_{j}^{c_j}(h^{-1}dB_j(x).(z-x))dz.
        \end{split}
    \end{equation*}
    Note however that the super-scaling condition implies
    \begin{equation*}
        \frac{\int\prod_{j=1}^m|h|^{-n_jc_j}f_{j}^{c_j}(h^{-1}dB_j(x).z)dz}{\prod_{j=1}^m\Big(\int |h|^{-n_j}f_{j}(h^{-1}y)dy\Big)^{c_j}}\geq\frac{\int\prod_{j=1}^mf_{j}^{c_j}(dB_j(x).z)dz}{\prod_{j=1}^m\Big(\int f_{j}(y)dy\Big)^{c_j}}=\operatorname{BL}(d\bf B(x),\bf c;\bf f).
    \end{equation*}
    Thus, we have shown that
    \begin{equation*}
        \operatorname{BL}(\bf B,\bf c)\geq\operatorname{BL}(d\bf B(x),\bf c;\bf f)
    \end{equation*}
    for any normalized $\bf f$ and any $x\in U$. We conclude that
    \begin{equation*}
        \sup_{x\in U}\operatorname{BL}(d\bf B(x),\bf c)\leq\operatorname{BL}(\bf B,\bf c).
    \end{equation*}
    
\end{proof}

\section{Nonlinear H\"older--Brascamp--Lieb inequalities over the integers: necessary conditions}\label{sec:nonlin_nec}

In this section, we discuss necessary conditions for the finiteness of polynomial H\"older--Brascamp--Lieb inequalities over the integers, and relate these conditions to our main conditions in Theorem~\ref{thm:int_by_padic}. The second condition, Proposition~\ref{prop:alg_rank}, is roughly equivalent to a necessary condition for real nonlinear H\"older--Brascamp--Lieb inequalities (morally arising from the theorem of Frobenius), together with the assumption that the $k$-dimensional ``integral manifold'' has $\approx N^k$-many lattice points in balls of radius $N$. When this latter assumption fails, one may expect ``smoothing'' phenomenae, with estimates stronger than H\"older--Brascamp--Lieb.

\begin{proposition}[Finite fibers]\label{prop:fiber}
    Suppose $\bf P=(P_j)_{1\leq j\leq m}$ is a tuple of integer polynomial maps such that, for some $b_1\in\Z^{n_1},\ldots,b_m\in\Z^{n_m}$, the fiber $\{x\in\Z^n:P_1(x)=b_1,\ldots,P_m(x)=b_n\}$ is infinite. Then the functional $\Lambda(f_1,\ldots,f_m)=\sum\prod_jf_j\circ P_j$ is unbounded on $\ell^{q_1}\times\cdots\times\ell^{q_m}$, for any $1\leq q_1,\ldots,q_m\leq+\infty$.
\end{proposition}
\begin{proof}
    We have
    \begin{equation*}
        \Lambda(1_{\{b_1\}},\ldots,1_{\{b_m\}})=+\infty.
    \end{equation*}
\end{proof}

The next condition, roughly speaking, asserts the following. Suppose that there is a $(n-k)$-dimensional submanifold $\cal M$ such that each $P_j$ is constant along a $(k-k_j)$-dimensional submanifold of $\cal M$. Then the rank condition $k\leq\sum_j\frac{k_j}{q_j}$ is, in fact, a necessary condition.
\begin{proposition}[Polynomial rank condition]\label{prop:alg_rank}
    Suppose $\bf P=(P_j)_{1\leq j\leq m}$ is a tuple of integer polynomial maps such that $\Lambda$ is bounded on $\ell^{q_1}\times\cdots\times\ell^{q_m}$, with $1\leq q_1,\ldots,q_m\leq +\infty$.
    
    Now, suppose there are $A\in\mathrm{GL}_n(\Z)$ and $1\leq k\leq n-1$, and $Q_1,\ldots,Q_{n-k}\in\Z[x_1,\ldots,x_k]$ are such that, for each $1\leq j\leq m$, there is an injective homomorphism $R_j:\Z^{k-k_j}\to\Z^k$ such that
    \begin{equation}\label{eq:poly_rank}
        \Z^{k-k_j}\ni y\overset{F_j}\mapsto P_j(A(R_j(y)+c,Q_1(R_j(y)+c),\ldots,Q_{n-k}(R_j(y)+c)))
    \end{equation}
    is constant, for each $c\in\Z^k$ and each $1\leq j\leq m$. Then
    \begin{equation*}
        k\leq\sum_j\frac{k_j}{q_j}.
    \end{equation*}
\end{proposition}
\begin{proof}
    For $N\in\N$, write $\Omega_N$ for the set
    \begin{equation*}
        \Omega_N=\{A(w,Q_1(w),\ldots,Q_{n-k}(w)):w\in[0,N-1]^k\cap\Z^k\}.
    \end{equation*}
    Note that $\#\Omega_N=N^k$. For $c\in\{0\}^{k-k_j}\times \Z^{k_j}$, write $\Phi_j(c)$ for the common value of~\eqref{eq:poly_rank}. Then we may express $\#\Omega_N$ as
    \begin{equation*}
        \#\Omega_N=\sum_{c\in\{0\}\times[0,N-1]_\Z^{k_j}}\frac{\#\{z\in\Omega_N:P_j(z)=\Phi_j(c)\}}{\#\{d\in\{0\}\times[0,N-1]_\Z^{k_j}:\Phi_j(d)=\Phi_j(c)\}}.
    \end{equation*}
   Write $\Omega_N^c$ for the set in the numerator. Then $A^{-1}(\Omega_N^c)$ contains the values $R_j(y)+c$, as long as $R_j(y)+c\in[0,N-1]^k.$ When $N$ is sufficiently large, we have $\# R_j^{-1}\big[-c+[0,N-1]^k\cap\Z^k\big]\gtrsim N^{k-k_j}$, because $R_j$ is a $\Z$-linear injection. Thus,
   \begin{equation*}
       \#\Omega_N\gtrsim N^{k-k_j}\sum_{c\in\{0\}\times[0,N-1]_\Z^{k_j}}\frac{1}{\#\{d\in\{0\}\times[0,N-1]_\Z^{k_j}:\Phi_j(d)=\Phi_j(c)\}}.
   \end{equation*}
   On the other hand, the sum on the right-hand side just coincides with $\#P_j(\Omega_N)$. Thus, we have arrived at
   \begin{equation*}
       \# P_j(\Omega_N)\lesssim N^{k_j},\quad\text{for $N$ sufficiently large.}
   \end{equation*}
    In particular, from the boundedness assumption applied to $f_j=1_{P_j(\Omega_N)}$, we obtain for large $N$
    \begin{equation*}
        N^k\lesssim N^{\sum_j\frac{k_j}{q_j}},
    \end{equation*}
    and the result follows.
\end{proof}

We now describe the relationship between the necessary conditions described in Propositions~\ref{prop:fiber} and ~\ref{prop:alg_rank}, and the rank assumptions in Theorem~\ref{thm:int_by_padic}.

\begin{lemma}
    Suppose $\bf P=(P_j)_{1\leq j\leq m}$ is a tuple of integer polynomial maps.
    \begin{enumerate}[label=(\alph*)]
        \item If $d\bf P(a)$ has full rank for every $a\in\Z^n$, then $\bf P$ has finite fibers.
        \item Let $k,k_j,A,Q_\iota,R_j$ be as in Proposition~\ref{prop:alg_rank}. Assume $k_j$ is minimal: that is, $R_j$ cannot be extended to a higher-rank injective map for which the same conclusion applies. Let $V$ be the $\R$-tangent space to the manifold $\{A(w,Q_1(w),\ldots,Q_{n-k}(w)):w\in\R^k\}$ at $w$. Then, for $a=A(w,Q_1(w),\ldots,Q_{n-k}(w))$, we have
        \begin{equation*}
            \dim V=k,\quad\dim(dP_j(a)(V))=k_j,
        \end{equation*}
        so that
        \begin{equation*}
            k\leq\sum_{j=1}^m\frac{k_j}{q_j}\,\,\iff\,\,\dim V\leq\sum_{j=1}^m\frac{\dim(dP_j(a)(V))}{q_j}.
        \end{equation*}
    \end{enumerate}
\end{lemma}
\begin{proof}
    (a):\footnote{I am indebted to the helpful Mathoverflow answer~\cite{MO2026answer} by Jack Huizenga, in identifying this argument.} If $\bf P$ has an infinite fiber, then the extension of $\bf P$ to a map $\C^n\to\C^{n_1+\ldots+n_m}$ has a positive-dimensional irreducible component of that same fiber; moreover, only finitely many integer points can be contained in zero-dimensional components, so this positive-dimensional component contains an integer point. At such a point $a$, $d\bf P(a)$ has kernel in $\C^n$. By easy algebraic considerations (e.g.\ flatness of $\Q\subseteq\C$), it follows that $d\bf P(a)$ has a rational kernel vector, and hence $d\bf P(a)$ has less than full rank, as an integer matrix.

    (b): Differentiating~\eqref{eq:poly_rank} with respect to $y$, where $w=R_jy+c$, we obtain (for $a=A(w,Q_1(w),\ldots, Q_{n-k}(w))$)
    \begin{equation*}
        dP_j(a).A.\begin{bmatrix}
            I_k\\dQ_1(R_jy+c)\\\vdots\\dQ_{n-k}(R_jy+c)
        \end{bmatrix}.R_j=0.
    \end{equation*}
    Note that $V$ is just the space spanned by the columns of
    \begin{equation*}
        A.\begin{bmatrix}
            I_k\\dQ_1(R_jy+c)\\\vdots\\dQ_{n-k}(R_jy+c)
        \end{bmatrix}.
    \end{equation*}
    Since $R_j$ has rank $k-k_j$, it follows that $dP_j(a)(V)$ is at most $k_j$-dimensional. Since $R_j$ could not be replaced by a higher-rank injective map, it follows that $dP_j(a)(V)$ is exactly $k_j$-dimensional. The result follows.
\end{proof}

Thus, the submersion condition on $d\bf P$ is stronger than the necessary condition that $\bf P$ has finite fibers; the rank condition
\begin{equation*}
    \mathrm{rank}(V)\leq\sum_{j=1}^m\frac{\mathrm{rank}(dP_j(a)(V))}{q_j}
\end{equation*}
is stronger than the assumption that there is no integer variety whose tangent space is $\mathrm{span}_\R V$ with $\approx N^{\mathrm{rank}(V)}$-many integral points in balls of radius $N$.

\section{Transference of polytopes between fields}\label{sec:rank_transfer}

In this section, we will prove that the rank conditions defining the H\"older--Brascamp--Lieb polytopes for a particular choice of homomorphisms $L_j$ are essentially equivalent for a broad class of underlying rings $R$, as long as the $L_j$ make suitable sense over $R$. Specifically, we will show that, for a given tuple of integer-linear maps $L_j$, one can compute a positive integer $\Delta$ such that the rank conditions over $\F_p$ agree with the rank conditions over $\Z$, as long as $p\nmid\Delta$. As an upshot, rank conditions over $\Z$ will suffice to be able to pass to the $p$-adic setting, which will enable us to use our local field theory. The main result of this section is Theorem~\ref{thm:rank_transfer}.

Fix an arbitrary field $\F$, understood without a topology, which we equip with counting measure. We will find a suitable ``combinatorial'' structure, which characterizes the set $\cal P$ of exponents such that a Brascamp--Lieb inequality holds. By using one such combinatorial structure over an $\F$-vector space $V$ to find another, matching combinatorial structure over an $\F'$-vector space $V'$, we may transfer Brascamp--Lieb boundedness over $\F$ to $\F'$.

Our approach is inspired by~\cite{christ2015discrete}. In the latter, a decision procedure is described which takes in a tuple $\bf L=(L_j)_j$ of surjective linear maps out of a common $\F$-vector space $V$, where $\F$ is a general (at most) countable field and $V$ is finite-dimensional, and by listing out subspaces $W$ of $V$ and calculating $\dim(W)$, $\dim(L_jW)$ it decides, in finite time, that it has fully characterized the polytope $\cal P(\bf L)$ defined by the scaling and rank conditions.~\cite{christ2015discrete} proves that this algorithm always succeeds. Our approach can be described as follows: we run the algorithm for integer matrices, regarded as linear maps over $\Q$, and after it terminates we read off the subspaces that were reached by the time it terminated. We then choose a prime $p$ such that, if we treat the maps $L_j$ as $\Z$-linear maps between lattices, the same algorithm run over $\F_p=\Z/p\Z$ with the corresponding reduced data, will terminate before it ``realizes'' that anything was different from the $\Q$ case. Thus, the same polytope is realized after reducing mod $p$, for cofinitely-many $p$.

The combinatorial structures we choose to reason with are ``witnesses'' and ``folios;'' they were implicitly present in the algorithms described in~\cite{christ2015discrete}.

\begin{definition}[Rank/scaling polytope]
    Let $\bf L=(L_j)_{1\leq j\leq m}$ be a tuple of surjective $\F$-linear maps defined out of a common $n$-dimensional $\F$-vector space $V$. Write $\cal P(\bf L)$ to be the subset of $[0,1]^m$ defined by the conditions
    \begin{equation*}
        \forall \bf 0\neq W\lneq V\,\,\text{subspace,}\quad\dim_\F W\leq\sum_{j=1}^mc_j\dim_\F(L_jW),
    \end{equation*}
    \begin{equation*}
        \dim_\F V=\sum_{j=1}^mc_j\dim_\F(L_jV).
    \end{equation*}
\end{definition}
\begin{remark}
    $\cal P(\bf L)$ is fully determined by a finite number of linear inequalities.
\end{remark}

\begin{definition}[Polytope determined by a family of subspaces]
    Let $\bf L=(L_j)_{1\leq j\leq m}$ be a tuple of surjective $\F$-linear maps out of a common $n$-dimensional $\F$-vector space $V$. Let $\{V_\tau\}_\tau$ be a finite set of vector subspaces of $V$. The \emph{polytope} $\operatorname{Poly}(\bf L,\{V_\tau\}_\tau)$ is the convex set in $\R^m$ cut out by the (in)equalities
    \begin{equation*}
        0\leq c_j\leq 1,\quad\forall j,
    \end{equation*}
    \begin{equation*}
        n=\sum_{j=1}^m c_j\dim(L_jV),
    \end{equation*}
    \begin{equation*}
        \dim V_\tau\leq\sum_{j=1}^mc_j\dim(L_jV_\tau),\quad\forall\tau\,\,\text{s.t. $\bf 0\neq V_\tau\subsetneq V$.}
    \end{equation*}
\end{definition}

\begin{definition}[Extreme points]
    Given $\cal C\subseteq\R^n$ compact convex, we write $\partial^*\cal C$ for the set of extreme points; i.e.\ those points of $\cal C$ which are not interpolants of two other points of $\cal C$.
\end{definition}

We will need the following two results from~\cite{christ2015discrete}:

\begin{proposition}[\cite{christ2015discrete}, Proposition 2.12]\label{prop:christ_crit_exist}
    Suppose $\bf c\in\partial^*\cal P(\bf L)$. Then at least one of the following is true: $\bf c\in\{0,1\}^m$, or there is a critical subspace $\bf 0\neq W\subsetneq V$.
\end{proposition}

\begin{theorem}[\cite{christ2015discrete}, Theorem 1.4]\label{thm:christ_disc_suff} The following conditions are mutually equivalent, for any discrete Abelian groups $G,\{G_j\}_{1\leq j\leq m}$ with $G$ torsion-free and homomorphisms $\phi_j:G\to G_j$ and any $\bf c\in[0,1]^m$.
\begin{enumerate}[label=(\alph*):]
    \item $\mathrm{rank}(H)\leq\sum_{j=1}^mc_j\mathrm{rank}(\phi_j(H))$, for all subgroups $H\leq G$.
    \item There exists a constant $C<+\infty$ such that, for any nonnegative finitely-supported functions $f_j$ on $G_j$,
    \begin{equation*}
        \sum_{x\in G}\prod_{j=1}^m f_j(\phi_j(x))\leq C\prod_{j=1}^m\|f_j\|_{\ell^{1/c_j}(G_j)}.
    \end{equation*}
    \item There exists $A<+\infty$ such that
    \begin{equation*}
        \#E\leq A\prod_{j=1}^m\big(\#\phi_j(E)\big)^{c_j},\quad\forall E\subseteq G\,\,\text{finite.}
    \end{equation*}
    \item Item (b) holds with $C=1$.
    \item Item (c) holds with $A=1$.
\end{enumerate}
\end{theorem}
We will also need the following multilinear interpolation theorem.
\begin{theorem}[Multilinear Riez--Thorin; Theorem 2.7 of~\cite{bennett1988interpolation}]\label{thm:ml_interp}
    Suppose $(S,\nu)$ and $(R_j,\mu_j)$, $1\leq j\leq m$ are measure spaces, and $T$ is a multilinear map defined on tuples of simple functions $(f_1,\ldots,f_m)$, with $f_j$ a complex-valued simple function over $R_j$, such that $T(f_1,\ldots,f_m)$ is a function over $S$. Suppose $1\leq q_1,q_2\leq+\infty$ and $1\leq p_{1,1},p_{2,1},\ldots,p_{1,m},p_{2,m}\leq+\infty$ and $M_1,M_2\in\R_{\geq 0}$ are such that
    \begin{equation*}
        \|T(f_1,\ldots,f_m)\|_{L^{q_i}(S,\nu)}\leq M_i\prod_{j=1}^m\|f_j\|_{L^{p_{i,j}}(R_j,\mu_j)},\quad\text{for each of $i=1,2$}.
    \end{equation*}
    Then, for any $0\leq\theta\leq 1$, we have the bound
    \begin{equation*}
        \|T(f_1,\ldots,f_m)\|_{L^{q}(S,\nu)}\leq M_1^{1-\theta}M_2^\theta\prod_{j=1}^m\|f_j\|_{L^{p_{j}}(R_j,\mu_j)},
    \end{equation*}
    when we choose exponents
    \begin{equation*}
        \frac{1}{q}=\frac{1-\theta}{q_1}+\frac{\theta}{q_2},\quad\frac{1}{p_j}=\frac{1-\theta}{p_{1,j}}+\frac{\theta}{p_{2,j}}.
    \end{equation*}
\end{theorem}

In subsection~\ref{subsec:folio_witness}, we construct the ``witnesses'' and ``folios'' which will characterize the sets $\cal P(\bf L)$, and prove many results about them; the main output is Proposition~\ref{prop:fol_transfer}. In subsection~\ref{subsec:transfer}, we apply Proposition~\ref{prop:fol_transfer} to show that the rank/scaling conditions over $\Z$ descend to rank/scaling conditions over $\F_p$ for cofinitely-many $p$.

\subsection{Folios and witnesses}\label{subsec:folio_witness}

As indicated, we reason in terms of two constructions, \emph{folios} and \emph{witnesses}. We define each of them inductively. The two are defined with reference to each other. After the two definitions, we make a remark which validates that there is in fact no problematic circularity.

We briefly motivate the concepts and corresponding names. A witness for $(\bf c,\bf L)$ is a collection of subspaces $\{V_\tau\}_\tau$ such that the dimensions $\dim V_\tau$, $\dim(L_jV_\tau)$ and containment relations among the $V_\tau$ suffice to prove that $\bf c\in\cal P(\bf L)$; thus, a witness ``witnesses'' boundedness of the functional at a particular $\bf c$. A folio for $\bf L$ is a collection of subspaces $\{V_\tau\}_\tau$ such that, for every $\bf c\in\partial^*\operatorname{Poly}(\bf L,\{V_\tau\}_\tau)$, there is a subfamily $\cal A\subseteq\{V_\tau\}_\tau$ such that $\cal A$ is a witness for $(\bf c,\bf L)$. It follows that $\partial^*\operatorname{Poly}(\bf L,\{V_\tau\}_\tau)\subseteq\cal P(\bf L)$, and hence $\cal P(\bf L)=\operatorname{Poly}(\bf L,\{V_\tau\}_\tau)$. Thus, the dimensions and containment relations of subspaces and their images from a folio, suffice to determine $\cal P(\bf L)$; put another way, a folio is a ``portfolio'' fully accounting for the shape of $\cal P(\bf L)$.

\begin{definition}[Witness]
    Let $\bf L=(L_j)_{1\leq j\leq m}$ be a tuple of surjective $\F$-linear maps out of a common $n$-dimensional $\F$-vector space $V$. Let $\bf c\in\R^m$. We inductively define the class of \emph{witnesses} $\cal A\in\operatorname{Wit}(\bf c,\bf L)$ of $(\bf c,\bf L)$ as a class of finite sets of subspaces of $V$, possibly empty.
    \begin{itemize}
        \item If $\bf c\not\in[0,1]^m$, then $\operatorname{Wit}(\bf c,\bf L)=\emptyset$.

        Otherwise, we assume from now on $\bf c\in[0,1]^m$.
        \item If $n\neq\sum_{j=1}^mc_j\mathrm{rank}(L_j)$, then $\operatorname{Wit}(\bf c,\bf L)=\emptyset$.

        Otherwise, we assume from now on $n=\sum_{j=1}^mc_j\mathrm{rank}(L_j)$.
        \item If 
        \begin{equation*}
            \bigcap_{\substack{1\leq j\leq m\\c_j\neq 0}}\ker L_j\neq 0,
        \end{equation*}
        then $\operatorname{Wit}(\bf c,\bf L)=\emptyset$. Otherwise, we assume that the sub-tuple of $L_j$ for which $c_j\neq 0$ is non-degenerate.
        \item If $n=1$, then $\operatorname{Wit}_1(\bf c,\bf L)=\{\{V\}\}$.
        \item If $m=1$,  then $\operatorname{Wit}_2(\bf c,\bf L)=\{\{V\}\}$.
        
        In the remaining bullet points, we assume that $n,m\geq 2$, and then we set $\operatorname{Wit}_1(\bf c,\bf L)=\operatorname{Wit}_2(\bf c,\bf L)=\emptyset$. Recall that $\bf c\in[0,1]^m$ and $n=\sum_{j=1}^mc_j\mathrm{rank}(L_j)$.
        \item If $\bf c\in\{0,1\}^{m}$, then $\operatorname{Wit}_3(\bf c,\bf L)=\{\{V\}\}$.

        Otherwise, we assume that $c_j\in(0,1)$ for at least one index $j$, and we set $\operatorname{Wit}_3(\bf c,\bf L)=\emptyset$.
        \item If $\bf c\in[0,1]^m\setminus\{0,1\}^m$, then
        \begin{equation*}
            \operatorname{Wit}_4(\bf c,\bf L)=\bigcup_{\substack{\bf 0\neq W\subsetneq V\\ W\,\text{critical subspace}}}\left\{\cal A\cup\cal B:\cal A\in\operatorname{Wit}\big(\bf c,\left.\bf L\right|_{W}\big),\cal B\in\operatorname{Wit}\big(\bf c,\left.\bf L\right|_{V/W}\big)\right\}.
        \end{equation*}
        Here, $\cal A\in\operatorname{Wit}\big(\bf c,\left.\bf L\right|_{W}\big)$ is treated as a set of subspaces of $V$ contained in $W$, and $\cal B\in\operatorname{Wit}\big(\bf c,\left.\bf L\right|_{V/W}\big)$ is treated as a set of subspaces of $V$ which contain $W$.

        \item We set
        \begin{equation*}
            \operatorname{Wit}_5(\bf c,\bf L)=\left\{\cal A\in\operatorname{Folio}(\bf L):\bf c\in\operatorname{Poly}(\bf L,\cal A)\right\}.
        \end{equation*}
        
        \item Finally, if we have not already declared $\operatorname{Wit}(\bf c,\bf L)=\emptyset$, then we write
        \begin{equation*}
            \operatorname{Wit}(\bf c,\bf L)=\bigcup_{i=1}^5\operatorname{Wit}_i(\bf c,\bf L).
        \end{equation*}
    \end{itemize}
\end{definition}

\begin{definition}[Folio]
    Let $\bf L=(L_j)_{1\leq j\leq m}$ be a tuple of $\F$-linear maps out of a common $n$-dimensional $\F$-vector space $V$. We inductively define the class of \emph{folios} $\operatorname{Folio}(\bf L)$ of $\bf L$.

    A \emph{folio} for $\bf L$ is a finite nonempty set $\cal F$ of nonzero subspaces of $V$ such that $\operatorname{Poly}(\bf L,\cal F)\neq\emptyset$ and, for every $\bf c\in\partial^*\operatorname{Poly}(\bf L,\cal F)$, there is a subcollection $\cal A\subseteq\cal F$ such that $\cal A\in\operatorname{Wit}_i(\bf c,\bf L)$ for some $1\leq i\leq 4$. $\operatorname{Folio}(\bf L)$ is the set of all folios of $\bf L$.
\end{definition}

\begin{remark}[Well-founded]
    Witnesses and folios are well-defined. When, in defining a folio, we appeal to witnesses, we are considering an extreme point $\bf c$ of a polytope. Any such $\bf c$, after potentially removing $0,1$ entries, will be associated to a critical subspace. Thus, a witness here will be with respect to lower-dimensional spaces. Similarly, when appealing to folios in the definition of witnesses, the dimension does not increase. So, the definitions are inductively well-defined.
\end{remark}

\begin{lemma}\label{lem:wit_has_top}
    Suppose $\cal A\in\operatorname{Wit}(\bf c,\bf L)$ for some $\bf c$, where $\bf L$ is defined on $V$. Then $V\in\cal A$.

    Similarly, if $\cal F\in\operatorname{Folio}(\bf L)$, we have that $V\in\cal F$. 
\end{lemma}
\begin{proof}
    We show this by induction. If $n=1$ and $\cal A\in\operatorname{Wit}_i(\bf c,\bf L)$ for $1\leq i\leq 3$, then the result is trivial. Since $n=1$, $\operatorname{Wit}_4(\bf c,\bf L)=\emptyset.$ Since $V$ is $1$-dimensional, the only possible folios are $\emptyset$ or $\{V\}$; thus, if $\cal A\in\operatorname{Wit}_5(\bf c,\bf L)$, then $\cal A=\{V\}$.

    Now, assume that the result has been shown for dimensions smaller than $n$. Suppose that $\cal A\in\operatorname{Wit}_i(\bf c,\bf L)$ for some $1\leq i\leq 4$. If $i\in\{1,2,3\}$, then the result is trivial. If $i=4$, then there is a critical subspace $\bf 0\neq W\subsetneq V$ such that $\cal A=\cal A'\cup\cal B'$ and
    \begin{equation*}
        \cal A'\in\operatorname{Wit}\big(\bf c,\left.\bf L\right|_W\big),\quad \cal B'\in\operatorname{Wit}\big(\bf c,\left.\bf L\right|_{V/W}\big).
    \end{equation*}
    By induction on $n$, we have $V\in\cal B'$. Thus, $V\in\cal A$. The remaining case is that $\cal A$ is a folio.

    Suppose $\cal F\in\operatorname{Folio}(\bf L)$. Then $\operatorname{Poly}(\bf L,\cal F)$ is nonempty compact and convex, so $\partial^*\operatorname{Poly}(\bf L,\cal F)\neq\emptyset$. Let $\bf c\in\partial^*\operatorname{Poly}(\bf L,\cal F)$. Then there is a subcollection $\cal A\subseteq\cal F$ such that $\cal A\in\operatorname{Wit}_i(\bf c,\bf L)$ for some $1\leq i\leq 4$. By the previous case, $V\in\cal A\subseteq\cal F$.
\end{proof}

\begin{example}
     Suppose $n=2$ and $\bf L=(L_1,L_2,L_3)$ are such that
     \begin{equation*}
         L_1(x,y)=x,\quad L_2(x,y)=y,\quad L_3(x,y)=x+y.
     \end{equation*}
     Then $\{\F^2\}\in\operatorname{Folio}(\bf L)$. Indeed, $\operatorname{Poly}(\bf L,\{\F^2\})$ is the set of $\bf c\in[0,1]^3$ such that $c_1+c_2+c_3=2$. Thus,
     \begin{equation*}
         \partial^*\operatorname{Poly}(\{V_i\}_{i})=\{(1,1,0),(1,0,1),(0,1,1)\}.
     \end{equation*}
     Since each such extreme $\bf c$ has witness $\{\F^2\}$, we conclude that $\{\F^2\}$ is a folio.
     
\end{example}

\begin{example}
     Suppose $n=3$ and $\bf L=(L_1,L_2,L_3)$ are such that
     \begin{equation*}
         L_1(x,y,z)=(y,z),\quad L_2(x,y,z)=(x,z),\quad L_3(x,y,z)=(x,y).
     \end{equation*}
     Then $V_1=\mathrm{span}_\F(\bf e_1), V_2=\mathrm{span}_\F(\bf e_1,\bf e_2), V_3=\F^3$ are such that $\{V_1,V_2,V_3\}\in\operatorname{Folio}(\bf L)$. Indeed, $\operatorname{Poly}(\bf L,\{V_i\}_{i})=\{(\frac{1}{2},\frac12,\frac12)\}$, so
     \begin{equation*}
         \partial^*\operatorname{Poly}(\{V_i\}_{i})=\Big\{\Big(\frac{1}{2},\frac12,\frac12\Big)\Big\}.
     \end{equation*}
     Then $V_1$ is critical for this point, and
     \begin{equation*}
         \left.\bf L\right|_{V_1}\simeq\bf L'=(L_1',L_2',L_3'):\F\to\{0\}\times\F\times\F,\quad L_1'=0,\,\,L_2'=L_3'=\operatorname{Id},
     \end{equation*}
     \begin{equation*}
         \left.\bf L\right|_{\F^3/V_1}\simeq\bf L''=(L_1'',L_2'',L_3''):\F^2\to\F^2\times\F\times\F,\,\,L_1''=\operatorname{Id},\,\,L_2''(x,y)=x,\,\,L_3''(x,y)=y.
     \end{equation*}
    With respect to $\bf c=(\frac 12,\frac 12,\frac 12)$, $\bf L'$ has witness $\{\F\}\in\operatorname{Wit}_1(\bf c,\bf L')$, corresponding to $V_1$. $\bf L''$ has critical subspace $\F\times\{0\}$, and each of the factor data is $1$-dimensional; thus, $(\bf c,\bf L'')$ has witness $\{\F\times\{0\},\F^2\}\in\operatorname{Wit}_4(\bf c,\bf L'')$, corresponding to $V_2$ and $V_3$, respectively. Thus, $(\bf c,\bf L)$ has witness $\{V_1,V_2,V_3\}\in\operatorname{Wit}_4(\bf c,\bf L)$. Since $\bf c$ is the unique member of $\partial^*\operatorname{Poly}(\bf L,\{V_i\}_{i})$, we conclude that $\{V_i\}_i$ is a folio for $\bf L$.

\end{example}

\begin{proposition}[Witnesses have good polytopes]
    For every $\cal A\in\operatorname{Wit}(\bf c,\bf L)$, $\bf c\in\operatorname{Poly}(\bf L,\cal A)$.
\end{proposition}
\begin{proof}
    By the assumption, $\bf c\in[0,1]^m$ and $n=\sum_jc_j\mathrm{rank}(L_j)$ and $\bigcap_{c_j\neq 0}\ker L_j=0$. If If $n=1$ or $m=1$ or $\bf c\in\{0,1\}^m$, the result is easy. Next, suppose that $n\geq 2$ and that the result is known in dimensions $<n$; as before, the case $\cal A\in\operatorname{Wit}_i(\bf c,\bf L)$ for $1\leq i\leq 3$ is easy. Suppose next that $\cal A\in\operatorname{Wit}_4(\bf c,\bf L)$. Then there is some critical subspace $W$ such that $\cal A=\cal A'\cup\cal B'$, where $\cal A'$ is a collection of subspaces of $W$ which make up a witness for $\big(\bf c,\left.\bf L\right|_{W}\big)$ and $\cal B'$ is a collection of subspaces intermediate between $W$ and $V$ that descend to a witness for $\big(\bf c,\left.\bf L\right|_{V/W}\big)$. By an inductive assumption, we may assume
    \begin{equation*}
        \bf c\in\operatorname{Poly}\big(\left.\bf L\right|_W,\cal A')\cap\operatorname{Poly}\big(\left.\bf L\right|_{V/W},\cal B').
    \end{equation*}
    Thus, for each $V_\tau\in\cal A'$, since $V_\tau\subseteq W$ we have
    \begin{equation*}
        \dim V_\tau\leq\sum_{j=1}^mc_j\dim\big(\left.L_j\right|_{W}(V_\tau)\big)=\sum_{j=1}^mc_j\dim\big(L_j(V_\tau)\big).
    \end{equation*}
    Similarly, for each $V_\tau\in\cal B'$, since $V_\tau\supseteq W$ we have
    \begin{equation*}
        \begin{split}
        \dim (V_\tau)=\dim W+\dim(V_\tau/W)&\leq\dim W+\sum_{j=1}^mc_j\dim\big(\left.L_j\right|_{V/W}(V_\tau/W)\big)\\
        &=\sum_{j=1}^mc_j\Big(\dim (L_jW)+\dim\big(\left.L_j\right|_{V/W}(V_\tau/W)\big)\Big)\\
        &=\sum_{j=1}^mc_j\dim L_j(V_\tau).
        \end{split}
    \end{equation*}
    Thus, $\bf c\in\operatorname{Poly}(\bf L,\cal A)$. Finally, if $\cal A\in\operatorname{Wit}_5(\bf c,\bf L)$, then definitionally we have $\bf c\in\operatorname{Poly}(\bf L,\cal A)$.
    
\end{proof}

\begin{proposition}[Witnesses imply boundedness]\label{prop:wit_imp_bdd}
    Suppose $\cal A\in\operatorname{Wit}(\bf c,\bf L)$. Then, giving $\F$ the counting measure, we have that $\operatorname{BL}(\bf L,\bf c)=1$ over $\F$.
\end{proposition}
\begin{proof}
    Trivially, $\operatorname{BL}(\bf L,\bf c)\geq 1$; the remaining inequality is the focus of our efforts. By the first three bullet points in the definition of $\operatorname{Wit}(\bf c,\bf L)$, we have $\bf c\in[0,1]^m$ and $n=\sum_jc_j\mathrm{rank}(L_j)$ and
    \begin{equation*}
        \bigcap_{\substack{1\leq j\leq m\\c_j\neq 0}}\ker L_j=0.
    \end{equation*}
    If $n=1$, then the bound $\operatorname{BL}(\bf L,\bf c)\leq 1$ follows easily from the scaling condition; from now on, we assume that $n\geq 2$, and the result is known for dimensions $<n$. If $m=1$, then the conditions from the first two bullet points imply $\mathrm{rank}(L_1)=n$ and $c_1=1$, in which case $\operatorname{BL}(\bf L,\bf c)\leq 1$ is immediate.

    Next, we divide into cases depending on which $\operatorname{Wit}_i(\bf c,\bf L)$ belongs to. By the above, we may assume that $i\geq 3$.

    Suppose $\cal A\in\operatorname{Wit}_3(\bf c,\bf L)$. Then $\cal A=\{V\}$ and $\bf c\in\{0,1\}^m$. Further,
    \begin{equation*}
        \bigcap_{j:c_j\neq 0}\ker L_j=0.
    \end{equation*}
    Then, for $J'=\{j:c_j=1\}$, we have that $(L_j)_{j\in J'}$ is injective. Thus, for any nonnegative finitely supported $f_j$ on $\F^{n_j}$,
    \begin{equation*}
        \begin{split}
            \sum_{x\in\F^n}\prod_{j=1}^mf_j(L_j(x))&\leq\Big[\prod_{j\not\in J}\|f_j\|_{\ell^\infty}\Big]\sum_{\substack{(y_j)_{j\in J}\in\prod_{j\in J}L_j(\F^n)}}\prod_{j\in J}f_j(y_j)\\
            &\leq \Big[\prod_{j\not\in J}\|f_j\|_{\ell^\infty}\Big]\Big[\prod_{j\in J}\|f_j\|_{\ell^1}\Big]=\prod_{j=1}^m\|f_j\|_{\ell^{1/c_j}}.
        \end{split}
    \end{equation*}
    So we are done in this case.

    Next, we assume $\cal A\in\operatorname{Wit}_4(\bf c,\bf L)$. Then, for some critical subspace $W\subseteq V$, $\bf 0\neq W\subsetneq V$, we have $\cal A=\cal A'\cup\cal B'$ where $\cal A'\in\operatorname{Wit}(\bf c,\left.\bf L\right|_W)$ and $\cal B'\in\operatorname{Wit}(\bf c,\left.\bf L\right|_{V/W})$, as described in the definition of $\operatorname{Wit}_4(\bf c,\bf L)$. By induction on $n$, we may assume $\operatorname{BL}\big(\left.\bf L\right|_W,\bf c\big)=1=\operatorname{BL}\big(\left.\bf L\right|_{V/W},\bf c\big)$. Then, if $\bf f=(f_j)_{1\leq j\leq m}$ with $f_j$ nonnegative and finitely-supported on $\F^{n_j}$, then (letting $\cal E$ be a set of representatives of cosets $x+V$ in $V/W$)
    \begin{equation*}
        \begin{split}
            \sum_{x\in V}\prod_{j=1}^mf_j(L_jx)&=\sum_{\substack{x\in\cal E\\w\in W}}\prod_{j=1}^mf_j(L_j(x+w))\\
            &\leq\sum_{x\in\cal E}\operatorname{BL}\big(\left.\bf L\right|_W,\bf c\big)\prod_{j=1}^m\Big(\sum_{z\in L_jW}f_j^{1/c_j}(L_jx+z)\Big)^{c_j}
        \end{split}
    \end{equation*}
    Writing 
    \begin{equation*}
        g_j(w)=\sum_{z\in L_jW}f_j^{1/c_j}(w+z),
    \end{equation*}
    we have
    \begin{equation*}
        \begin{split}
            \sum_{x\in V}\prod_{j=1}^mf_j(L_jx)&\leq\operatorname{BL}\big(\left.\bf L\right|_W,\bf c\big)\sum_{x\in\cal E}\prod_{j=1}^mg_j(L_jx)\\
            &\leq\operatorname{BL}\big(\left.\bf L\right|_W,\bf c\big)\operatorname{BL}\big(\left.\bf L\right|_{V/W},\bf c\big)\prod_{j=1}^m\Big(\sum_{z+L_jW}g_j(z)^{1/c_j}\Big)^{c_j}\\
            &=\prod_{j=1}^m\|f_j\|_{\ell^{1/c_j}(\F^{n_j})}.
        \end{split}
    \end{equation*}
    Thus, $\operatorname{BL}(\bf L,\bf c)\leq\operatorname{BL}\big(\left.\bf L\right|_W,\bf c\big)\operatorname{BL}\big(\left.\bf L\right|_{V/W},\bf c\big)=1$.

    Finally, we assume that $\cal A\in\operatorname{Wit}_5(\bf c,\bf L)$. In particular, $\cal A$ is a folio. For every $\bf c^*\in\partial^*\operatorname{Poly}(\bf L,\cal A)$, the above cases imply that $\operatorname{BL}(\bf L,\bf c^*)=1$. By multilinear interpolation, Theorem~\ref{thm:ml_interp}, we obtain the same at $\bf c$.

\end{proof}

\begin{corollary}
    Suppose $\cal A$ is a folio for $\bf L$. Then, for all $\bf c\in\operatorname{Poly}(\bf L,\cal A)$, we have that $\operatorname{BL}(\bf L,\bf c)=1$.
\end{corollary}
\begin{proof}
    For $\bf c\in\partial^*\operatorname{Poly}(\bf L,\cal A)$, the result follows from Proposition~\ref{prop:wit_imp_bdd}. For all other $\bf c$, we use multilinear interpolation, Theorem~\ref{thm:ml_interp}.
\end{proof}

\begin{proposition}[Transference of folios]\label{prop:fol_transfer}
    Let $\F,\F'$ be fields and $\bf L,\bf L'$ be $m$-tuples of surjective linear maps out of $\F$- and $\F'$-vector spaces $V$ and $V'$, respectively, each of dimension $n$; we assume $\mathrm{rank}_\F(L_j)=\mathrm{rank}_{\F'}(L_j')$, for each $j$. Suppose $\cal A$ is a folio for $\bf L$, $\cal B$ is a set of subspaces of $V'$, and $\phi:\cal A\to\cal B$ is a bijection such that: for each $V_\tau,V_\beta\in\cal A$,
    \begin{enumerate}
        \item If $V_\tau\subseteq V_\beta$, then $\phi(V_\tau)\subseteq\phi(V_\beta)$.
        \item $\dim V_\tau=\dim\phi(V_\tau)$.
        \item For each $1\leq j\leq m$, $\dim_\F L_j(V_\tau)=\dim_{\F'} L_j'(\phi(V_\tau))$.
        \item For each tuple $J\subseteq\{1,\ldots,m\}$,
        \begin{equation*}
            \dim_\F\left(\bigcap_{j\in J}\ker L_j\right)=\dim_{\F'}\left(\bigcap_{j\in J}\ker L_j'\right)
        \end{equation*}
    \end{enumerate}
    Then $\cal B$ is a folio for $\bf L'$.
\end{proposition}
An important aspect of this proposition is that the two vector spaces are over different fields. Thus, we will be able to eventually perform a transference procedure to use Brascamp--Lieb analysis over one field to conclude Brascamp--Lieb inequalities over another.

\begin{proof}
    Clearly $\operatorname{Poly}(\bf L,\cal A)=\operatorname{Poly}(\bf L',\cal B)$. Let $\bf c\in\partial^*\operatorname{Poly}(\bf L',\cal B)$; thus, $\bf c\in\partial^*\operatorname{Poly}(\bf L,\cal A)$. Thus, there is a subcollection $\cal A'\subseteq\cal A$ such that $\cal A'\in\operatorname{Wit}_i(\bf c,\bf L)$, for some $1\leq i\leq 4$. We will demonstrate that $\cal B'=\{\phi(V_\tau):V_\tau\in\cal A'\}\in\operatorname{Wit}_i(\bf c,\bf L)$. Note in passing that $\phi(V)=V'$.

    We work through the definition of a witness. Since $\operatorname{Wit}(\bf c,\bf L)\neq\emptyset$, we have $\bf c\in[0,1]^m$ and $n=\sum_jc_j\mathrm{rank}_\F(L_j)=\sum_jc_j\mathrm{rank}_{\F'}(L_j')$. Also,
    \begin{equation*}
        \bigcap_{\substack{1\leq j\leq m\\c_j\neq 0}}\ker L_j=0,
    \end{equation*}
    and
    \begin{equation*}
        0=\dim_\F\left(\bigcap_{\substack{1\leq j\leq m\\c_j\neq 0}}\ker L_j\right)=\dim_{\F'}\left(\bigcap_{\substack{1\leq j\leq m\\c_j\neq 0}}\ker L_j'\right),
    \end{equation*}
    so that
    \begin{equation*}
        \bigcap_{\substack{1\leq j\leq m\\c_j\neq 0}}\ker L_j'=0.
    \end{equation*}

    Suppose $\cal A'\in\operatorname{Wit}_i(\bf c,\bf L)$ for some $1\leq i\leq 3$. Then $\cal A'=\{V\}$, so $\cal B'=\{V'\}$. In each of $1\leq i\leq 3$, it follows that $\cal B'\in\operatorname{Wit}_i(\bf c,\bf L)$.

    Finally, suppose that $\cal A'\in\operatorname{Wit}_4(\bf c,\bf L)$. Then there is a critical subspace $\bf 0\neq W\subsetneq V$ such that
    \begin{equation*}
        \cal A'=\cal C\cup\cal D,\quad\cal C\in\operatorname{Wit}\big(\bf c,\left.\bf L\right|_W\big)\text{ and }\cal D\in\operatorname{Wit}\big(\bf c,\left.\bf L\right|_{V/W}\big).
    \end{equation*}
    By Lemma~\ref{lem:wit_has_top}, $W\in\cal C$, so $W\in\cal A'$. Let $W'=\phi(W)\in\cal B'$. Then, since
    \begin{equation*}
        \dim_\F W=\sum_{j=1}^mc_j\dim_\F(L_jW),
    \end{equation*}
    we have
    \begin{equation*}
        \dim_{\F'}(W')=\sum_{j=1}^m c_j\dim_{\F'}(L_j'W').
    \end{equation*}
    So, $W'$ is critical for $(\bf c,\bf L')$. Write $\cal C'=\{\phi(V_\tau):V_\tau\in\cal C\}$, $\cal D'=\{\phi(V_\tau):V_\tau\in\cal D\}$. Then $\cal B'=\cal C'\cup\cal D'$, and for every $V_\tau'\in\cal C'$ we have $V_\tau'\subseteq W'$, and similarly for every $V_\tau'\in\cal D'$ we have $V_\tau'\supseteq W'$. Finally, by induction on the parameter $n$, we may assume that $\cal C'\in\operatorname{Wit}\big(\bf c,\left.\bf L'\right|_W\big)$ and $\cal D'\in\operatorname{Wit}\big(\bf c,\left.\bf L'\right|_{V/W}\big)$. Thus, $\cal B'\in\operatorname{Wit}_4(\bf c,\bf L')$.

    Thus, we have found that for every $\bf c\in\partial^*\operatorname{Poly}(\bf L,\cal B)$, there is a subcollection $\cal B'\subseteq\cal B$ such that $\cal B'\in\operatorname{Wit}_i(\bf c,\bf L')$, for some $1\leq i\leq 4$. That is, $\cal B\in\operatorname{Folio}(\bf L')$.

\end{proof}

\begin{theorem}[Witnesses exist over $\Q$]\label{thm:wit_exist}
    Suppose $\F$ is characteristic zero, and suppose that $\bf c\in[0,1]^m$ and $n=\sum_jc_j\mathrm{rank}(L_j)$ and $\operatorname{BL}(\bf L,\bf c)<+\infty$. Then $\operatorname{Wit}(\bf c,\bf L)\neq\emptyset$.
\end{theorem}
\begin{proof}
    By the assumptions and Theorem~\ref{thm:christ_disc_suff}, we see that $\bf c\in\cal P(\bf L)$. In particular, we have the conditions
    \begin{equation*}
        \bf c\in[0,1]^m,
    \end{equation*}
    \begin{equation*}
        n=\sum_{j=1}^mc_j\mathrm{rank}(L_j),
    \end{equation*}
    \begin{equation*}
        \bigcap_{\substack{1\leq j\leq m\\c_j\neq 0}}\ker L_j=0.
    \end{equation*}
    If $n=1$ or $m=1$, then $\operatorname{Wit}(\bf c,\bf L)\neq\emptyset$. Otherwise, we take $n,m\geq 2$. If $\bf c\in\{0,1\}^m$, then $\operatorname{Wit}_3(\bf c,\bf L)\neq\emptyset$, so we may assume that $c_j\in(0,1)$ for at least one $j$.

    Take now $\bf c\in[0,1]^m\setminus\{0,1\}^m$. First, we handle the case that $\bf c\in\partial^*\cal P(\bf L)$. By Proposition~\ref{prop:christ_crit_exist}, there is a critical subspace $\bf 0\neq W\subsetneq V$. We would like to show that
    \begin{equation*}
        \operatorname{BL}\big(\left.\bf L\right|_{W},\bf c\big),\operatorname{BL}\big(\left.\bf L\right|_{V/W},\bf c\big)<+\infty.
    \end{equation*}
    We consider the first. For any $f_1,\ldots,f_m$ nonnegative and finitely-supported on the $L_jW$,
    \begin{equation*}
        \sum_{w\in W}\prod_{j=1}^mf_j(L_jw)=\sum_{x\in V}\prod_{j=1}^mf_j(L_jx)\leq\operatorname{BL}(\bf L,\bf c)\prod_{j=1}^m\|f_j\|_{\ell^{1/c_j}},
    \end{equation*}
    so that $\operatorname{BL}\big(\left.\bf L\right|_{W},\bf c\big)\leq\operatorname{BL}(\bf L,\bf c)<+\infty$. Similarly, if $f_1,\ldots,f_m$ are nonnegative and finitely-supported on the $L_jV/L_jW$, then for $\cal E$ a complete set of representatives of $V/W$ in $V$,
    \begin{equation*}
        \sum_{a\in V/W}\prod_{j=1}^mf_j(L_ja+L_jW)=\sum_{x\in\cal E}\prod_{j=1}^mf_j(L_jx)\leq\operatorname{BL}(\bf L,\bf c)\prod_{j=1}^m\|f_j\|_{\ell^{1/c_j}},
    \end{equation*}
    so that $\operatorname{BL}\big(\left.\bf L\right|_{V/W},\bf c\big)\leq\operatorname{BL}(\bf L,\bf c)<+\infty$.
    
    By induction on $n$, we may conclude that $\operatorname{Wit}\big(\bf c,\left.\bf L\right|_W\big)\neq\emptyset$ and $\operatorname{Wit}\big(\bf c,\left.\bf L\right|_{V/W}\big)\neq\emptyset$. Thus, by the definition of $\operatorname{Wit}_4(\bf c,\bf L)$, we have that $\operatorname{Wit}_4(\bf c,\bf L)\neq\emptyset$.

    Finally, we handle the case that $\bf c\not\in\partial^*\cal P(\bf L)$. We may assume that the result is true for each element of $\partial^*\cal P(\bf L)$. Let $\cal A_1,\ldots,\cal A_N$ be a corresponding list of witnesses, and suppose that $\cal A_0$ is a finite set of subspaces of $V$ such that $\operatorname{Poly}(\bf L,\cal A_0)=\cal P(\bf L)$. Then $\cal F:=\cal A_0\cup\cal A_1\cup\cdots\cup\cal A_N\in\operatorname{Folio}(\bf L)$: clearly $\operatorname{Poly}(\bf L,\cal F)=\cal P(\bf L)$, so $\partial^*\operatorname{Poly}(\bf L,\cal F)=\partial^*\cal P(\bf L)$, so at every extreme point of $\operatorname{Poly}(\bf L,\cal F)$ we have a subset of $\cal F$ which is a witness. Furthermore, since $\bf c\in\cal P(\bf L)$, we have $\bf c\in\operatorname{Poly}(\bf L,\cal F)$. Thus, $\cal F\in\operatorname{Wit}_5(\bf c,\bf L)$, and we are done.

\end{proof}

\begin{corollary}[Folios exist over $\Q$]\label{cor:fol_exist}
    Let $\F$ be characteristic zero, and suppose that $\bf c\in[0,1]^m$ is such that $n=\sum_jc_j\mathrm{rank}(L_j)$ and $\operatorname{BL}(\bf L,\bf c)<+\infty$. Then $\operatorname{Folio}(\bf L)\neq\emptyset$.
\end{corollary}
\begin{proof}
    By Theorem~\ref{thm:christ_disc_suff}, under the first two hypotheses on $\bf c$, we have $\bf c\in\cal P(\bf L)$ if and only if $\operatorname{BL}(\bf L,\bf c)<+\infty$. So, by our assumption, $\cal P(\bf L)$ is nonempty. Then $\partial^*\cal P(\bf L)$ is a finite set, and for each $\bf c\in\partial^*\cal P(\bf L)$ we have that $\operatorname{BL}(\bf L,\bf c)<+\infty$, and hence $\operatorname{Wit}(\bf c,\bf L)$ is nonempty. Let $\cal A_1,\ldots,\cal A_N$ be a choice of witness for each $\bf c\in\partial^*\cal P(\bf L)$, $\cal A_0$ be a set of subspaces such that $\cal P(\bf L)=\operatorname{Poly}(\bf L,\cal A_0)$, and choose $\cal A=\cal A_0\cup\cal A_1\cup\cdots\cup\cal A_N$.

    We claim that $\cal A\in\operatorname{Folio}(\bf L)$. Let $\bf c\in\partial^*\operatorname{Poly}(\bf L,\cal A)$ be arbitrary. Since $\operatorname{Poly}(\bf L,\cal A_0)=\cal P(\bf L)$ and $\operatorname{Poly}(\bf L,\cal A_i)\supseteq\cal P(\bf L)$ for every $1\leq i\leq N$, it follows that $\operatorname{Poly}(\bf L,\cal A)=\cal P(\bf L)$; thus, $\bf c\in\partial^*\cal P(\bf L)$ as well. Then $\cal A_i\in\operatorname{Wit}(\bf c,\bf L)$ for some $i$; since $\cal A_i\subseteq\cal A$, we are done.
\end{proof}

\begin{corollary}[Folios characterize boundedness over $\Q$]\label{cor:fol_poly}
    Let $\F$ be characteristic zero, and suppose $\cal A\in\operatorname{Folio}(\bf L)$. Then $\cal P(\bf L)=\operatorname{Poly}(\bf L,\cal A)$.
\end{corollary}
\begin{proof}
    Note that $\cal P(\bf L)\subseteq\operatorname{Poly}(\bf L,\cal A)$ is trivial. It suffices to show that $\partial^*\operatorname{Poly}(\bf L,\cal A)\subseteq\cal P(\bf L)$. For each $\bf c\in\partial^*\operatorname{Poly}(\bf L,\cal A)$, there is a witness $\cal B\subseteq\cal A$ for $(\bf c,\bf L)$, so $\operatorname{BL}(\bf L,\bf c)=1$. By Theorem~\ref{thm:christ_disc_suff}, $\bf c\in\cal P(\bf L)$, and we are done.
\end{proof}

\subsection{Applying the transference of folios}\label{subsec:transfer}

We may now show our main transference result.

\begin{theorem}[Finding a good prime]\label{thm:rank_transfer}
    Suppose $\bf L=(L_j)_{1\leq j\leq m}$, $L_j\in\operatorname{Mat}_{n_j\times n}(\Z)$, with $\mathrm{rank}(L_j)=n_j$. Suppose $\alpha\geq 0$ is arbitrary. Write $\cal P$ for the set of $(q_1^{-1},\ldots,q_m^{-1})\in[0,1]^m$ such that we have the rank and scaling conditions
    \begin{equation*}
        n=\sum_{j=1}^m\frac{n_j}{q_j},
    \end{equation*}
    \begin{equation*}
        \mathrm{rank}(H)\leq-\alpha+\sum_{j=1}^m\frac{\mathrm{rank}(L_jH)}{q_j},\quad\forall H\leq\Z^n\text{  subgroup of $0<\mathrm{rank}(H)\leq n-1$}.
    \end{equation*}
    Then there is a positive integer $\Delta=\Delta(\bf L)\in\N$, which can be determined algorithmically from $\bf L$, such that the following holds. For every prime number $p$ such that $p\nmid\Delta$, we have that the $\F_p$-linear maps $\widetilde{L}_j=L_j\,\,\text{mod $p$}$ satisfy
    \begin{equation*}
        \mathrm{dim}_{\F_p}\big(\widetilde{L}_j\F_p^n\big)=n_j,\,\,\forall j,
    \end{equation*}
    \begin{equation*}
        \mathrm{dim}_{\F_p}(V)\leq -\frac{\alpha}{n-1}+\sum_{j=1}^m\frac{\dim_{\F_p}\big(\widetilde{L}_jV\big)}{q_j},\quad \forall \bf 0\neq V\subsetneq\F_p^n\,\,\text{strict subspace},
    \end{equation*}
    for all $(q_1^{-1},\ldots,q_m^{-1})\in\cal P$.
\end{theorem}
\begin{proof}
    We do not focus on the algorithmic aspect;~\cite{christ2015discrete} shows that the polytope $\cal P$ may be determined by an algorithm, and our constructions of witnesses and folios is intended to mimic their analysis.
    
    We first suppose that $\alpha=0$. Write $L_j$ also for the induced maps $\Q^n\to\Q^{n_j}$; note that they are surjective. If $V\subseteq\Q^n$ is a $\Q$-subspace of dimension $k$, then (writing $u_1,\ldots,u_k$ for a basis for $V$ whose entries are integers)
    \begin{equation*}
        \dim(L_jV)=\mathrm{rank}\begin{bmatrix}
            \,L_ju_1 & \cdots & L_ju_k\,
        \end{bmatrix}=\mathrm{rank}\big(L_j\big(\Z\langle u_1,\ldots,u_k\rangle\big)\big),
    \end{equation*}
    where $\Z\langle u_1,\ldots,u_k\rangle\leq\Z^n$ is the subgroup of rank $k$ spanned by the $u_i$. Thus, the vector $\bf c:=(q_1^{-1},\ldots,q_m^{-1})\in\cal P(\bf L)$, where in the latter we use the field $\F=\Q$. Thus $\operatorname{Wit}(\bf c,\bf L)\neq\emptyset$ by Theorem~\ref{thm:christ_disc_suff} and Corollary~\ref{cor:fol_exist}, so $\operatorname{Folio}(\bf L)\neq\emptyset$. Choose $\cal A\in\operatorname{Folio}(\bf L)$. By Corollary~\ref{cor:fol_poly}, $\cal P(\bf L)=\operatorname{Poly}(\bf L,\cal A)$.

    Write $\cal A=\{V_\tau\}_\tau$; we will freely assume that $\beta\neq\tau$ implies $V_\beta\neq V_\tau$. For each $\tau$, write $u_{1,\tau},\ldots,u_{k_\tau,\tau}$ for a $\Q$-basis for the space $V_\tau$ whose entries are integers. Write
    \begin{equation*}
        M_\tau=\begin{bmatrix}
            \,u_{1,\tau} & \cdots & u_{k_\tau,\tau}\,
        \end{bmatrix},\quad M_{j,\tau}=L_j.M_\tau.
    \end{equation*}
    For each $\tau$ and each $j$, write $k_{j,\tau}$ for $\dim_{\Q}(L_jV_\tau)$; thus, $\mathrm{rank}(M_{j,\tau})=k_{j,\tau}$. Similarly, we have written $k_\tau=\dim_\Q(V_\tau)=\mathrm{rank}\big([\,u_{1,\tau}\,\cdots\, u_{k_\tau,\tau}\,]\big)$. Choose $d_\tau$ a nonzero $(k_\tau\times k_\tau)$-minor of $M_\tau$, $d_{j,\tau}$ a nonzero $(k_{j,\tau}\times k_{j,\tau})$-minor of $M_{j,\tau}$. Write
    \begin{equation*}
        d=\mathrm{lcm}_{j,\tau}(d_{j,\tau}),\quad f=\mathrm{lcm}_\tau(d_\tau).
    \end{equation*}
    Next, since $\mathrm{rank}(L_j)=n_j$, there is a nonvanishing $(n_j\times n_j)$-minor $\delta_j\in\Z$; we write
    \begin{equation*}
        \delta=\mathrm{lcm}_j(\delta_j).
    \end{equation*}
    Finally, for each $J\subseteq\{1,\ldots,m\}$, we choose a $\Q$-basis $\{w_{\iota,J}\}_{1\leq\iota\leq \kappa_J}$ with integer entries for the $\Q$-subspace $\bigcap_{j\in J}\ker L_j$. Choose a nonvanishing $(\kappa_J\times \kappa_J)$-minor $\rho_J$ of the matrix $[\,w_{1,J}\,\cdots\, w_{\kappa_J,J}\,]$, and write
    \begin{equation*}
        \rho=\mathrm{lcm}_{J}(\rho_J).
    \end{equation*}
    Next, consider any containment $V_\tau\subsetneq V_\beta$. Then each $u_{i,\tau}\in V_\beta$, which is to say that 
    \begin{equation*}
        u_{i,\tau}=\sum_{\ell=1}^{k_\beta}c_{i\ell}^{\tau\leftarrow\beta}u_{\ell,\beta}
    \end{equation*}
    for some choices of $c_{i\ell}^{\tau\leftarrow\beta}\in\Q$. If we write $\tau\prec\beta$ exactly when $V_\tau\subsetneq V_\beta$, then we write $Q$ for the least member of $\N$ such that $Qc_{i\ell}^{\tau\leftarrow\beta}\in\Z$ for all $\tau\prec\beta$, all $1\leq i\leq k_\tau$, and all $1\leq\ell\leq k_\beta$. It follows that $Q \Z\langle u_{1,\tau},\ldots,u_{k_\tau,\tau}\rangle\subseteq\Z\langle u_{1,\beta},\ldots,u_{k_\beta,\beta}\rangle$, whenever $V_\tau\subseteq V_\beta$.
    
    Finally, we write
    \begin{equation*}
        \Delta=\mathrm{lcm}(d,f,\delta,\rho,Q).
    \end{equation*}
    
    We choose $p$ to be any prime such that $p\nmid \Delta$. Write $\F=\Q,\F'=\F_p$, $V'=\F_p^n$,
    \begin{equation*}
        \phi(V_\tau)=V_\tau'\quad\text{where}\quad V_\tau'=\Z\langle u_{1,\tau},\ldots,u_{k_\tau,\tau}\rangle+p\Z^n\subseteq(\Z/p\Z)^n=\F_p^n,
    \end{equation*}
    and $\widetilde L_j$ for the reduction of $L_j$ mod $p$. We claim that the hypotheses of Proposition~\ref{prop:fol_transfer} hold true. First, note that, if $V_\tau\subseteq V_\beta$, we have $Qu_{i,\tau}\in\Z\langle u_{1,\beta},\ldots, u_{k_\beta,\beta}\rangle$ for each $i$, so that $Qu_{i,\tau}+p\Z^n\in\phi(V_\beta)$. Since $p\nmid Q$, we conclude that $u_{i,\tau}+p\Z^n\in\phi(V_\beta)$. Since $\phi(V_\tau)$ is spanned by the $u_{i,\tau}+p\Z^n$, we conclude that $\phi(V_\tau)\subseteq\phi(V_\beta)$.
    
    Next, for each $\tau$, since $p\nmid f$, we have $p\nmid d_\tau$, so the $\F_p$-valued matrix
    \begin{equation*}
        \begin{bmatrix}
            \,u_{1,\tau}+p\Z^n & \cdots & u_{k_\tau,\tau}+p\Z^n\,
        \end{bmatrix}
    \end{equation*}
    has a nonvanishing minor in $\F_p$, hence $V_\tau'$ has $\F_p$-dimension $k_\tau$. Similarly, since $p\nmid d_{j,\tau}$ for each $j$ and each $\tau$, we have that $M_{j,\tau}+p\Z^n$ has nonvanishing $(k_{j,\tau}\times k_{j,\tau})$-minor in $\F_p$, so $\widetilde L_j(V_\tau')$ has $\F_p$-dimension equal to $k_{j,\tau}=\dim_\Q(L_jV_\tau)$. Next, since $p\nmid\delta_j$ for each $j$, we see that $\widetilde L_j$ has a nonvanishing $(n_j\times n_j)$-minor in $\F_p$, so $\mathrm{dim}_{\F_p}(\widetilde L_j\F_p^n)=\dim_\Q(L_j\Q^n)$. Finally, for each subset $J\subseteq\{1,\ldots,m\}$, since $p\nmid\rho_J$ we see that $[\,w_{1,J}\,\cdots\,w_{\kappa_J,J}\,]$ has a nonvanishing $(\kappa_J\times\kappa_J)$-minor in $\F_p$, so
    \begin{equation*}
        \dim_{\F_p}\left(\bigcap_{j\in J}\ker L_j'\right)=\kappa_J=\dim_\Q\left(\bigcap_{j\in J}\ker L_j\right).
    \end{equation*}

    Thus, we are in a position to apply Proposition~\ref{prop:fol_transfer}. Since $\cal A=\{V_\tau\}_\tau$ is a folio for $\bf L$, we conclude that $\cal B=\{\phi(V_\tau)\}_\tau$ is a folio for $\widetilde{\bf L}$. For any $\bf c^*\in\partial^*\mathrm{Poly}(\widetilde{\bf L},\cal B)$, there is some $\cal B'\subseteq\cal B$ for which $\cal B'\in\operatorname{Wit}(\bf c,\widetilde{\bf L})$, and hence (by Proposition~\ref{prop:wit_imp_bdd}) $\operatorname{BL}(\widetilde{\bf L},\bf c^*)=1$. On the other hand, we certainly have $\operatorname{Poly}(\widetilde{\bf L},\cal B)=\operatorname{Poly}(\bf L,\cal B)$, so $\partial^*\operatorname{Poly}(\widetilde{\bf L},\cal B')=\partial^*\operatorname{Poly}(\bf L,\cal B)$, while by Corollary~\ref{cor:fol_poly} we have $\cal P(\bf L)=\operatorname{Poly}(\bf L,\cal B)$. In particular, since $\bf c\in\cal P(\bf L)$, we have $\bf c\in\mathrm{conv}\big(\partial^*\operatorname{Poly}(\widetilde{\bf L},\cal B')\big)$. By multilinear interpolation, we must have $\operatorname{BL}(\widetilde{\bf L},\bf c)=1$. Finally, testing $\operatorname{BL}(\widetilde{\bf L},\bf c)$ with functions of the form $1_{\widetilde L_jV}$, the rank condition immediately follows.

    Finally, we indicate how to get to the $\alpha>0$ case. Suppose $p$ is a prime such that we have the $\F_p$-rank conditions corresponding to $\alpha=0$. If $\widetilde{\cal P}$ denotes the polytope defined by the full set of rank and scaling conditions on the $\widetilde{L}_j$ (with $\alpha=0$), then we have demonstrated that $\cal P\subseteq\widetilde{\cal P}$. Thus, for each $\bf 0\neq V\subsetneq\F_p^n$ subspace, the tuple
    \begin{equation*}
        \big(\dim_{\F_p}(V),\dim_{\F_p}(\widetilde{L}_1V),\ldots,\dim_{\F_p}(\widetilde{L}_mV)\big)=(k,k_1,\ldots,k_m),
    \end{equation*}
    has the property that
    \begin{equation*}
        (k,k_1,\ldots,k_m)\cdot(1,-q_1^{-1},\ldots,-q_m^{-1})\leq 0,\quad\forall (q_1^{-1},\ldots,q_m^{-1})\in\cal P.
    \end{equation*}
    On the other hand, $\cal P$ is characterized by a finite set of rank and scaling conditions; that is, there are finitely many strict subspaces $\bf 0\neq V_\tau\subsetneq\Q^n$ such that
    \begin{equation*}
        (q_1^{-1},\ldots,q_m^{-1})\in\cal P\quad\iff\quad\sum_{j=1}^m\frac{n_j}{q_j}=n\,\,\text{and}\,\,\dim V_\tau\leq\sum_{j=1}^m\frac{\dim(L_jV_\tau)}{q_j},\quad\forall\tau.
    \end{equation*}
    Thus, if $\mathbb M$ is the matrix whose $\tau$th row is $(\dim V_\tau,\dim(L_1V_\tau),\ldots,\dim(L_mV_\tau))$, we have
    \begin{equation*}
        (q_1^{-1},\ldots,q_m^{-1})\in\cal P\quad\iff\quad\sum_{j=1}^m\frac{n_j}{q_j}=n\,\,\text{and}\,\,\mathbb M.\begin{bmatrix}
            1\\-q_1^{-1}\\\vdots\\-q_m^{-1}
        \end{bmatrix}\leq 0,
    \end{equation*}
    where ``$\bf v\leq 0$'' means $\bf v_\tau\leq 0$ for all $\tau$. By Farkas' lemma (\cite{rockafellar1997convex} Corollary 22.3.1), $(k,k_1,\ldots,k_m)$ is a nonnegative linear combination of the rows of $\mathbb M$:
    \begin{equation*}
        \exists \lambda_1,\ldots,\lambda_\tau\in\R_{\geq 0}\,\,\mathrm{s.t.}\,\,\dim_{\F_p}(V)=\sum_\tau\lambda_\tau\dim(V_\tau)\quad\text{and}\quad\dim_{\F_p}(\widetilde{L}_jV)=\sum_\tau\lambda_\tau\dim(L_jV_\tau).
    \end{equation*}
    Finally, since for our particular $(q_1^{-1},\ldots,q_m^{-1})$ we have
    \begin{equation*}
        \dim V_\tau\leq-\alpha+\sum_{j=1}^m\frac{\dim(L_jV_\tau)}{q_j},\quad\forall\tau,
    \end{equation*}
    we may multiply through by $\lambda_\tau$ and sum to obtain
    \begin{equation*}
        \dim_{\F_p}(V)\leq(-\alpha)\Big(\sum_\tau\lambda_\tau\Big)+\sum_{j=1}^m\frac{\dim_{\F_p}(\widetilde{L}_jV)}{q_j}.
    \end{equation*}
    Since $1\leq\dim_{\F_p}(V)=\sum_\tau\lambda_\tau\dim(V_\tau)\leq(n-1)\sum_\tau\lambda_\tau$, we conclude that
    \begin{equation*}
        \dim_{\F_p}(V)\leq-\frac{\alpha}{n-1}+\sum_{j=1}^m\frac{\dim_{\F_p}(\widetilde{L}_jV)}{q_j},
    \end{equation*}
    as was to be shown.
\end{proof}

\section{Nonlinear H\"older--Brascamp--Lieb inequalities over the integers: sufficient conditions}\label{sec:nonlin_suff}

In this section, we prove our main result for nonlinear Brascamp--Lieb inequalities over the integers, Theorem~\ref{thm:int_by_padic}. We re-state the result for convenience.

\begin{theorem}[Discrete nonlinear Brascamp--Lieb inequality, local version]\label{thm:local_integer_nonlinear_bl}
    Suppose $\bf P=(P_j)_{1\leq j\leq m}$ is a vector of polynomial maps with integer coefficients, $P_j:\Z^n\to\Z^{n_j}$. Suppose $1\leq q_1,\ldots,q_m\leq+\infty$ and $a\in\Z^n$ are such that we have the submersion, scaling, and rank conditions
    \begin{equation*}
        \mathrm{rank}(dP_j(a))=n_j,\quad 1\leq j\leq m,
    \end{equation*}
    \begin{equation*}
        n=\sum_{j=1}^m\frac{n_j}{q_j},
    \end{equation*}
    \begin{equation*}
        \mathrm{rank}(V)\leq\sum_{j=1}^m\frac{\mathrm{rank}\big(dP_j(a)(V)\big)}{q_j},\quad\forall V\leq\Z^n.
    \end{equation*}
    Then there is a positive integer $\Delta_a$ such that the following holds. If $p\nmid\Delta_a$, then for any $(f_j)_{1\leq j\leq m}$, $f_j:\Z^{n_j}\to\C$, we have
    \begin{equation*}
        \left|\sum_{x\in a+p\Z^n}\prod_{j=1}^mf_j(P_j(x))\right|\leq\prod_{j=1}^m\|f_j\|_{\ell^{q_j}(\Z^{n_j})}.
    \end{equation*}
    Moreover, $\Delta_a$ may be determined algorithmically from $d\bf P(a)$.
\end{theorem}
\begin{proof}
    Let $\Delta_1=\Delta(d\bf P(a))$ from Theorem~\ref{thm:rank_transfer}. Since each $dP_j(a)$ is full rank, there is a nonvanishing $n_j\times n_j$ minor $d_j\in\Z$. Let $\Delta_2=\mathrm{lcm}(d_1,\ldots,d_m)$, and $\Delta_a=\mathrm{lcm}(\Delta_1,\Delta_2)$. Then, for any $u\in a+\Z_p^n$, we have for each $j$,
    \begin{equation*}
        dP_j(u)\in dP_j(a)+p\mathrm{Mat}_{n_j\times n}(\Z_p),
    \end{equation*}
    because the $P_j$ are polynomials with integer coefficients. In particular, $\widetilde{dP_j}(u)=\widetilde{dP_j}(a)$, where the tilde denotes reducing mod $p\Z_p^n$ to obtain an $\F_p$-linear map from $\F_p^n$ to $\F_p^{n_j}$. By Theorem~\ref{thm:rank_transfer}, we have
    \begin{equation*}
        \mathrm{dim}_{\F_p}\big(\widetilde{dP_j}(u)(\F_p^n)\big)=n_j,\,\,\forall j,
    \end{equation*}
    \begin{equation*}
        \mathrm{dim}_{\F_p}(V)\leq\sum_{j=1}^m\frac{\dim_{\F_p}\big(\widetilde{dP_j}(u)(V)\big)}{q_j},\quad \forall V\subseteq\F_p^n\,\,\text{strict subspace}.
    \end{equation*}
    We demonstrate that the compression factor $\Phi_0(V)=1$ for all $\Q_p$-subspaces $V\subsetneq\Q_p^n$. To prove this, write $k=\dim V$. Let $U=[u_1,\ldots,u_k]$ be an $n\times k$ matrix whose columns form an isometric basis for $V$. Thus, $\widetilde{U}=U+p\Z_p^n$ represents a $k$-dimensional $\F_p$-subspace $V'$ of $\F_p^n$. By the previous display, we know that the quantities
    \begin{equation*}
        k_j=\dim_{\F_p}(\widetilde{dP_j}(u)(V'))
    \end{equation*}
    satisfy
    \begin{equation*}
        \dim_{\F_p}(V)\leq\sum_{j=1}^m\frac{k_j}{q_j}.
    \end{equation*}
    From the definition of the $k_j$, $dP_j(u).U$ (which has $\Z_p$ entries) possesses a $k_j\times k_j$ submatrix $R$ whose reduction modulo $p\Z_p$ is invertible over $\F_p$. Thus, $|\det R|_p=1$. It follows directly that the compression factor $\Phi_0(V)=1$.
    
    By Theorem~\ref{thm:lin_sub_extr}, $\operatorname{BL}_{\mathrm{grp}}(d\bf P(u),(q_j^{-1})_j)$ is finite and is extremized by a regular `subgroup of eccentricity $1$; that is, by $\Z_p^n$. In other words,
    \begin{equation*}
        \operatorname{BL}(d\bf P(u),(q_j^{-1})_j)=\frac{1}{\prod_{j=1}^m\mu_{n_j}(dP_j(u)(\Z_p^n))^{1/q_j}}.
    \end{equation*}
    Since $p\nmid\Delta_2$, each $dP_j(u)$ is full rank modulo $p$. Thus, $\widetilde{dP_j}(u)(\F_p^n)=\F_p^{n_j}$. It follows that $dP_j(u)(\Z_p^n)=\Z_p^{n_j}$; thus, we have the critical
    \begin{equation}\label{eq:unit_BL}
        \operatorname{BL}(d\bf P(u),(q_j^{-1})_j)=1,\quad\forall u\in a+p\Z_p^n.
    \end{equation}
    
    By Theorem~\ref{thm:nonlin_local_bdd}, there a $\delta\in\K^\times$ depending on $\bf P$ and $(q_j)_j$ such that, for any $u\in a+p\Z_p^n$,
    \begin{equation*}
        \operatorname{BL}\big(\left.\bf P\right|_{u+\delta\Z_p^n},(q_j^{-1})_j\big)\leq\operatorname{BL}(d\bf P(u),(q_j^{-1})_j)=1.
    \end{equation*}
    Combining the facts that (a) $\bf P$ is a contraction over the $p$-adics, (b) $\Z_p^n$ has eccentricity $1$, (c) the $P_j$ are submersions mod $p$ at $u$, and (d) the fact that $\widetilde{d\bf P}(u)$ is injective (via the mod $p$ rank condition), and (e) the general property of integer polynomial maps that $x,y\in u+p\Z_p^n$ implies $P_j(y)-P_j(x)-dP_j(x).(y-x)\in p^2\Z_p^{n_j}$, we conclude that it suffices to choose $\delta=p$.

    We now demonstrate the desired discrete functional bound using the non-Archimedean Brascamp--Lieb functional. Suppose first that each $f_j$ is finitely-supported and nonnegative. Choose $N\in\N$ such that, for any $1\leq j\leq m$ and any $a,b\in\mathrm{supp}(f_j)$, we have $\|a-b\|>p^{-N}$. Write then
    \begin{equation*}
        g_j=\sum_{a\in\mathrm{supp}(f_j)}f_j(a)1_{a+p^N\Z_p^{n_j}}.
    \end{equation*}
    Note that $\|g_j\|_{L^q}=p^{-Nn_j/q_j}\|f_j\|_{\ell^{q_j}}$. On the other hand,
    \begin{equation*}
        \int_{u+p\Z_p^n}\prod_{j=1}^mg_j(P_j(x))d\mu_n(x)=\sum_{a_j\in\mathrm{supp}(f_j)}\mu_{n}(\{x\in u+p\Z_p^n:P_j(x)\in a_j+p^N\Z_p^{n_j}\,\,\forall j\})\prod_{j=1}^mf_j(a_j).
    \end{equation*}
    If $x\in\Z^n$ has the property that $P_j(x)=a_j$, then $P_j(y)\in a_j+p^N\Z_p^{n_j}$ for all $y\in x+p^N\Z_p^n$. Thus,
    \begin{equation*}
        \begin{split}
        \sum_{a_j\in\mathrm{supp}(f_j)}&\mu_{n}(\{x\in u+p\Z_p^n:P_j(x)\in a_j+p^N\Z_p^{n_j}\,\,\forall j\})\prod_{j=1}^mf_j(a_j)\\
        &\geq\sum_{a_j\in\mathrm{supp}(f_j)}\sum_{\substack{x\in u+p\Z^n\\P_j(x)=a_j\,\,\forall j}}p^{-Nn}\prod_{j=1}^mf_j(a_j)=p^{-Nn}\sum_{x\in u+p\Z^n}\prod_{j=1}^mf_j(P_j(x)).
        \end{split}
    \end{equation*}
    Applying~\eqref{eq:unit_BL}, we conclude that
    \begin{equation*}
        \begin{split}
            \sum_{x\in u+p\Z^n}\prod_{j=1}^mf_j(P_j(x))&\leq p^{Nn}\int_{u+p\Z_p^n}\prod_{j=1}^mg_j(P_j(x))d\mu_n(x)\\
            &\leq p^{Nn}\prod_{j=1}^m\|g_j\|_{L^{q_j}(\Z_p^{n_j})}\\
            &\leq\prod_{j=1}^m\|f_j\|_{\ell^{q_j}(\Z^{n_j})}.
        \end{split}
    \end{equation*}
    Finally, we may freely allow the $f_j$ to be complex valued and non-finitely-supported by standard arguments. The result follows.

\end{proof}

Note that Theorem~\ref{thm:local_integer_nonlinear_bl} is a ``local'' Brascamp--Lieb inequality; that is, it controls functionals which have been localized to certain sublattices $a+p\Z^n$. It is tempting to attempt to ``globalize'' this result, possibly at the cost of a constant. Indeed, each individual lattice $a+p\Z^n$ occupies a positive proportion of the full lattice $\Z^n$, and in many cases we will be able to find a good $p$ for every single choice of $a$. Nevertheless, our methods do not appear adequate to handle such global bounds, at the current level of generality.

\begin{example}[Enemy scenario for globalizing bounds]\label{ex:local_to_global_enemy}
    For each $a\in\Z$, choose a distinct prime $p_a$. For $N\in\N$, choose $\Omega_N\subseteq\Z$ such that $\#\Omega_N=N$ and $\#(\Omega\cap(a+p_a))\in\{0,1\}$ for all $a\in\Z$. For this $\Omega_N$, note that we have the ``local Brascamp--Lieb inequality:''
    \begin{equation*}
        \#(\Omega_N\cap (a+p_a\Z))\leq \#(\Omega_N\cap(a+p_a\Z))^{1/2}.
    \end{equation*}
    However, it is clearly not possible to globalize: the inequality
    \begin{equation*}
        \#\Omega_N\lesssim \#\Omega_N^{1/2}
    \end{equation*}
    fails for large $N$; nor is this remedied by ``moving away from the endpoint'' to
    \begin{equation*}
        \#\Omega_N\lesssim \#\Omega_N^{c},\quad c\in(1/2,1).
    \end{equation*}
\end{example}
Thus, if the lattices used in our local bounds have distinct primes, no global bound may be possible. On the other hand, by the Chinese remainder theorem, the only way to cover $\Z^n$ by finitely-many sets of the form $a+p\Z^n$ is to at least have one prime $p$ for which all cosets $\Z^n/p\Z^n$ are represented in the corresponding $a$'s. Thus, a version of Theorem~\ref{thm:local_integer_nonlinear_bl} would be needed where a single prime $p$ has the property that $p\nmid\Delta_a$, for all $a\in\Z^n$. Unfortunately, just from the assumption about integer points, it does not appear possible to make the corresponding conclusion about $\F_p$-points; this would roughly has the structure of a statement of the form ``if the Diophantine equation has no $\Z$-solution, then it has no $\F_p$-solution for large enough $p$'' -- which is generally valid only in the r\'egime where the Hasse principle succeeds. Since the Hasse principle only holds for certain sorts of equations, it seems unlikely that Theorem~\ref{thm:local_integer_nonlinear_bl} can be globalized in its full generality without different methods.

We now turn to some additional hypotheses which would suffice to prove a global bound for the discrete nonlinear Brascamp--Lieb functional. The critical input will be the following lemma. It requires two results. The first is Hilbert's Nullstellensatz, and the second is the fact that finite extensions of fields $K/\Q$ admit ``traces,'' which here we take to be $\Q$-linear maps $\mathrm{Tr}_{K/\Q}:K\to\Q$ which fix $\Q\subseteq K$. Hilbert's Nullstellensatz can be found in any standard commutative algebra text, e.g.~\cite{rotman2010advanced,atiyah2018introduction}.

\begin{lemma}[Brascamp--Lieb Nullstellensatz]\label{lem:BL_nullstellensatz}
    Suppose that $\bf P=(P_j)_{1\leq j\leq m}$ is a vector of polynomial maps with $\Z$-coefficients, $P_j:\Z^n\to\Z^{n_j}$. Suppose $1\leq q_1,\ldots,q_m\leq+\infty$ are such that~\eqref{eq:intro_scaling} holds. Suppose further that we have the submersion and rank conditions
    \begin{equation*}
        \mathrm{rank}(dP_j(z))=n_j,\quad\forall z\in \C^n,
    \end{equation*}
    \begin{equation*}
        \dim_\C(V)\leq\sum_{j=1}^m\frac{\dim_\C(dP_j(z)(V))}{q_j},\quad\forall z\in\C^n,\,\, V\leq\C^n.
    \end{equation*}
    Then there is $\Delta=\Delta(\bf P)\in\N$ such that the following holds. If $p\in\N$ is a prime such that $p\nmid\Delta$, then the submersion and rank conditions descend to $\F_p^n$:
    \begin{equation*}
        \mathrm{rank}_{\F_p}\big(\widetilde{dP_j}(a)\big)=n_j,\quad\forall a\in\F_p^n,
    \end{equation*}
    \begin{equation*}
        \dim_{\F_p}(V)\leq\sum_{j=1}^m\frac{\dim_{\F_p}\big(\widetilde{dP_j}(a)(V)\big)}{q_j},\quad\forall a\in\F_p^n,\,\,V\leq\F_p^n.
    \end{equation*}
    Moreover, $\Delta$ may be determined by an algorithm from  $\bf P$.
\end{lemma}
\begin{proof}
    The ``effectively'' claim will not be carefully explored; the main input that is needed is an ``effective Nullstellensatz,'' which reduces the problem of finding the various $g_\iota$ and $h_{k,\bf A}^{\tau,i}$ utilized below to a finite-dimensional linear algebra problem. One adequate result for this purpose is~\cite{kollar1988sharp}, Theorem 1.1. The restriction to rational numbers (rather than computable, say) is just to ensure that inequalities of the sort $1\leq\frac{1}{q_1}+\frac{1}{q_2}$, e.g., can be tested.

    We first handle the submersion condition. Write $N=n_1+\ldots+n_m$. For $z\in\C^n$, we write
    \begin{equation*}
        \Phi(z)=\Big(\mathrm{Minor}_{J}(d\bf P(z))\Big)_{J\in\binom{[N]}{n}}.
    \end{equation*}
    That is, $\Phi(z)$ is the list of $(n\times n)$ minors of $d\bf P(z)$. In particular, $\Phi$ is a polynomial map
    \begin{equation*}
        \Phi:\C^n\to\C^{\binom{N}{n}},
    \end{equation*}
    and $\Phi$ has integer coefficients. Since $d\bf P(z)$ is always full rank, $\Phi(z)\neq 0$ for all $z\in\C^n$. Write $(f_\iota)_{1\leq\iota\leq\binom{N}{n}}$ for the columns of $\Phi$; thus, each $f_\iota$ is a polynomial in $\Z[x_1,\ldots,x_n]$. Since $\Phi$ is nonvanishing, the $f_\iota$ have no simultaneous zero. By Hilbert's Nullstellensatz (\cite{rotman2010advanced}, Theorem B-6.12), there are $(g_\iota)_{1\leq\iota\leq\binom{N}{n}}$ polynomials in $\C[x_1,\ldots,x_n]$ for which
    \begin{equation*}
        1=\sum_{\iota}f_\iota g_\iota.
    \end{equation*}
    Let $K$ be the smallest field extension of $\Q$ containing all the coefficients of the $g_\iota$. Certainly $[K:\Q]<+\infty$, and hence there exists a field trace $\mathrm{Tr}_{K/\Q}:K\to\Q$, i.e. a $\Q$-linear map from the finite-dimensional $\Q$-vector space $K$ which pointwise fixes $\Q\subseteq K$. For each $\iota$, write $h_\iota$ for the polynomial obtained by applying $\mathrm{Tr}_{K/\Q}$ to every coefficient of $g_\iota$. Thus, for any $(a_1,\ldots,a_n)\in\Q^n$ we have
    \begin{equation*}
        \mathrm{Tr}_{K/\Q}[g_\iota(a_1,\ldots,a_n)]=h_\iota(a_1,\ldots,a_n).
    \end{equation*}
    In particular, for each $a_1,\ldots,a_n\in\Q$ we have the identity
    \begin{equation*}
        1=\sum_{\iota}f_\iota(a_1,\ldots,a_n)h_\iota(a_1,\ldots,a_n).
    \end{equation*}
    Thus, we have that $1=\sum_\iota f_\iota h_\iota$ as elements of $\Q[x_1,\ldots,x_n]$. Clearing denominators, we find $\Delta\in\N$ such that
    \begin{equation*}
        \Delta'\equiv\sum_\iota f_\iota\tilde h_\iota,
    \end{equation*}
    where the $\tilde h_\iota=\Delta' h_\iota\in\Z[x_1,\ldots,x_n]$.

    Now, suppose $p\nmid \Delta'$. Then, for each $x_1,\ldots,x_n\in\Z$, there is some $\iota$ such that $f_\iota(x_1,\ldots,x_n)\neq 0$ mod $p$. In particular, $\Phi$ is non-vanishing mod $p$, for every choice of inputs $x_1,\ldots,x_n\in\Z$. Thus, for any $p\nmid\Delta'$ and any $a\in\Z^n$, the matrix $dP_j(a)$ has an $n_j\times n_j$ minor which is nonzero modulo $p$. Thus, the submersion condition descends to $\F_p$, for any such $p$.

    Next, for $1\leq k\leq n-1$, write $\mathscr F_k$ for the set of tuples $(k,k_1,\ldots,k_m)\in\Z_{\geq 0}^{m+1}$ such that $0\leq k_j\leq \min(k,n_j)$, and such that
    \begin{equation*}
        k\leq\sum_{j=1}^m\frac{k_j}{q_j}.
    \end{equation*}
    For every $z\in \C^n$ and $u_1,\ldots,u_k\in \C^n$ and $\tau\in\mathscr{F}_k$, write $\Phi_k^\tau(z;u_1,\ldots,u_k)=(\Phi_{k1A_1}\cdots\Phi_{kmA_m})_{A_1,\ldots,A_m}$, where $\Phi_{kjA_j}^\tau(z;u_1,\ldots,u_k)$ is the minor of $dP_j(z).[u_1,\ldots,u_k]$ corresponding to the $k_j\times k_j$ submatrix $A_j$, and we range over all possible choices of submatrix for all $1\leq j\leq m$. Thus, $\Phi_k^\tau=0$ exactly when $\mathrm{rank}(dP_j(z).[u_1,\ldots,u_k])<k_j$ for some $j$. Thus, the vector $\Phi_k=(\Phi_k^\tau)_{\tau\in\mathscr{F}_k}$ vanishes exactly when the rank condition
    \begin{equation*}
        k\leq\sum_{j=1}^m\frac{\mathrm{rank}\big(dP_j(z).[u_1,\ldots,u_k]\big)}{q_j}
    \end{equation*}
    is violated. Note carefully that we have not assumed that the $u_1,\ldots,u_k$ are linearly independent; thus, we will need to relate the vanishing of $\Phi_k$ to the linear dependence of the $u_1,\ldots,u_k$. As such, we define $\Psi(u_1,\ldots,u_k)$ to be the vector of $k\times k$ minors of $[u_1,\ldots,u_k]$. It follows that, by the rank condition, $\Phi=(\Phi_k^\tau)_{\substack{1\leq k\leq n-1\\\tau\in\mathscr{F}_k}}(z;u)$ vanishes only when $\Psi(u)$ vanishes.

    Thus, we are in a position to apply Hilbert's Nullstellensatz. Because the entries of $\Phi$ have a simultaneous zero only when the entries of $\Psi$ have a simultaneous zero, we have for each $i$, there are $h_{k,\bf A}^{\tau,i}\in\C[z;u]$ and $r_i\in\N$ such that
    \begin{equation*}
        \Psi_i^{r_i}=\sum_{\substack{\tau\\\bf A=(A_1,\ldots,A_m)}}\Phi_{k1A_1}\cdots\Phi_{kmA_m}h_{k,\bf A}^{\tau,i}.
    \end{equation*}
    Write now $K$ for the smallest field extension of $\Q$ which contains all the coefficients of the (finitely-many!) $h_{k,\bf A}^{\tau,i}$. Thus, $[K:\Q]<+\infty$, and hence there is a field trace $\mathrm{Tr}_{K/\Q}$, as before. Write $\Gamma^{\tau,i}_{k,\bf A}\in\Q[z;u]$ for the polynomial obtained by applying $\mathrm{Tr}_{K/\Q}$ to each coefficient of $h_{k,\bf A}^{\tau,i}$. It follows, as before, that
    \begin{equation*}
        \Psi_i^{r_i}=\sum_{\substack{\tau\\\bf A=(A_1,\ldots,A_m)}}\Phi_{k1A_1}\cdots\Phi_{kmA_m}\Gamma_{k,\bf A}^{\tau,i}.
    \end{equation*}
    Clearing denominators, we find $\Delta_{ik}\in\Z\setminus\{0\}$ such that
    \begin{equation*}
        \Delta_{ik}\Psi_i^{r_i}=\sum_{\substack{\tau\\\bf A=(A_1,\ldots,A_m)}}\Phi_{k1A_1}\cdots\Phi_{kmA_m}\tilde\Gamma_{k,\bf A}^{\tau,i},
    \end{equation*}
    where $\tilde\Gamma_{k,\bf A}^{\tau,i}=\Delta_{ik}\Gamma_{k,\bf A}^{\tau,i}$ may now be assumed to belong to $\Z[z;u]$.

    Finally, suppose $p$ is a prime such that $p\nmid\Delta_{ik}$ for all $i,k$. Then, if $z\in\Z^n$ and $u_1,\ldots,u_k\in\Z^n$ represent a $k$-dimensional subspace in $\F_p^n$, we get that $\Delta_{ik}\Psi_i^{r_i}\not\equiv 0$ mod $p$, and hence
    \begin{equation*}
        \Phi_{k1A_1}\cdots\Phi_{kmA_m}(z;u)\not\equiv 0\quad\text{mod $p$},
    \end{equation*}
    for some choice of $\tau$ and $A_1,\ldots,A_m$. In particular, for each choice of $z$ and $u_1,\ldots,u_k$ as above, we get that $\Phi\not\equiv 0$ mod $p$. Put another way, at every $z\in\Z^n$, the rank conditions descend mod $p$.

    Finally, writing
    \begin{equation*}
        \Delta=\mathrm{lcm}(\Delta',\{\Delta_{ik}\}_{i,k}),
    \end{equation*}
    we see that the submersion and rank conditions descend to any $p\nmid\Delta$, as claimed.
\end{proof}

We proceed to showing global boundedness of the discrete nonlinear Brascamp--Lieb functional. One quick way to see that the hypothesis implies \textit{some} bound, would be to use the fact that $\Q_p$ embeds in $\C$, for each $p$; thus, by Theorem~\ref{thm:nonlin_local_bdd} and compactness of $\Z_p$, we can patch together finitely many local bounds to obtain a global bound. However, because of the appeal to compactness, the constant is not effective; our approach gives some information on the constant, and is suggestive of possible generalizations.

\begin{theorem}[Discrete nonlinear Brascamp--Lieb inequality, global version]\label{thm:global_integer_nonlinear_bl}
    Suppose $\bf P=(P_j)_{1\leq j\leq m}$ is a vector of polynomial maps with integer coefficients, $P_j:\Z^n\to\Z^{n_j}$. Suppose $1\leq q_1,\ldots,q_m\leq+\infty$ are such the scaling condition~\eqref{eq:intro_scaling} holds. For every $z\in\C^n$, we assume the submersion and rank conditions
    \begin{equation*}
        \mathrm{rank}(dP_j(z))=n_j,\quad 1\leq j\leq m,
    \end{equation*}
    \begin{equation*}
        \dim_\C(V)\leq\sum_{j=1}^m\frac{\dim_\C\big(dP_j(z)(V)\big)}{q_j},\quad\forall V\leq\C^n.
    \end{equation*}
    Then, for some $C\geq 1$, we have the functional inequality
    \begin{equation*}
        \left|\sum_{x\in\Z^n}\prod_{j=1}^mf_j(P_j(x))\right|\leq C\prod_{j=1}^m\|f_j\|_{\ell^{q_j}(\Z^{n_j})},
    \end{equation*}
    uniformly over $f_j\in\ell^{q_j}(\Z^{n_j})$. Moreover, the constant $C$ is effective.
\end{theorem}
\begin{proof}
    
    Let $\Delta$ be as in Lemma~\ref{lem:BL_nullstellensatz}. Fix any $p\nmid\Delta$; thus, for each $a\in\Z^n$ and each $V\leq\F_p^n$, we have the submersion and rank conditions
    \begin{equation*}
        \mathrm{rank}\big(\widetilde{dP_j}(a)\big)=n_j,\quad 1\leq j\leq m,
    \end{equation*}
    \begin{equation*}
        \dim_{\F_p}(V)\leq\sum_{j=1}^m\frac{\dim_{\F_p}\big(\widetilde{dP_j}(a)(V)\big)}{q_j}.
    \end{equation*}
    Following the argument of Theorem~\ref{thm:local_integer_nonlinear_bl}, we obtain the local bound
    \begin{equation*}
        \left|\sum_{x\in a+p\Z^n}\prod_{j=1}^mf_j(P_j(x))\right|\leq \prod_{j=1}^m\|f_j\|_{\ell^{q_j}(\Z^{n_j})},\quad\forall a\in\Z^n.
    \end{equation*}
    Summing over representatives $a$ of $\Z^n/p\Z^n$, the desired inequality follows.
    
\end{proof}

Studying the proof of Lemma~\ref{lem:BL_nullstellensatz}, one can obtain a similar global bound using only assumptions about the ideals generated by various minors of the $dP_j$. We leave open the question of the best possible formulation.

It is very tempting at this point to discuss the subject of extremizers for Theorem~\ref{thm:local_integer_nonlinear_bl}; indeed, Corollary~\ref{cor:nonlin_quant_ineq} does a great deal of heavy lifting towards such a result. The only ingredient that is missing is to provide an adequate \emph{functional} Brascamp--Lieb extremizer theory over finite fields, in the spirit of Theorem~\ref{thm:lin_sub_extr}. It might be possible to perform a modification of the heat-flow argument in~\cite{bennett2008brascamp}, but this has proved difficult; similarly, the methods in~\cite{christ2013optimal}, essentially reproduced in subsection~\ref{subsec:bl_module} above, may provide some leverage. We leave this question for future consideration.

\section{Applications}\label{sec:applications}

In this section, we recall some well-known applications and special cases of Brascamp--Lieb inequalities, and connect these to our results. In each case, we opt for the ``local'' version; one may obtain ``global'' versions with stronger assumptions, per Theorem~\ref{thm:global_integer_nonlinear_bl}.

\begin{theorem}[Subadditivity of the entropy]\label{thm:entropy}
    Suppose $P_j:\Z^n\to\Z^{n_j}$ is a polynomial map with integer coefficients, for each $1\leq j\leq m$. Suppose $a\in\Z^n$ and $c_1,\ldots,c_m\in[0,1]^m$ are such that
    \begin{equation*}
        \mathrm{rank}(dP_j(a))=n_j,\quad 1\leq j\leq m,
    \end{equation*}
    \begin{equation*}
        n=\sum_{j=1}^m c_jn_j,
    \end{equation*}
    \begin{equation*}
        \mathrm{rank}(V)\leq\sum_{j=1}^mc_j\mathrm{rank}(dP_j(a).(V)),\quad\forall V\leq\Z^n.
    \end{equation*}
    Then, for all but finitely-many primes $p$, the following holds. If $f:a+p\Z^n\to[0,1]$ is a probability density, and
    \begin{equation*}
        f_j(z)=\sum_{\substack{x\in a+p\Z^n\\P_j(x)=z}}f(x),
    \end{equation*}
    then
    \begin{equation*}
        \sum_{j=1}^mc_j\sum_{z\in\Z^{n_j}}f_j(z)\ln(f_j(z))\leq\sum_{x\in a+p\Z^n}f(x)\ln(f(x)).
    \end{equation*}
\end{theorem}
\begin{proof}
    Combine Theorem~\ref{thm:int_by_padic} with~\cite{carlen2009subadditivity}, Theorem 2.1.
\end{proof}

\begin{theorem}[Discrete nonlinear multilinear Kakeya]\label{thm:ml_kakeya}
    Suppose $P_j:\Z^n\to\Z^{n-1}$ is a polynomial map with integer coefficients, for each $1\leq j\leq n$. Suppose $a\in\Z^n$ is such that the following holds. Each of $dP_j(a)$ has rank $n-1$, and for some $\Q$-linearly independent $v_1,\ldots,v_n\in\Z^n$ we have $dP_j(a).v_j=0$. Then, for all but finitely-many primes $p$, the following holds. If $f_1,\ldots,f_n\in\ell^{n-1}(\Z^{n-1})$, then
    \begin{equation*}
        \left|\sum_{x\in a+p\Z^n}\prod_{j=1}^nf_j(P_j(x))\right|\leq\prod_{j=1}^n\|f_{j}\|_{\ell^{n-1}(\Z^{n-1})}.
    \end{equation*}
\end{theorem}
\begin{proof}
    It is easy to see that there exist $A:\Q^n\to\Q^n$ and $A_j:\Q^{n-1}\to\Q^{n-1}$ invertible such that $A_j\circ d\bf P(a)\circ A^{-1}=\pi_j$, where $\pi_j:\Q^n\to\Q^{n-1}$ forgets the $j$th coordinate. Let $\Delta\in\N$ be such that $A$ and each of the $A_j$ have entries whose denominators do not contain any factors of $\Delta$, and such that $\det(A),\det(A_j)$ has $p$-valuation $0$ for every $p\nmid\Delta$. In particular, $A,A_j$ descend to invertible matrices over $\F_p$, whenever $p\nmid\Delta$. If $p\nmid\Delta_a$ as well, where $\Delta_a$ is obtained by applying the data to Theorem~\ref{thm:int_by_padic}, then we conclude that $(\widetilde{d\bf P}(a),(\frac{1}{n-1})_{1\leq j\leq n})$ is $\F_p$-equivalent to the Loomis--Whitney data $((\pi_j)_j,(\frac{1}{n-1})_j)$ over $\F_p^n$. The result then comes by following the proof strategy of Theorem~\ref{thm:int_by_padic} above.
\end{proof}

\begin{corollary}
    Let $(P_j)_{1\leq j\leq n}$ and $a\in\Z^n$ be as in Theorem~\ref{thm:ml_kakeya}. Then, for all but finitely-many primes $p$, the following holds. If $A_1,\ldots,A_n\subseteq\Z^{n-1}$ are finite, then
    \begin{equation*}
        \#\Big((a+p\Z^n)\cap\bigcap_{j=1}^n P_j^{-1}(A_j)\Big)\leq\prod_{j=1}^n(\# A_j)^{\frac{1}{n-1}}.
    \end{equation*}
\end{corollary}
\begin{proof}
    Apply Theorem~\ref{thm:ml_kakeya} to $f_j=1_{A_j}$.
\end{proof}

\begin{theorem}[Rank-one Brascamp--Lieb]
    Let $p_1,\ldots,p_m\in\Z[x_1,\ldots,x_n]$ and $a\in\Z^n$ be such that the following holds. Write $u_j=dp_j(a)\in\Z^n$. Suppose that
    \begin{equation*}
        \bf c\in\mathrm{conv}\big\{1_I:I\subseteq\{1,\ldots,m\}\,\,\text{and}\,\,(u_i)_{i\in I}\,\,\text{are a basis for $\Q^n$}\big\}\subseteq[0,1]^m.
    \end{equation*}
    Then, for all but finitely-many primes $p$, the following holds. If $f_1,\ldots,f_m$ are finitely-supported over $\Z$, then
    \begin{equation*}
        \left|\sum_{x\in a+p\Z^n}\prod_{j=1}^mf_j(p_j(x))\right|\leq\prod_{j=1}^m\|f_j\|_{\ell^{1/c_j}(\Z)}.
    \end{equation*}
\end{theorem}
\begin{proof}
    Let $\mathscr{C}\subseteq[0,1]^m$ be the convex hull in question. If $u_j=0$, then $c_j=0$ for all $\bf c\in\mathscr C$; by zero-contraction, we may assume that $u_j\neq 0$ for all $j$, so that the $p_j$ are all submersions at $a$. For any $\bf c\in\mathscr{C}$, we have $\sum_{j=1}^m c_j=n$. If $\bf c=1_I$ with $I\subseteq\{1,\ldots,m\}$ is such that $\{u_i\}_{i\in I}$ is a basis of $\Q^n$, then for any subspace $V\subseteq \Q^n$ of dimension $0<k<n$, there are at most $n-k$-many $i\in I$ such that $u_i$ is orthogonal to $V$. Thus,
    \begin{equation*}
        \dim V\leq \sum_{i\in I}\dim(dp_i(a).(V))=\sum_{j=1}^mc_j\dim(dp_j(a).(V)).
    \end{equation*}
    Finally, since the polytope $\cal P$ cut out by the rank and scaling conditions is convex, we conclude that $\mathscr{C}\subseteq\cal P$. The result then follows from Theorem~\ref{thm:local_integer_nonlinear_bl}.
\end{proof}

For an integrable function $w:(\R/\Z)^k$ and $a\in\Z^n$, we write $$\hat{w}(a)=\int_{(\R/\Z)^k}w(x)\exp(-2\pi ia\cdot x)dx.$$ For a summable function $\eta:\Z^n\to\C$ and $\xi\in(\R/\Z)^n$, we write $$\eta^\vee(\xi)=\sum_{x\in\Z^n}\eta(x)e(x\cdot\xi).$$

The following inequality may be useful for interpolating against decay-type bounds on oscillatory integrals that use a weaker Lebesgue exponent; c.f.~\cite{bennett2010finite}, Proposition 3.1.
\begin{theorem}[Weighted restriction estimate]
    Let $P_j:\Z^n\to\Z^{n_j}$ be a polynomial map with integer coefficients, $1\leq q_1,\ldots,q_m\leq+\infty$ be a tuple of exponents, $a\in\Z^n$ such that $\mathrm{rank}(dP_j(a))=n_j$ for all $j$,
    \begin{equation*}
        n=\sum_{j=1}^m\frac{n_j}{q_j},
    \end{equation*}
    \begin{equation*}
        \mathrm{rank}(V)\leq\sum_{j=1}^m\frac{\mathrm{rank}(dP_j(a).(V))}{q_j},\quad\forall V\leq \Z^n.
    \end{equation*}
    Then, for all but finitely-many primes $p$, the following holds. Suppose $c:a+p\Z^n\to\C$ is finitely-supported. Then, for any weights $w_j\in L^{q_j}((\R/\Z)^{n_j})$, the exponential sum
    \begin{equation*}
        g(\xi_1,\ldots,\xi_m)=\sum_{x\in a+p\Z^n}c(x)\exp(2\pi i(\xi_1 P_1(x)+\ldots+\xi_m P_m(x)))
    \end{equation*}
    satisfies
    \begin{equation*}
        \left|\int_{(\R/\Z)^{n_1}\times\cdots\times(\R/\Z)^{n_m}}\Big[\prod_{j=1}^mw_j(\xi_j)\Big]g(\xi_1,\ldots,\xi_m)d\xi_1\wedge\cdots\wedge d\xi_m\right|\leq\|c\|_\infty\cdot\prod_{j=1}^m\|\hat w_j\|_{\ell^{q_j}(\Z^{n_j})}.
    \end{equation*}
\end{theorem}

\begin{proof}
    Write $\mu=\sum_{x\in a+p\Z^n}c(x)\delta_{\bf P(x)}$. Then
    \begin{equation*}
        \begin{split}
            \sum_{x\in a+p\Z^n}\bar c(x)\prod_{j=1}^m\widehat{\bar w}_j(P_j(x))&=\langle \widehat{\bar w}_1\otimes\cdots\otimes \widehat{\bar w}_m,\mu\rangle_{\Z^{n_1}\times\cdots\times\Z^{n_m}}\\
            &=\langle \bar w_1\otimes\cdots\otimes \bar w_m,\mu^\vee\rangle_{(\R/\Z)^{n_1}\times\cdots\times(\R/\Z)^{n_m}}\\
            &=\int_{(\R/\Z)^{n_1}\times\cdots\times(\R/\Z)^{n_m}}\Big[\prod_{j=1}^m \bar w_j(\xi_j)\Big]\overline{\mu^\vee}(\xi_1,\ldots,\xi_m)d\xi_1\wedge\cdots\wedge d\xi_m,
        \end{split}
    \end{equation*}
    whereas
    \begin{equation*}
        \mu^\vee(\xi_1,\ldots,\xi_m)=g(\xi_1,\ldots,\xi_m).
    \end{equation*}
    Thus, applying Theorem~\ref{thm:int_by_padic},
    \begin{equation*}
        \left|\int_{\mathbb T^{n_1}\times\cdots\times\mathbb{T}^{n_m}}\Big[\prod_{j=1}^m w_j(\xi_j)\Big]g(\xi_1,\ldots,\xi_m)d\xi_1\wedge\cdots\wedge d\xi_m\right|\leq\prod_{j=1}^m\big\|\hat w_j\big\|_{\ell^{q_j}(\Z^{n_j})}.
    \end{equation*}
\end{proof}

  \section{Appendix}\label{sec:appendix}

  This appendix serves to provide the basic technical theory of $C^k$ maps in the context of non-Archimedean local fields. Our definition is a strict generalization of the one in~\cite{schikhof1985}; the latter is primarily concerned with one-variable functions, and only briefly touches on the multivariate theory. Our definitions and results are equivalent to those of~\cite{schikhof1985} in the univariate case.

  We adopt the following combinatorial notations, when $n\in\N$ and $\alpha,\beta\in\Z_{\geq 0}^n$ and $x\in\K^n$.
  \begin{itemize}
    \item $[n]=\{1,\ldots,n\}$.
    \item $\bf e_i=(0,\ldots,0,1,0,\ldots,0)$ is the $i$th standard basis vector in $\Z^n$.
    \item For $c=0\in\K$, and $k=0\in\Z_{\geq 0}$, we write $c^k=1\in\K$.
    \item $x^\alpha=\prod_{j=1}^nx_j^{\alpha_j}$.
    \item $\binom{\alpha}{\beta}=\prod_{j=1}^n\binom{\alpha_j}{\beta_j}$ is a product of ordinary binomial coefficients. In particular, if $\beta_j>\alpha_j$ for some $j$, then $\binom{\alpha}{\beta}=0$.
    \item If $\beta_i\leq\alpha_i$ for all $i$, then we write $\beta\leq\alpha$. If $\beta\leq\alpha$ and $\beta\neq\alpha$, then we write $\beta<\alpha$.
  \end{itemize}

  \definition[Newton quotients]\label{def:newton} Suppose $x\in\K^n$ and $\gamma\in\K^\times$,and suppose $F:x +\delta\cal O_n\to\K^m$ is a function. For $\alpha\in\Z_{\geq 0}^n$, we define the $\alpha$-Newton quotient $\Phi^\alpha F$ by
  \begin{equation*}
      \Phi^\alpha F:\big((x_1+\delta\cal O_1)^{\alpha_1+1}\times\cdots\times(x_n+\delta\cal O_1)^{\alpha_n+1}\big)\setminus\nabla\to\K^m,
  \end{equation*}
  \begin{equation*}
      \Phi^\alpha F(z^1,\ldots,z^n)=\sum_{\substack{(i_1,\ldots,i_n)\\1\leq i_\tau\leq\alpha_\tau+1}}\left[\prod_{\tau=1}^n\prod_{\substack{1\leq j_\tau\leq\alpha_{\tau}+1\\
      j_\tau\neq i_\tau}}\big(z_{i_\tau}^\tau-z_{j_\tau}^\tau\big)^{-1}\right]F\big(z_{i_1}^1,\ldots,z_{i_n}^n\big).
  \end{equation*}
  Here $\nabla$ is a suitable ``fat diagonal:'' $(z^1,\ldots,z^n)\in\K^{\alpha_1+1}\times\cdots\times\K^{\alpha_n+1}$ belongs to $\nabla$ precisely when
  \begin{equation*}
      z^\tau_i=z_j^\tau\quad\text{for some}\quad 1\leq i,j\leq\alpha_\tau+1,\quad i\neq j.
  \end{equation*}

  \definition[Non-Archimedean $C^\alpha$ maps]:\label{def:Calpha} Suppose $U\subseteq\K^n$ is nonempty and open, $\alpha\in\Z_{\geq 0}^n$, and suppose
  \begin{equation*}
      F=(f_1,\ldots,f_m):U\to\K^m
  \end{equation*}
  is a map. We say that $F$ is $C^\alpha$ if, for each $\beta\leq\alpha$ and each $x+\delta\cal O_n\subseteq U$, the Newton quotient $\Phi^\beta F$ extends to a continuous function on
  \begin{equation*}
      (x_1+\delta\cal O_1)^{\beta_1+1}\times\cdots\times(x_n+\delta\cal O_1)^{\beta_n+1}\to\K^m.
  \end{equation*}
  When $F$ is $C^\alpha$, we write
  \begin{equation*}
      D^\alpha F(x)=\Phi^\alpha F\big(\overbrace{(x_1,\ldots,x_{1})}^{\alpha_1+1}\,,\,\ldots\,,\,\overbrace{(x_n,\ldots,x_n)}^{\alpha_n+1}\big).
  \end{equation*}

  \definition[Non-Archimedean $C^k$ maps]\label{def:Ck} Suppose $U\subseteq\K^n$ is nonempty and open, $k\geq 1$, and suppose
  \begin{equation*}
      F=(f_1,\ldots,f_m):U\to\K^m
  \end{equation*}
  is a map. We say that $F$ is $C^k$ if, for each $\alpha\in\Z_{\geq 0}^n$ with $|\alpha|=k$, it is the case that $F$ is $C^\alpha$. In this case, we write
  \begin{equation*}
      dF(x)=[D^{\bf e_i}f_j(x)]_{j,i}.
  \end{equation*}
  If $F$ is $C^k$, then we write
  \begin{equation*}
      \|F\|_{C^k(U)}=\max_{u+\delta\cal O_n\subseteq U}\max_{0\leq|\alpha|\leq k}\sup\big\{\|\Phi^\alpha F(z^1,\ldots,z^n)\|:z^i\in (x_i+\delta\cal O_1)^{\alpha_i+1}\,\,\forall i\big\}.
  \end{equation*}

    \begin{lemma}[Basic properties of $\Phi^\alpha$]\label{lem:phi_basic}
        The following properties hold for $\Phi^\alpha$, $\alpha\in\Z_{\geq 0}^{n}$.
        \begin{enumerate}[label=(\alph*)]
            \item $\Phi^\alpha[af+bg]=a\Phi^\alpha [f]+b\Phi^\alpha[g]$.
            \item For each $1\leq i\leq n$ and each $1\leq j<k\leq\alpha_i+1$,
            \begin{equation*}
                \begin{split}
                &\Phi^\alpha f(z^1,\ldots,(z_1^i,\ldots,z_j^i,\ldots,z_k^i,\ldots,z_{\alpha_i+1}^i),\ldots,z^n)\\
                &\quad=\Phi^\alpha f(z^1,\ldots,(z_1^i,\ldots,z_k^i,\ldots,z_j^i,\ldots,z_{\alpha_i+1}^i),\ldots,z^n)
                \end{split}
            \end{equation*}
            \item
            \begin{equation*}
                \begin{split}
                    &\Phi^{\alpha+\bf e_i}f(z^1,\ldots,(z_j^i)_{1\leq j\leq\alpha_i+2},\ldots,z^n)\\
                    &=\frac{\Phi^\alpha f(z^1,\ldots,(z_1^i,z_3^i,\ldots,z_{\alpha_i+2}^i),\ldots,z^n)-\Phi^\alpha f(z^1,\ldots,(z^i_2,z_3^i,\ldots,z_{\alpha_i+2}^i),\ldots,z^n)}{z_1^i-z_2^i}
                \end{split}
            \end{equation*}
            \item If $f:\K^n\to\K$ takes the form of $f(x)=f_1(x_1)\cdots f_n(x_n)$, $f_j:\K\to\K$, then we have
            \begin{equation*}
                \Phi^\alpha f(z^1,\ldots,z^n)=\prod_{j=1}^n\Phi^{\alpha_j}f_j(z^j).
            \end{equation*}
            \item If $\beta_i<\alpha_i$ for some $i$, then $\Phi^\alpha[x^\beta]\equiv 0$.
            \item If $f,g:\K^n\to\K$, then
            \begin{equation*}
                \Phi^\alpha(fg)(z^1,\ldots,z^n)=\sum_{0\leq\beta\leq\alpha}\Phi^\beta f\Big(\big(z_j^i\big)_{\substack{1\leq i\leq n\\1\leq j\leq\beta_i+1}}\Big)\Phi^{\alpha-\beta}f\Big(\big(z_j^i\big)_{\substack{1\leq i\leq n\\\beta_{i}+1\leq j\leq\alpha_{i}+1}}\Big)
            \end{equation*}
            \item $\Phi^\alpha[x^\alpha]\equiv 1$.
        \end{enumerate}
    \end{lemma}

    \begin{proof}
    \begin{enumerate}[label=(\alph*):]
        \item Immediate.
        \item Follows from identities of the form $\frac{f(a)-f(b)}{a-b}=\frac{f(b)-f(a)}{b-a}$.
        \item Follows from Definition~\ref{def:newton} and the identity
        \begin{equation*}
            (z_{1}^i-z_{2}^i)^{-1}\big[(z_j^i-z_1^i)^{-1}-(z_j^i-z_2^i)^{-1}\big]=(z_j^i-z_1^i)^{-1}(z_j^i-z_2^i)^{-1}.
        \end{equation*}
        \item Quickly follows from Definition~\ref{def:newton}.
        \item By (d), we may assume $n=1$. The result then follows by (c) and an easy inductive argument.
        \item By (c), we may assume $n=1$. The result then follows by~\cite{schikhof1985}, Lemma 29.2(v).
        \item Inductive consequence of (e) and (f).
    \end{enumerate}
    \end{proof}

    \begin{lemma}
        Suppose $\alpha,\beta\in\Z_{\geq 0}^n$. Then $x^\beta$ is $C^k$ for every $k\in\N$, and
        \begin{equation*}
            D^\alpha\big[x^\beta\big]=\binom{\alpha}{\beta}x^{\beta-\alpha}.
        \end{equation*}
    \end{lemma}

    \begin{remark}
        In particular, note that $D^\alpha D^{\alpha'}\neq D^{\alpha+\alpha'}$ in general.
    \end{remark}

    \begin{proof}
        Use Lemma~\ref{lem:phi_basic}(d) and~\cite{schikhof1985}, Lemma 29.2(v).
    \end{proof}

    \begin{lemma}\label{lem:comm_deriv}
        Suppose $\alpha,\beta\in\Z_{\geq 0}^n$ have disjoint support: for each $1\leq i\leq n$, $\alpha_i\beta_i=0$. Then $\Phi^{\alpha}\Phi^\beta f=\Phi^{\alpha+\beta}f=\Phi^\beta\Phi^\alpha f$. In particular, if $f$ is $C^{|\alpha|+|\beta|}$, then $D^\alpha D^\beta f=D^{\alpha+\beta}f=D^\beta D^\alpha f$.
    \end{lemma}
    \begin{proof}
        Just apply Lemma~\ref{lem:phi_basic}(c) repeatedly.
    \end{proof}

    \begin{lemma}[Non-Archimedean Taylor approximation]
        Let $k\geq 1$. Suppose $F:\cal O_n\to\K^m$ is $C^k$. Then there are continuous functions $E_\alpha:\cal O_n^2\to\K^m$, vanishing on the diagonal, such that for all $x,y\in\cal O_n$,
        \begin{equation*}
            F(y)=F(x)+\sum_{0<|\alpha|\leq k}D^\alpha F(x)(y-x)^\alpha+\sum_{|\alpha|=k}(y-x)^\alpha E_\alpha(x,y).
        \end{equation*}
    \end{lemma}
    \begin{proof}
        We adopt the convention that $\Phi^0F=D^0F=F$. Suppose first that $n=1$. Then, since
        \begin{equation*}
            (y-x)\Phi^1 F(y,x)=-F(x)+F(y)
        \end{equation*}
        and, for $\alpha\geq 2$,
        \begin{equation*}
            (y-x)\Phi^{\alpha}F\big((y,\overbrace{x,\ldots,x}^\alpha)\big)=-D^{\alpha-1}F(x)+\Phi^{\alpha-1}F\big((y,\overbrace{x,\ldots,x}^{\alpha-1})\big),
        \end{equation*}
        we obtain by induction
        \begin{equation*}
            (y-x)^\alpha\Phi^{\alpha}F\big((y,\overbrace{x,\ldots,x}^\alpha)\big)=-\sum_{\beta=0}^{\alpha-1}D^\beta F(x)(y-x)^\beta+F(y).
        \end{equation*}
        Rearranging, we obtain the desired
        \begin{equation*}
            F(y)=\sum_{\beta=0}^\alpha D^\beta F(x)(y-x)^\beta+(y-x)^\alpha\Big[\Phi^\alpha F\big(y,\overbrace{x,\ldots,x}^\alpha\big)-D^\alpha F(x)\Big].
        \end{equation*}

        Now, we use this to obtain the general $n$ case. By applying the one-dimensional case along the $n$th variable, we first obtain
        \begin{equation*}
        \begin{split}
            F(y)&=\sum_{\beta_n=0}^kD^{\beta_n\bf e_n}F(y_{<n},x_n)(y_n-x_n)^{\beta_n}\\
            &+(y_n-x_n)^{k}\Big[\Phi^{k\bf e_n}F\big(y_{<n},(y_n,\overbrace{x_n,\ldots,x_n}^k)\big)-D^{k\bf e_n}F(y_{<n},x_n)\Big].
        \end{split}
        \end{equation*}
        For each $0\leq\beta_n\leq k$, we may apply the $n=1$ case again to obtain
        \begin{equation*}
            \begin{split}
            &\hspace{10em}D^{\beta_n\bf e_n}F(y_{<n},x_n)(y_n-x_n)^{\beta_n}\\
            &=\sum_{\beta_{n-1}=0}^{k-\beta_n}D^{\beta_{n-1}\bf e_{n-1}+\beta_n\bf e_n}F(y_{<n-1},x_{n-1},x_n)(y_{n-1}-x_{n-1})^{\beta_{n-1}}(y_n-x_n)^{\beta_n}\\
            &\hspace{3em}+(y_{n-1}-x_{n-1})^{k-\beta_n}(y_n-x_n)^{\beta_n}\\
                &\quad\times\Big[\Phi^{(k-\beta_{n})\bf e_{n-1}}D^{\beta_n\bf e_n} F\big(y_{<n-1},(y_{n-1},\overbrace{x_{n-1},\ldots,x_{n-1}}^{k-\beta_n}),x_n\big)-D^{(k-\beta_n)\bf e_{n-1}+\beta_n\bf e_n}F(y_{<n-1},x_{n-1},x_n)\Big].
            \end{split}
        \end{equation*}
        Continuing in this manner, we obtain an expansion of the form
        \begin{equation*}
            \begin{split}
            F(y)&=\sum_{|\alpha|\leq k} D^\alpha F(x)(y-x)^\alpha\\
            &+\sum_{|\alpha|=k}(y-x)^\alpha\Big[\Phi^{\alpha}F\big(y_{<t(\alpha)},(y_{t(\alpha)},\overbrace{x_{t(\alpha)},\ldots,x_{t(\alpha)}}^{\alpha_{t(\alpha)}}),x_{>t(\alpha)}\big)-D^\alpha F\big(y_{<t(\alpha)},x_{\geq t(\alpha)}\big)\Big];
            \end{split}
        \end{equation*}
        here, $t(\alpha)=\max\{1\leq j\leq n:\sum_{i=j}^n\alpha_i=k\}$. Finally, by the definition of a $C^k$ function, the bracketed expression is continuous in $x,y$ and vanishes along $x=y$.

    \end{proof}

    \begin{lemma}
        Suppose $\K^n\supseteq U\overset F\longrightarrow V\subseteq\K^m$ and $G:V\to\K^p$ are $C^k$ maps. Then $G\circ F$ is $C^k$.
    \end{lemma}
    \begin{proof}
        For simplicity, we restrict attention to $n=p=1$. We may assume that $U=\cal O_1$, $V=\cal O_m$. With a bit of algebra, one obtains
        \begin{equation*}
            \Phi^1(G\circ F)(x,y)=\sum_{j=1}^m\Phi^1 F_j(x,y)\Phi^{\bf e_j}G(F_{<j}(x),(F_j(x),F_j(y)),F_{>j}(y)).
        \end{equation*}
        We claim inductively that
        \begin{equation*}
            \Phi^\ell (G\circ F)(x_1,\ldots,x_{\ell+1})=\sum_{\mathrm{finite}}\mathrm{Poly}(\Phi^{j}F_\iota(x_{I_j})_{j\leq\ell})\Phi^{\alpha} G(\pi(F_i(x_k)));
        \end{equation*}
        here, we have a finite sum over terms taking the form $\mathrm{Poly}(\Phi^jF_\iota(x_{I_j}))\cdot\Phi^\alpha G(\pi(F_i(x_k)))$, where the first factor is a polynomial in expressions of the form $\Phi^j F_\iota$, $1\leq\iota\leq m$ and $j\leq\ell$, evaluated along some appropriately-sized sub-tuple $x_{I_j}\subseteq(x_1,\ldots,x_{\ell+1})$, and the second factor is the evaluation of $\Phi^\alpha G$, where $\alpha\in\Z_{\geq 0}^m$ has $|\alpha|\leq \ell$, along an appropriate selection $\pi$ of various $F_i(x_k)$, $1\leq i\leq m$ and $1\leq k\leq\ell+1$. Indeed, applying $\Phi^1$ to both sides, and using the base case, we conclude. Now, since each $\Phi^j F_\iota$ extends to be continuous on all of $\cal O^k_1$, the first factor extends to be continuous over $(x_1,\ldots,x_{\ell+1})\in\cal O_1^{\ell+1}$. Similarly, since $\Phi^\alpha G$ extends continuously to all of $\cal O_1^{\alpha_1+1}\times\cdots\times\cal O_1^{\alpha_m+1}$, and the $F_i$ are continuous, we conclude that the second factor has a continuous extension. Multiplying and adding, the result follows.
    \end{proof}

    \begin{theorem}[Weak non-Archimedean Sard's theorem]\label{thm:sard}
      Let $F:\cal O_n\to\K^m$ be a $C^k$ map, for some $k\in\N$. Suppose $r\in\Z_{\geq 0}$ is such that $r\leq m-\frac{n-m}{k-1}$, and $x+\delta\cal O_n\subseteq\cal O_n$ is such that $dF$ has rank $\leq r$ over $y\in x+\delta\cal O_n$. Then $\left.F\right|_{x+\delta\cal O_n}$ has null image.
  \end{theorem}
  \begin{proof}
     We may freely assume that $x=0$ and $F(0)=0$. Shrinking the domain further and possibly redefining $r$, we may assume that $dF(y)$ has rank exactly $r$ in $\delta\cal O_n$. Lastly, we may take
     \begin{equation*}
         dF(0)=\begin{bmatrix}
             I_r & 0_{r\times (n-r)}\\0_{(m-r)\times r} & 0_{(m-r)\times(n-r)}
         \end{bmatrix}.
     \end{equation*}
     For $y\in\delta\cal O_n$ with $\delta$ sufficiently small, we will have that the top-left $(r\times r)$ minor of $dF(y)$ is nonsingular. Since $dF(y)$ is of rank $r$, it follows that the last $m-r$ columns of $dF(y)$ are identically equal to $0$ over $y\in\delta\cal O_n$.

     Now, we compare two Taylor expansions. Select first $y=(\delta,\ldots,\delta)\in\delta\cal O_n$, and assume that $\gamma=\delta\sigma\in\{0\}^r\times\delta\cal O_{n-r}$ is such that $\gamma_j=0$ whenever $j\leq r$. For one, we have
     \begin{equation*}
         F(y+\gamma)=dF(0).(y+\gamma)+\sum_{1<|\alpha|\leq k}(y+\gamma)^\alpha D^\alpha F(0)+\sum_{|\alpha|=k}(y+\gamma)^\alpha E_\alpha(0,y+\gamma),
     \end{equation*}
     and for another we have
     \begin{equation*}
         F(y+\gamma)=F(y)+dF(y).\gamma+\sum_{1<|\alpha|\leq k}\gamma^\alpha D^\alpha F(y)+\sum_{|\alpha|=k}\gamma^\alpha E_\alpha(y,y+\gamma).
     \end{equation*}
     Note, however, that since $\gamma$ is supported in the indices $j>r$, and $dF(y),dF(0)$ are zero on the corresponding block, we get $dF(y).\gamma=0=dF(0).\gamma$. Thus, after equating both formulae for $F(y+\gamma)$, we obtain
     \begin{equation}\label{eq:taylor_compare}
        \begin{split}
         F(y)-dF(0).y&=\sum_{1<|\alpha|\leq k}\Big[(y+\gamma)^\alpha D^\alpha F(0)-\gamma^\alpha D^\alpha F(y)\Big]\\
         &+\sum_{|\alpha|=k}\Big[(y+\gamma)^\alpha E_\alpha(0,y+\gamma)-\gamma^\alpha E_\alpha(y,y+\gamma)\Big].
         \end{split}
     \end{equation}
     For $1<\ell\leq k$ and $\gamma\in\{0\}^r\times\delta\cal O_{n-r}$, we write
     \begin{equation*}
         Q_\ell(\gamma)=\sum_{|\alpha|=\ell}\Big[(y+\gamma)^\alpha D^\alpha F(0)-\gamma^\alpha D^\alpha F(y)\Big].
     \end{equation*}
     We would like to show that $D^\alpha F(0)=0$ whenever $\alpha_i>0$ for some $i>r$. For the sake of contradiction, suppose that $\ell$ is minimal such that $y^{\alpha'}D^\alpha F(0)-1_{\alpha'=0}D^\alpha F(y)\neq 0$, for some $\alpha=\alpha'+\alpha''$, $\alpha'\in\Z_{\geq 0}^r\times\{0\}^{n-r}$ and $\alpha''\in\{0\}^r\times\Z_{\geq 0}^{n-r}$, where $\alpha''\neq 0$ and $|\alpha|=\ell$. We have
     \begin{equation*}
         \Phi^{\alpha''}_\gamma[Q_\ell]\equiv y^{\alpha'} D^{\alpha}F(0)-1_{\alpha'=0}D^\alpha F(y).
     \end{equation*}
     Write $\frak e_i\in\N$ to be the least integer greater than $\frac{\log (\alpha_i+1)}{\log(|\varpi|^{-1})}$. In particular, we see that $\alpha_i+1\leq\#(\cal O_1/\varpi^{\frak e_i}\cal O_1)$ for all $i$. Correspondingly, we select for each $r<i\leq n$ elements $\gamma^i_1,\ldots,\gamma^i_{\alpha_i+1}$ to represent distinct cosets in $\delta\cal O_1/\delta\varpi^{\frak e_i}\cal O_1$. In particular,
     \begin{equation*}
         |\gamma^i_j-\gamma^i_{j'}|\geq|\delta||\varpi|^{\frak e_i},\quad \forall 1\leq j\neq j'\leq\alpha_i+1.
     \end{equation*}
     Utilizing Lemma~\ref{lem:phi_basic}(c), we see that $\Phi^{\alpha''}[Q_{\ell}]$ is a sum of difference quotients of the form
     \begin{equation*}
         \frac{Q_{\ell}(\gamma)-Q_{\ell}(\gamma')}{\prod(\gamma_j^i-\gamma_{j'}^i)},
     \end{equation*}
     where $\gamma,\gamma'\in\{0\}^r\times\delta\cal O_{n-r}$ differ by one entry. Thus, for some such $\gamma,\gamma'$:
     \begin{equation*}
         \|Q_{\ell}(\gamma)-Q_{\ell}(\gamma')\|\geq|\delta|^\ell\|D^\alpha F(0)-1_{\alpha'=0}D^\alpha F(y)\||\varpi|^{\frak e_{r+1}\alpha_{r+1}+\ldots+\frak e_n\alpha_n}.
     \end{equation*}
     On the other hand, for $\ell<\rho\leq k$,
     \begin{equation*}
         \|Q_\rho(\gamma)\|,\|Q_\rho(\gamma')\|\leq|\delta|^\rho\|F\|_{C^k},
     \end{equation*}
     so that $\|Q_\rho(\gamma)-Q_\rho(\gamma')\|<\|Q_\ell(\gamma)-Q_\ell(\gamma')\|$ when $\delta$ is sufficiently small. Similarly, if $|\beta|=k$,
     \begin{equation*}
         \|(y+\gamma)^\beta E_\beta(0,y+\gamma)-\gamma^\beta E(y,y+\beta)\|\leq o_{\delta\to 0}(1)\times|\delta|^k,
     \end{equation*}
     and similarly for $\gamma'$. It follows that the right-hand side of~\eqref{eq:taylor_compare} is non-constant in $\gamma$, contradicting the constancy on the left-hand side. Thus, we must have $y^\alpha D^\alpha F(0)=1_{\alpha'=0}D^\alpha F(y)$, for all $1<|\alpha|\leq k$ such that $\alpha''\neq 0$. It follows that $D^\alpha F(0)=0$ whenever $1<|\alpha|\leq k$ and $\alpha',\alpha''$ are both nonzero.

     To handle the remaining case of $\alpha'=0,\alpha''\neq 0$, we consider the Taylor expansions in~\eqref{eq:taylor_compare} where we now choose $y=\big(\overbrace{0,\ldots,0}^r\,,\,\overbrace{\delta,\ldots,\delta}^{n-r}\big)$. The only nonzero terms arise from $\beta'=0$. Again, for $1<\ell\leq k$ we consider
     \begin{equation*}
         P_\ell(\gamma)=\sum_{\substack{\beta''\in\{0\}^r\times\Z_{\geq 0}^{n-r}\\|\beta''|=\ell}}\Big[(y+\gamma)^{\beta''}D^{\beta''}F(0)-\gamma^{\beta''}D^{\beta''}F(y)\Big].
     \end{equation*}
     By the same argument as before, if $y^{\alpha''}D^{\alpha''}F(0)\neq 0$ then the minimal-order such term has more variation than the higher-order terms, contradicting the constancy of the left-hand side of~\eqref{eq:taylor_compare}. Thus, $D^{\alpha''}F(0)=0$ for all $\alpha''\in\{0\}^r\times\Z_{\geq 0}^{n-r}$ with $|\alpha''|\leq k$. Putting the two results together, we obtain that $D^\alpha F(0)=0$ unless $\alpha''=0$. Finally, $dF(y)$ is zero in the last $n-r$ coordinates as well, for all $y\in\delta\cal O_n$, and the role of $0$ in the above was arbitrary, we may conclude that $D^\alpha F(y)=0$ whenever $\alpha''\neq 0$ and $y\in\delta\cal O_n$.
      
      Suppose then that $\eps\leq|\delta|$ and $y,y'\in \delta\cal O_n$ are such that $\|y-y'\|\leq\eps$. Using the Taylor expansion around $y$, we obtain that
      \begin{equation*}
          |F_j(y)-F_j(y')|=o_{\eps\to 0}(\eps^k),\quad\forall j>r.
      \end{equation*}
      On the other hand, if $j\leq r$,
      \begin{equation*}
          |F_j(y)-F_j(y)|\lesssim\eps.
      \end{equation*}
      Thus, for fixed $y$, the image of the $\eps$-neighborhood of $y$ under $F$ has volume $o_{\eps\to 0}(\eps^{r+k(m-r)})$. On the other hand, the volume of $x+\delta\cal O_n$ is $|\delta|^n=\eps^n(|\delta|\eps^{-1})^n$. Comparing both sides, we conclude that the image of $\left.F\right|_{x+\delta\cal O_n}$ has volume $o_{\eps\to 0}(1)$.
  \end{proof}

  \begin{remark}
      It is tempting to prove a few classical theorems of differential topology in our setting: the constant rank theorem, the inverse function theorem, a more literal translation of Sard's theorem, etc. However, several difficulties arise. The most prominent of these is the fact that an ``infinitesimally-constant'' $C^k$ function need not be constant; that is, $dF\equiv 0$ does not imply any sort of local constancy of $F$, even under strong regularity assumptions. Another problem is the fact that $D^2\neq D^1\circ D^1$, so arguments by induction on the number of derivatives often become problematic. For the latter problem, assumptions on the characteristic of the field $\K$ may be helpful. Rather than exploring these, we opt for the theoretical minimum needed in this manuscript and invite the interested reader to develop this topic further.
  \end{remark}

\printbibliography

\end{document}